\documentclass[12pt]{amsart}
\usepackage{amsmath,amssymb, cite,amsthm, color, graphicx,tikz}
\usepackage{verbatim}
\usepackage{tikz,enumitem,graphicx, subfig, microtype, color}
\usepackage[hyperpageref]{backref}
\usepackage[english]{babel}
\usepackage{xcolor}
\usepackage[left=2.3cm,right=2.3cm,top=2.8cm,bottom=2.7cm]{geometry}
\usepackage[T1]{fontenc}
\usepackage{lmodern}
\usepackage[colorlinks=true,urlcolor=blue, citecolor=red,linkcolor=blue,
linktocpage,pdfpagelabels, bookmarksnumbered,bookmarksopen]{hyperref}

\allowdisplaybreaks

\renewcommand{\epsilon}{\varepsilon}

\numberwithin{equation}{section}
\newtheorem{theorem}{Theorem}[section]
\newtheorem{proposition}[theorem]{Proposition}
\newtheorem{lemma}[theorem]{Lemma}

\newtheorem{corollary}[theorem]{Corollary}

\theoremstyle{definition}
\newcommand{\B}{\mathcal B}
\newcommand{\eps}{\varepsilon}

\newcommand{\R}{\mathbb{R}}

\newcommand{\cal}{\mathcal}
\newcommand{\weakto}{\rightharpoonup}

\def\Xint#1{\mathchoice
{\XXint\displaystyle\textstyle{#1}}%
{\XXint\textstyle\scriptstyle{#1}}%
{\XXint\scriptstyle\scriptscriptstyle{#1}}%
{\XXint\scriptscriptstyle\scriptscriptstyle{#1}}%
\!\int}
\def\XXint#1#2#3{{\setbox0=\hbox{$#1{#2#3}{\int}$ }
\vcenter{\hbox{$#2#3$ }}\kern-.6\wd0}}

\def\intmed{\Xint-}
\begin{document}

\title[Harnack Estimates]{Harnack and pointwise estimates for  degenerate or singular parabolic equations}

\author[F.\ G.\ D\"uzg\"un, S.\ Mosconi \& V. Vespri]{Fatma Gamze D\"uzg\"un, Sunra Mosconi$^{*}$ and Vincenzo Vespri}

\address[F.\ G.\ D\"uzg\"un]{Department of Mathematics, 
\newline\indent
Hacettepe University, 06800, Beytepe, Ankara, Turkey}
\email{gamzeduz@hacettepe.edu.tr }

\address[S.\ Mosconi]{Dipartimento di Matematica e Informatica,
\newline\indent
Universit\`a degli Studi di Catania,
Viale A.\ Doria 6, 95125 Catania, Italy}
\email[Corresponding author]{mosconi@dmi.unict.it}

\address[V. Vespri]{Dipartimento di Matematica e Informatica ``U. Dini'', 
\newline\indent
Università di Firenze, Viale Morgagni 67/A, 50134 Firenze, Italy}
\email{vespri@math.unifi.it}

\bigskip


\begin{abstract}
In this paper we give both an historical and technical overview of the theory of Harnack inequalities for nonlinear parabolic equations in divergence form. We start reviewing the elliptic case with some of its variants and geometrical consequences. The linear parabolic Harnack inequality of Moser is discussed extensively, together with its link to two-sided kernel estimates and to the Li-Yau differential Harnack inequality. Then we overview the more recent developements of the theory for nonlinear degenerate/singular equations, highlighting the differences with the quadratic case and introducing the so-called {\em intrinsic} Harnack inequalities. Finally, we provide complete proofs of the Harnack inequalities in some paramount case to introduce the reader to the {\em expansion of positivity} method. 
\end{abstract}

\maketitle

\tableofcontents

\section{Introduction}
Generally speaking, given class of functions ${\cal C}$ defined on a set $\Omega$, a Harnack inequality is a pointwise control of the form $u(x)\le C\, u(y)$ for all  $u\in {\cal C}$ (with a constant independent of $u$)  where the inequality holds for $x\in X\subseteq \Omega$ and $y\in Y\subseteq\Omega$, $(X, Y)$ belonging to a certain family  ${\cal F}$  determined by ${\cal C}$. Thus it takes the form 
\begin{equation}
\label{ghar}
\exists\  C=C({\cal C}, {\cal F}) \quad \text{such that} \quad \sup_{X}u\le C\,  \inf_{Y} u\qquad \forall \ (X, Y)\in {\cal F},\  u\in {\cal C}.
\end{equation}
Given ${\cal C}$, one is ideally interested in maximal families ${\cal F}$. In this respect, certain properties of maximal familes are immediate, e.g., if $X'\subseteq X$ and $(X, Y)\in {\cal F}$, then $(X', Y)\in {\cal F}$. The so-called {\em Harnack chain} argument consists in the elementary observation that if both $(X, Y)$ and $(Y',  Z)$ belong to ${\cal F}$ and $y_{0}\in Y\cap Y'\ne \emptyset$, then
\[
\sup_{X} u\le C\, \inf_{Y} u\le C\, u(y_{0})\le C\, \sup_{Y'}u\le C^{2}\, \inf_{Z} u,
\]
hence we can add all such couples $(X, Z)$ to  ${\cal F}$ by considering the constant $C^{2}$. Other properties of ${\cal F}$ follows from the structure of ${\cal C}$: if, for instance, ${\cal C}$ is invariant by a suitable semi-group $\{\Phi_{\lambda}\}_{\lambda>0}$ of domain transformations (meaning that $u\in {\cal C}\,  \Rightarrow \, u \circ\Phi_{\lambda}\in {\cal C}$ for all $\lambda>0$), then ${\cal F}$ should also exhibits this invariance. 

Formally, to a larger class ${\cal C}$ corresponds a smaller family ${\cal F}$ and the more powerful Harnack inequalities aim at ``maximize'' the two sets at once.  Typically, ${\cal C}$ is the set of nonnegative solutions to certain classes of PDE in an ambient metric space $\Omega$ and ${\cal F}$ should at least cluster near each point of $\Omega$ (i.\,e.\,$\forall P\in \Omega, r>0$ there exists $(X, Y)\in {\cal F}$ such that both $X$ and $Y$ lie in the ball of center $P$ and radius $r$). Another example is the class of ratios of nonnegative harmonic functions vanishing on the same set, giving rise to the so-called {\em boundary Harnack inequalities}. Given ${\cal C}$, searching for a suitable maximal family ${\cal F}$ such that \eqref{ghar} holds, informally takes the name of {\em finding the right form} of the Harnack inequality in ${\cal C}$. Rich examples of such instance arise in the theory of hypoelliptic PDE's.

Historically, the first of such pointwise control was proved by Harnack in 1887 for the class ${\cal C}$ of nonnegative harmonic functions in a domain $\Omega\subseteq \R^{2}$, with ${\cal F}$ being made of couples of identical balls well contained in $\Omega$. Since then, extensions and variants of the Harnack inequality grew steadily in the mathematical literature, with plentiful applications in PDE and differential geometry. Correspondingly, its proof in the various settings has been obtained through many different points of view. To mention a few: the original potential theoretic approach, the measure-theoretical approach of Moser, the probabilistic one of Krylov-Safonov and the differential approach of Li and Yau. 

Many very good books and surveys on the Harnack inequality already exist (see e.g. \cite{Kass}) and we are thus forced to justify the novelty of this one.
Our main focus will be the quest for the right form ${\cal F}$ of various Harnack inequalities and, to this end, we will mainly deal with parabolic ones, which naturally exhibit a richer structure. Even restricting the theme to the parabolic setting requires a further choice, as the theory naturally splits into two large branches:  one can either consider {\em divergence form} (or variational) equations, whose basic linear example is 
$u_{t}={\rm div}(A(x)\, Du)$, or equations in  {\em non-divergence form} (or non-variational), such as $u_{t}=A(x)\cdot D^{2}u$. While some attempts to build a unified approach to the Harnack inequality has been made (see \cite{FeSa}), structural differences seem unavoidable. Moreover, both examples have nonlinear counterparts  and the corresponding theories rapidly diverge. We will deal with parabolic nonlinear equations in  divergence form, referring to the surveys \cite{IS, Kryl} for the non-divergence theory.

Rather than simply collecting known result to describe the state of the art, we aim at giving both an historical and technical overview on the subject, with emphasis on the different proofs and approaches to the subject. 

The first part, consisting in sections 2 to 4, will focus on the various form of \eqref{ghar}, mentioning some applications and giving from time to time proofs of well-known facts which we found somehow hard to track in the literature. In particular, we will deal with the elliptic case in section 2, with the linear parabolic Harnack inequality in section 3 and with the singular and degenerate parabolic setting in section 4. Here we will describe the so-called {\em intrinsic Harnack inequalities}, by which we mean a generalization of \eqref{ghar} where the sets $X$ and $Y$ also depend on $u$ (or, equivalently, \eqref{ghar} holds in a restricted class ${\cal C}$ determined by non-homogeneous scalings). 

The second part consists of the final and longest section, which is devoted to detailed proofs of the most relevant Harnack inequalities for equations in divergence form. Our aim is to obtain the elliptic and parabolic Harnack inequalities in a unified way, following the measure-theoretical approach of De Giorgi to regularity and departing from Moser's one. This roadmap has been explored before (see \cite{Mald} for an axiomatic treatment), but we push it further to gather what we believe are the most simple proofs of the Harnack inequalities up to date. Credits to the main ideas and techniques should be given to the original De Giorgi paper \cite{G}, the book of Landis \cite{Landis} and the work of Di Benedetto and collaborators gathered in the monograph \cite{HR}. We will focus on model problems rather than on generality in the hope to make the proofs more transparent and attract non-experts to this fascinating  research field.

\section{Elliptic Harnack inequality}
    \subsection{Original Harnack}

In 1887, The German mathematician C.G. Axel von
Harnack proved the following result in \cite{DT}.
\begin{theorem}
Let $u$ be a nonnegative harmonic function in $B_R(x_0)\subseteq \R^2$. Then for all  $x \in B_r(x_0) \subset  B_R(x_0)$ it holds
\[
\frac{R-r}{R+r}u(x_0)\leq u(x) \leq \frac{R+r}{R-r}u(x_0).
\]
\end{theorem}
    
The estimate can be generalized to any dimension $N\geq 1$, resulting in
\begin{equation}
\label{OH}
\left(\frac{R}{R+r}\right)^{N-2} \frac{R-r}{R+r}u(x_0)\leq u(x) \leq\left(\frac{R}{R-r}\right)^{N-2} \frac{R+r}{R-r}u(x_0),
\end{equation}
however, the modern version of the Harnack inequality for harmonic functions is the following special case of the previous one.
\begin{theorem}
Let $ N \geq 1$. Then there exists
a constant $C=C(N)>1$, such that if $u$ is a nonnegative, harmonic function in $B_{2r}(x_0)$, then
\begin{equation}
\label{H}
\sup_{B_r(x_0)} u \leq C  \inf_{B_r(x_0)}u.
\end{equation}
\end{theorem}

The proof of this latter form of the Harnack inequality is an easy consequence of the mean value theorem, while the more precise form \eqref{OH} can be derived through Poisson representation formula. For the early historical developments related to the first Harnack inequality we refer to the survey \cite{Kass}.

The Harnack inequality has several deep and powerful consequences.
On the local side, Harnack himself in \cite{DT} derived from it a precisely quantified oscillation estimate. Due to the ubiquity of this argument we recall its elementary proof.
Let $x_0=0$ and  
\[
M_r(u)=\sup_{B_r}u, \qquad m_r(u)=\inf_{B_r} u, \quad {\rm osc}(u, B_{r})=M_r(u)-m_r(u).
\]
Both $M_{2r}(u)-u$ and $u-m_{2r}(u)$ are nonnegative and harmonic in $B_{2r}$, so \eqref{H} holds for them, resulting in
\[
M_{2r}(u)-m_r(u)\leq C(M_{2r}(u)-M_r(u)),\quad M_r(u)-m_{2r}(u)\leq C(m_r(u)-m_{2r}(u)),
\]
which added together give
\[
M_{2r}(u)-m_{2r}(u)+M_r(u)-m_r(u)\leq C\big(M_{2r}(u)-m_{2r}(u)-(M_r(u)-m_r(u))\big).
\]
Rearranging, we obtain
\[
{\rm osc}(u, B_{r})\leq \frac{C-1}{C+1} {\rm osc}(u, B_{2r}),
\]
which is the claimed quantitive estimate of decrease in oscillation. 

Removable singularity results can also be obtained through the Harnack inequality, as well as two classical convergence criterions for sequences of harmonic functions. At the global level, it implies Liouville and Picard type theorems. For example, Liouville's theorem asserts that any globally defined harmonic function bounded from below must be constant, as can be clearly seen by applying \eqref{H} to $u-\inf_{\R^N}u$ and letting $r\to +\infty$.

\subsection{Modern developements}
In his celebrated paper \cite{G}, De Giorgi introduced the measure theoretical approach to regularity, proving the local H\"older continuity of weak solutions of linear elliptic equations in divergence form
\begin{equation}
\label{linell}
L(u):=\sum_{i,j=1}^N D_i(a_{ij}(x) D_ju)=0
\end{equation}
with merely measurable coefficients satisfying the ellipticity condition
\begin{equation}
\label{EC}
\lambda|\xi|^2\leq \sum_{i, j=1}^Na_{i j}(x)\xi_i\xi_j\leq \Lambda |\xi|^2,\qquad 0\leq \lambda\leq \Lambda<+\infty.
\end{equation}
The modern regularity theory descending from his ideas is a vast field and the relevant literature is huge. We refer to \cite{Min} for a general overview and bibliographic references; the monograph \cite{GIU} contains the regularity theory of quasi-minima, while for systems one should see \cite{KM} and the literature therein.
\vskip2pt
 
Regarding the Harnack inequality, Moser extended  in his fundamental work \cite{HAR} its validity to solutions of \eqref{linell}.
\begin{theorem}
Suppose $u\geq 0$ solves \eqref{linell} in a ball $B_{2r}(x_0)$  where $a_{i j}$ obeying \eqref{EC}.
Then there exists a constant $C>1$ depending only on $N$ and the {\em ellipticity ratio} $\Lambda/\lambda$ such that 
\[
\sup_{B_r(x_0)} u \leq C  \inf_{B_r(x_0)}u.
\]
\end{theorem}

Moser's proof is also measure-theoretical, stemming from the De Giorgi approach but introducing pioneering new ideas. It relied on the John-Nirenberg Lemma \cite{JN} and certainly contributed to its diffusion in the mathematical community. Such a level of generality allowed  to apply essentially the same technique for the general quasilinear equation
\begin{equation}
\label{qle}
{\rm  div} A(x ,u, Du) = 0. 
\end{equation}
Indeed, in \cite{SER} \cite{TRU}, the same statement of the Harnack inequality has been proved for \eqref{qle} instead of the linear equation \eqref{linell}, provided $A$ satisfies for some $p>1$ and $\Lambda\geq \lambda>0$ 
\begin{equation}
\label{pgrowth}
\begin{cases}
A(x, s, z)\cdot z \geq  \lambda |z|^p\\
|A(x, s, z)| \leq  \Lambda |z|^{p-1} 
\end{cases}\qquad x\in B_{2r}(x_0), s\in \R, z\in \R^N.
\end{equation}
The power of the measure-theoretical approach was then fully exploited in \cite{DBTr}, where the Harnack inequality has been deduced without any reference to an elliptic equation, proving that it is a consequence of very general energy estimates of Caccioppoli type, encoded in what are the nowadays called {\em De Giorgi classes}. For a comprehensive treatment of the latters see \cite{DBG16}.

\subsection{Moser's proof and weak Harnack inequalities}
Moser's proof of the Harnack inequality is splitted in two steps:
\vskip5pt
\noindent
(I) -- {\em $L^{p}-L^{\infty}$ bound}:\\
Let $u$ be a nonnegative subsolution of \eqref{linell}  in $B_{2r}$, i.e., obeys  $-L(u)\leq 0$ (supersolutions being defined through the opposite inequality). For any $p>0$ it holds
\begin{equation}
\label{linftylr}
\sup_{B_{r}}u\leq C\left(\intmed_{B_{2r}}|u|^{p}\, dx\right)^{\frac{1}{p}}
\end{equation}
for some constant $C=C(N, \Lambda/\lambda, p)$.
If on the other hand $u$ is a positive supersolution, then $u^{-1}$ is a positive subsolution, and \eqref{linftylr} can be rewritten as
\[
\inf_{B_{r}}u\ge C^{-1}\left(\intmed_{B_{2r}}u^{-p}\, dx\right)^{-\frac 1 p}.
\]
\vskip5pt
\noindent
(II) - {\em Crossover Lemma}. The Harnack inequality then follows if one has
\begin{equation}
\label{cross}
\intmed_{B_{r}} u^{\bar p}\, dx\, \intmed_{B_{r}} u^{-\bar p}\, dx\le C(N)
\end{equation}
for some (small) $\bar p=\bar p(N, \Lambda, \lambda)>0$. This is the most delicate part of Moser's approach, and is dealt with the so-called {\em logarithmic estimates}. The idea is to prove a universal bound on $\log u$, as suggested by the Harnack inequality itself. To this end, consider a ball $B_{2\rho}(x_{0})\subseteq B_{2r}$ and test the equation with $u^{-1}\eta^{2}$, $\eta$ being a cutoff function in $C^{\infty}_{c}(B_{2\rho}(x_{0}))$. This yields 
\[
\lambda \int_{B_{2\rho}(x_{0})}|Du|^{2}u^{-2}\eta^{2}\, dx \le 2\Lambda\int_{B_{2\rho}(x_{0})} |Du|\,u^{-1}\, |\eta| \, |D\eta|\, dx
\]
with $\lambda, \Lambda$ given in \eqref{EC}. Apply Young inequality on the right and note that $|D\eta|\le c\, \rho^{-1}$ to get
\begin{equation}
\label{logest}
\intmed_{B_{\rho}(x_{0})}|D \log u|^{2}\, dx\le C(\Lambda/\lambda)\, \rho^{-2}
\end{equation}
as long as $\eta\equiv 1$ in $B_{\rho}(x_{0})$. Poincar\'e inequality then implies
\[
\intmed_{B_{\rho}(x_{0})}\Big(\log u-\intmed_{B_{\rho}(x_{0})}\log u\, dx\Big)^{2}\, dx\le C(N, \Lambda/\lambda),\qquad \text{for all $B_{2\rho}(x_{0})\subseteq B_{2r}$},
\]
wich means that $\log u \in BMO(B_{2r})$. Then John-Nirenberg's Lemma ensures that 
\[
\intmed_{B_{r}}e^{\bar p\, |w|}\, dx\le c,\qquad w=\log u-m,\qquad m=\intmed_{B_{r}}\log u\, dx
\]
for some small $\bar p=\bar p (N, \Lambda)>0$ and $c=c(N)$, and inequality \ref{cross} follows by multiplying
\[
 \intmed_{B_{r}} u^{\bar p}\, dx=e^{\bar p\, m} \intmed_{B_{r}}e^{\bar p\, w}\, dx\le c\, e^{\bar p\, m}\qquad\text{and} \qquad 
 \intmed_{B_{r}} u^{-\bar p}\, dx=e^{-\bar p\, m} \intmed_{B_{r}}e^{-\bar p\, w}\, dx\le c\, e^{-\bar p\, m}.
 \]

\vskip2pt
In particular, Moser's proof shows that  a weaker form of Harnack inequality holds  for the larger class of non-negative  supersolutions to \eqref{linell} in $B_{2r}$. Namely, for any $p\in \ ]0, \frac{N}{N-2}[$, the following {\em weak Harnack inequality} holds 
\[
\left(\intmed_{B_{2r}}u^{p}\, dx\right)^{\frac{1}{p}}\leq C\inf_{B_{r}}u
\]
for some constant $C=C(N, \Lambda/\lambda, p)$.
The range of exponents in the weak Harnack inequality is optimal, as the fundamental solution for the Laplacian shows.
Notice that the $L^{\infty}-L^{p}$ bound also implies a Liouville theorem for $L^{p}(\R^{N})$ nonnegative subsolutions, while the weak Harnack inequality gives a lower asymptotic estimate for positive $L^{p}(\R^{N})$ supersolutions. From the local point of view, the latter is also sufficient for H\"older regularity and for strong comparison principles. 

A different and detailed proof of the elliptic Harnack inequality via the expansion of positivity technique will be given in section 5.1.

\subsection{Harnack inequality on minimal surfaces}
After considering the Harnack inequality for nonlinear operator, a very fruitful framework  was to consider its validity for linear elliptic operators defined  on {\em nonlinear} ambient spaces, such as Riemannian manifolds. One of the first examples of this approach was the Bombieri - De Giorgi - Miranda gradient bound \cite{BDGM} for solutions of the {\em minimal surface equation}
\begin{equation}
\label{minsurf}
{\rm div}\left(\frac{D u}{\sqrt{1+|D u|^{2}}}\right)=0.
\end{equation}
The approach of \cite{BDGM}, later simplified in \cite{T72}, consisted in showing that $w=\log \sqrt{1+|D u|^{2}}$ is a subsolution of the Laplace-Beltrami operator naturally defined on the graph of $u$ considered as a Riemannian manifold. Since a Sobolev-Poincar\'e inequality can be proved for minimal graphs (see \cite{MS73} for a refinement to smooth minimal submanifolds), the Moser iteration yelds an $L^{\infty}-L^{1}$ bound on $w$ which is the core of the proof.

Another realm of application of the Harnack inequality are Bernstein theorem, i.e. Liouville type theorem for the minimal surface equation \eqref{minsurf}. More precisely Bernstein's theorem asserts that {\em any entire solution to \eqref{minsurf}} in $\R^{2}$ is affine. This statement is  known to be true in all dimension $N\le 7$ and false from $N=8$ onwards. One of the first applications in \cite{HAR} of Moser's (euclidean) Harnack inequality was to show that if in addition $u$ has bounded gradient the Bernstein statement holds true in any dimension. Indeed, one can differentiate \eqref{minsurf} with respect to $x_{i}$, giving a nonlinear equation which however can be seen as linear in $u_{x_{i}}$ with freezed coefficients. It turns out that if $|Du|$ is bounded  then the coefficients are elliptic and the Liouville property gives the conclusion. 

The approach of \cite{BDGM} was pushed forward in \cite{BG}, where a pure Harnack inequality was shown for general linear operators on minimal graphs. Taking advantage of their Harnack inequality, Bombieri and Giusti proved that if $N-1$ derivatives of a solution to \eqref{minsurf} are bounded, then also the $N^{\text{th}}$ is bounded, thus ensuring the Bernstein statement in any dimension thanks to the Moser result. See also \cite{Farina} for a direct proof of this fact using the Harnack inequality on minimal graph alone. 

For  other applications of the Harnack inequality on minimal graphs, see \cite{CY75}.

\subsection{Differential Harnack inequality}
A natural way to look at the Harnack estimate $u(x)\le C\, u(y)$ is to rewrite it as 
\[
\log(u(x))-\log(u(y))\le \log C=C',\qquad \text{for all $x, y\in B_{r}$}
\]
as long as $u> 0$ in $B_{2r}$. If one considers smooth functions (such as solutions to smooth elliptic equations) a way to prove the latter would be to look at it as a gradient bound on $\log u$. More concretely, it is a classical fact that Harmonic functions in $B_{2r}(x_{0})$ satisfy the gradient estimate
\[
|D u(x_{0})|\leq C(N)\frac{\sup_{B_{r}(x_{0})}|u|}{r},
\]
therefore Harnack's inequality implies that 
\[
\text{$u\geq 0$ in $B_{r}(x_{0})$}\quad \Rightarrow\quad |D u(x_{0})|\leq C(N)\frac{u(x_{0})}{r}.
\]
This can be rewritten in the following form:
\begin{theorem}[Differential Harnack inequality]
Let $u>0$ be harmonic in $B_{r}(x_{0})\subseteq \R^{N}$. Then 
\begin{equation}
\label{DH}
|D \log u (x_{0})|\leq \frac{C(N)}{r}.
\end{equation}
\end{theorem}

Inequality \eqref{DH} can be seen as the pointwise version of the integral estimate \eqref{logest} and as such it can be integrated  back along segments, to give the original Harnack inequality.  The differential form \eqref{DH} of the Harnack inequality clearly requires much more regularity than the Moser's one, however, it was proved to hold in the Riemannian setting for the Laplace-Beltrami equation in the ground-breaking works \cite{Y75, CY75}, under the assumption of non-negative Ricci curvature for the manifold. To appreciate the result, notice that all proofs of the Harnack inequality known at the time required a global Sobolev inequality, which is known to be false in general under the ${\rm Ric}\geq 0$ assumption alone. 

The elliptic Harnack inequality in the Riemannian setting proved in \cite{Y75} (and, even more importantly, its parabolic version proved soon after in \cite{LY}) again implies the Liouville property for semi bounded harmonic functions and it was one of the pillars on which modern geometric analysis grew. See for example the survey article \cite{L} for recent results on the relationship between Liouville-type theorems and geometric aspects of the underlying manifold. The book \cite{M} gives an in-depth exposition of the technique of differential Harnack inequalities in the framework of Ricci flow, culminating in Perelman differential Harnack inequality.

\subsection{Beyond smooth manifolds}

Clearly, the differential approach to the Harnack inequality is restricted to the Laplace-Beltrami operator, due to its smoothness and its close relationship with Ricci curvature given by the Bochner identity
\[
\Delta u=0\quad \Rightarrow\quad \Delta\frac{|D u|^{2}}{2}=|D^{2}u|^{2}+{\rm Ric}(D u, D u).
\]
It was only after the works \cite{G92, SC92} that a different approach to Moser's Harnack inequality on manifolds was found.\footnote{Actually, to a {\em parabolic version} of the Harnack inequality, which readily implies the elliptic one. For further details see the discussion on the parabolic Harnack inequality below and for a nice historical overview on the subject see \cite{SAL}, section 5.5.} Essentially, it was realized that in order to obtain the Harnack inequality, on a Riemannian manifold $(M, g)$ with corresponding volume form $m$ and geodesic distance, two ingredients suffices:
\begin{equation}
\label{dp}
\begin{split}
\text{--{\em Doubling  condition}:}\qquad &
m\big(B_{2r}(x_{0})\big)\leq C m\big(B_{r}(x_{0})\big)\\[10pt]
\text{--{\em Poincar\'e inequality}:}\qquad &
\int_{B_{r}(x_{0})}\Big|u-\intmed_{B_{r}(x_{0})} u\, dm\Big|^{2}\, dm \leq C\int_{B_{r}(x_{0})}|D u|^{2}\,dm
\end{split}
\end{equation}
for any $x_0\in M$ and $r>0$.
These two properties hold in any Riemannian manifold with nonnegative Ricci curvature, thus giving a Moser-theoretic approach to the Harnack inequality in this framework. What is relevant here is that Doubling\, \&\, Poincar\'e are stable with respect to quasi-isometries (i.e. bilipschitz homeomorphisms) and thus can hold in non-smooth manifolds, manifolds where ${\rm Ric}\geq 0$ does not hold (since curvature is not preserved through quasi-isometries), and/or for merely measurable coefficients elliptic operators. 
It is worth mentioning that Doubling\, \&\, Poincar\'e were also shown in \cite{CM97} to be sufficient conditions for the solution of Yau's conjecture on the finite-dimensionality of the space of harmonic functions of polynomial growth.

It was a long standing problem to give geometric conditions which are actually {\em equivalent} to the validity of the Harnack inequality, and thus to establish the stability of the latter with respect to quasi (or even rough) isometries. This problem has recently been settled in \cite{BM}, to which we refer the interested reader for bibliographic reference and discussion.

\section{Parabolic Harnack inequality}

\subsection{Original Parabolic Harnack}
Looking at the fundamental solution for the heat equation
\[
u_t -\Delta u=0,
\]
one finds out that there is no hope to prove a straightforward generalization of the Harnack inequality \eqref{H}. In the stationary case, ellipticity is preserved by spatial homotheties and traslations, thus the corresponding Harnack inequality turns out to be scale and traslation invariant. For the heat equation, the natural scaling  $(x, t)\mapsto (\lambda x, \lambda^2 t)$ preserves the equation and one expects a parabolic Harnack inequality to obey this invariance. Actually, an explicit calculation shows that it cannot hold for {\em fixed} times $t_0>0$ and corresponding space balls $B_{R_{0}}(x_0)$, even assuming that $t_0\geq 1$. However, a similar argument rules out the possibility of a Harnack inequality in parabolic cylinders as well. 
The correct parabolic form of the Harnack inequality was found and proved independently by Pini and Hadamard in \cite{PIN, HAD} and reads as follows.

\begin{theorem}
Let $u\geq 0$ be a solution of the heat equation in $B_{2\rho}(x_0) \times  [t_0 - 4\rho^2, t_0 + 4\rho^2]$. Then there exists a constant $C(N)$, $N$ being the dimension, such that
\begin{equation}
\label{pH}
\sup_{B_\rho(x_0)}u(\cdot, t_0 -  \rho^2)\leq C(N)\inf_{B_\rho(x_0)} u(\cdot, t_0 +\rho^2).
\end{equation}
\end{theorem}

As expected, this form of Harnack's inequality respects the scaling of the equations and introduces the notion of {\em waiting time} for a pointwise control to hold. It represents a quantitative bound from below on how much the positivity of $u(x_0, t_0)$  (physically, the temperature of a body at a point)  propagates forward in time: in order to have such  a bound in a whole ball of radius $r$ we have to wait a time proportional to $r^2$. 
\begin{figure}
\centering
\begin{tikzpicture}[scale=1.8]
\draw[->, thick] (0, -0.5) -- (0, 4.5) node[above left]{$t$};
\draw[->, thick] (-1, 0) -- (4, 0) node[below right]{$x$};
\fill[white!70!black] (0.5, 3) parabola bend (1.5, 2) (2.5, 3) -- cycle;
\draw (0.5, 3) parabola bend (1.5, 2) (2.5, 3);
\draw (1.5, 2.5) node{$P_T^+$};
\fill[white!90!black] (0.5, 1) parabola bend (1.5, 2) (2.5, 1) -- cycle;
\draw (0.5, 1) parabola bend (1.5, 2) (2.5, 1);
\draw (1.5, 1.5) node{$P_T^-$};
\filldraw (1.5,2) circle (0.03cm);
\draw (0, 0) -- (3, 0) -- (3, 4) -- (0, 4) -- cycle;
\draw (0, 4) node[left]{$4T$};
\draw (3, 0) node[below]{$2\sqrt{T}$};
\draw (2.5, 3.5) node{$u> 0$};
\end{tikzpicture}
\caption{Assuming $u>0$ in the boxed region, the dark grey area is  $P_T^+$ where $u$ is bounded below by $u(x_0, 2T)$, while the light grey is $P_+^T$ where $u$ is bounded above by $u(x_0, 2T)$.}
\label{pHfig1}
\end{figure}
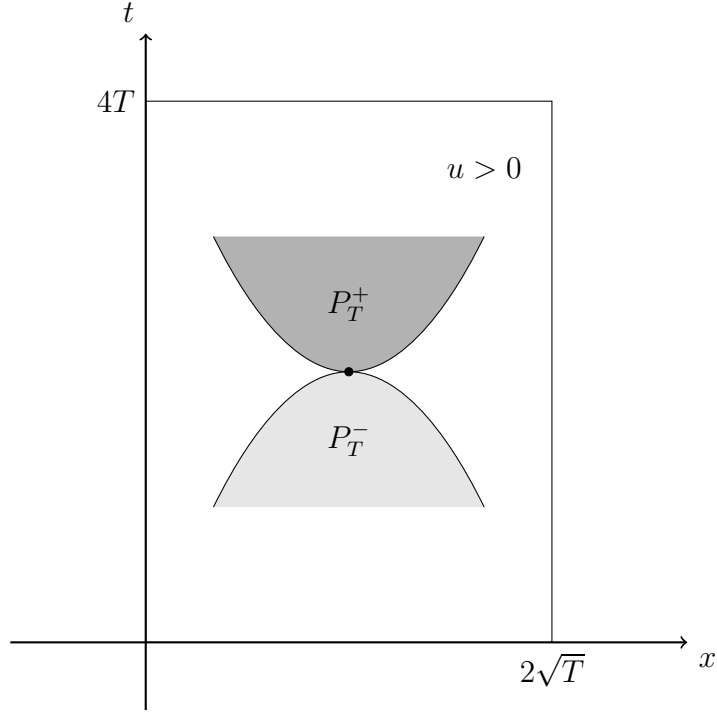

Another way of expressing this propagation for a nonnegative solution on $B_{2\sqrt{T}}(x_{0})\times [0, 4T]$ is the following, which, up to numerical factors is equivalent to \eqref{pH},
 \begin{equation}
\label{php}
C\inf_{P_T^+(x_0)}u\geq u(x_0, 2T)\geq C^{-1} \sup_{P_T^-(x_0)}u,
\end{equation}
where $P_T^\pm(x_0)$ are the part of the forward (resp. backward) space-time paraboloid with vertex $(x_0, 2T)$ in $B_{\sqrt{T}}(x_0)\times [T, 3T]$ (see Figure \ref{pHfig1}):
\[
P_T^+(x_0)=\{(x, t): T-t_{0}\geq t-t_0\geq |x-x_0|^2\}, \qquad P_T^-(x_0)=\{(x, t): t_{0}-T\geq t_0-t\geq |x-x_0|^2\}.
\]

A consequence of the parabolic Harnack inequality is the following form of the strong maximum principle. We sketch a proof here since this argument will play a r\^ole in the discussion of the Harnack inequality for nonlinear equations.

\begin{corollary}[Parabolic Strong Minimum Principle]
\label{minp}
Let $u\geq 0$ be a solution of the heat equation in $\Omega\times [0, T]$, where $\Omega$ is connected, and suppose $u(x_0, t_0)=0$. Then $u\equiv 0$ in $\Omega\times [0, t_0]$.
\end{corollary} 

\begin{proof}(sketch)
Pick $(x_1, t_1)\in \Omega\times \, ]0, t_0[$ and join it to $(x_0, t_0)$ with a smooth curve $\gamma:[0, 1]\to \Omega\times \,]0, t_0]$ such that $\gamma'$ has always a positive $t$-component. By compactness there is $\delta>0$ and a small forward parabolic sector $P^+_\eps=\{\eps\geq t\geq |x|^2\}$ such that: {\em 1)} $\gamma(\sigma)\in \gamma(\tau)+P_\eps^+$ for all $\sigma\in [\tau, \tau+\delta]$ and {\em 2)} the  Harnack inequality holds in the form \eqref{php} for all $s\in [0, 1]$, i.e. 
\[
u(\gamma(s))\leq \inf_{\gamma(s)+P_\eps^+} u.
\]
These two properties and $u(\gamma(1))=0$ readily imply $u(\gamma)\equiv 0$.
\end{proof}

\subsection{The linear case with coefficients}
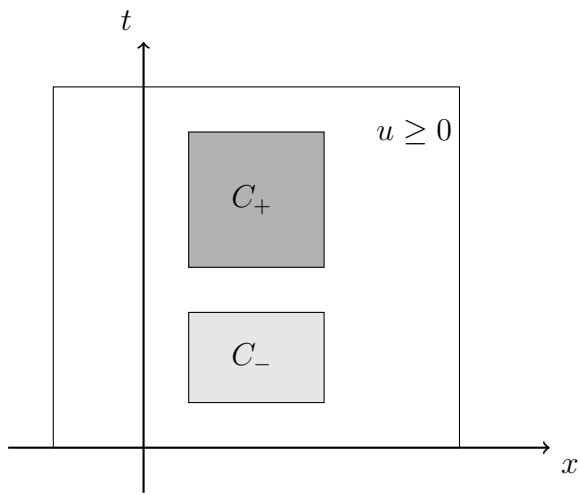
\begin{figure}
\centering
\begin{tikzpicture}[scale=1.2]
\draw[->, thick] (0, -0.5) -- (0, 4.5) node[above left]{$t$};
\draw[->, thick] (-1.5, 0) -- (4.5, 0) node[below right]{$x$};
\fill[white!90!black] (0.5, 0.5) -- (2, 0.5) -- (2, 1.5) -- (0.5, 1.5) --  cycle;
\draw (0.5, 0.5) -- (2, 0.5) -- (2, 1.5) -- (0.5, 1.5) --  cycle;
\draw (1.2, 1) node{$C_-$};
\fill[white!70!black] (0.5, 2) -- (2, 2) -- (2, 3.5) -- (0.5, 3.5) --  cycle;
\draw (0.5, 2) -- (2, 2) -- (2, 3.5) -- (0.5, 3.5) --  cycle;
\draw (1.2, 2.75) node{$C_+$};
\draw (-1, 0)-- (3.5, 0) -- (3.5, 4) -- (-1, 4) -- cycle;
\draw (3, 3.5) node{$u\geq 0$};
\end{tikzpicture}
\caption{The cylinders $C_+$ and $C_-$ where the Harnack inequality is stated.}
\label{pHfig2}
\end{figure}

In the seminal paper \cite{NA} on the H\"older regularity of solutions to elliptic parabolic equations with measurable coefficients, Nash already mentioned the possibility to obtain a parabolic Harnack inequality through his techniques. However, the first one to actually prove it was again  Moser, who in  \cite{MOS}  extended the Harnack inequality to linear parabolic  equations of the form
\begin{equation}
\label{PL}
u_t =\sum_{j,i=1}^N D_i(a_{ij}(x,t) D_ju).
\end{equation}

\begin{theorem}[Moser]
Let $u$ be a positive weak solution of \eqref{PL} in $B_{2r}\times [0, T]$, where $a_{ij}$ are measurable and satisfy the ellipticity condition \eqref{EC}. For any $0<t_{1}^{-}<t_{2}^{-}<t_{1}^{+}<t_{2}^{+}<T$ define 
\[
C_-:=B_{r}\times [t_{1}^{-}, t_{2}^{-}],\quad C_+:=B_{r}\times [t_{1}^{+}, t_{2}^{+}].
\]
Then it holds
\begin{equation}
\label{mph}
\sup_{C_-} u\leq C(N, \Lambda, \lambda, t_{1,2}^{\pm})\inf_{C_+} u,
\end{equation}
with a constant which is bounded as long as $t_{1}^{+}-t_{2}^{-}$ is bounded away from $0$.
\end{theorem}

Using the natural scaling of the equation, the previous form the parabolic Harnack inequality can be reduced to \eqref{pH}.

As in the elliptic case, the first step of Moser's proof consisted in the $L^{p}-L^{\infty}$ estimates for subsolutions, obtained by testing the equation with recursively higher powers of the solution.  This leads to 
\begin{equation}
\label{ml1}
\sup_{Q(\rho)}u^{p}\le \frac{C(N, p, \Lambda, \lambda)}{(r-\rho)^{N+2}}\iint_{Q(r)}u^{p}\, dx\, dt,\qquad r>\rho, \quad p>0
\end{equation}
where $Q(r)$ are parabolic cylinders having top boundary at the same fixed time $t_{0}$, say $Q(r)=B_{r}\times [t_{0}-r^{2}, t_{0}]$.
Since if $u$ is a positive solution, $u^{-1}$ is a positive subsolution, \eqref{ml1} holds true also for negative powers of $p$, yielding a bound from below for $u$ in terms of integrals of $u^{p}$. Similarly to the elliptic case, in order to obtain the parabolic Harnack inequality, Moser proceeded to prove
a crossover lemma which reads as
\begin{equation}
\label{col}
\int_{-1}^{0}\int_{B_{1}} u^{p_{0}}\, dx\, \int_{1}^{2}\int_{B_{1}} u^{-p_{0}}\, dx\le C,
\end{equation}
for some $C$ and small $p_{0}>0$ depending on $N$ and the ellipticity constants. This proved to be much harder than in the elliptic case, mainly because the integrals are taken on the two different and distant sets and no appropriate John-Nirenberg inequality dealing with this situation was known at the time. Moser himself proved such a parabolic version of the John-Nirenberg lemma yielding \eqref{col}, but the proof was so involved that he was forced to an erratum three years later. In \cite{MOS2} he gave a different proof avoiding it, following an approach of Bombieri and Giusti \cite{BG}. For this to work, he refined his  $L^{p}-L^{\infty}$ estimates \eqref{ml1}, showing that  they hold with constants independent from $p$, at least for sufficiently small values of $|p|$. As we will see, this was necessary for the Bombieri-Giusti argument to carry over. Despite the parabolic John-Nirenberg Lemma has later been given a simpler proof in \cite{FG},  the  {\em abstract John-Nirenberg Lemma} technique of \cite{BG} is nowadays the standard tool to prove parabolic Harnack inequalities, see e.g. \cite{SAL, KK}. On the other hand, Nash's program was later established in \cite{FAB}.

\vskip2pt

We next sketch the proof in \cite{MOS2}. The starting point is a logarithmic estimate, obtained by multiplying the equation by $u^{-1}\eta^{2}$, with $\eta\in C^{\infty}_{c}(B_{3})$, $\eta\ge 0$ and $\eta\equiv 1$ on $B_{2}$ and integrate in space only. Proceeding as in the elliptic case we obtain the differential inequality 
\[
\frac{d}{dt}\, \intmed_{B_{3}}\eta^{2}(x)\log u(x, t)\, dx+c\, \intmed_{B_{3}}|D\log u(x, t)|^{2}\, \eta^{2}(x)\, dx\le C.
\]
Under mild concavity assumptions on $\eta$, a weighted Poincar\'e inequality holds true with respect to the measure $d\mu=\eta^{2}(x)\, dx$, so that we infer
\[
\frac{d}{dt}\, \intmed_{B_{3}}\log u(x, t)\, d\mu+c\, \intmed_{B_{3}}\Big(\log u(x, t)-\intmed_{B_{3}}\log u(x, t)\, d\mu\Big)^{2}\, d\mu\le C.
\]
By letting  
\[
v(x, t)=\log u(x, t)-C\, t, \qquad M(t)=\intmed_{B_{3}} v(x, t)\, d\mu
\]
the previous inequality can be rewritten as
\[
\frac{d}{dt}\, M(t)+c\, \intmed_{B_{3}}\big(v(x, t)-M(t)\big)^{2}\, d\mu\le 0,
\]
so that $M(t)$ is decreasing.
Next, for $\lambda>0$ and $t\in [0, 4]$, restrict the integral over $\{x\in B_{2}: v(x, t)\ge M(0)+\lambda\}$ where, by monotonicity, $v(x, t)-M(t)\ge M(0)-M(t)+\lambda\ge \lambda$, to get
\[
\frac{d}{dt}\, M(t)+c\, (M(0)-M(t)+\lambda)^{2}\, |B_{2}\cap \{v(x, t)\ge M(0)+\lambda\}|\le 0,
\]
(notice that $d\mu=dx$ on $B_{2}$). Dividing by $(M(0)-M(t)+\lambda)^{2}$, integrating in $t\in [0, 4]$ and recalling that $M(t)\le M(0)$ we deduce
\[
|Q_{+}(2)\cap \{v\ge M(0)+\lambda\}|\le C/\lambda, \qquad Q_{+}(2):=B_{2}\times [0, 4].
\]
Similarly, for any $t\in [-4, 0]$, on $\{x\in B_{2}: v(x, t)\le M(0)-\lambda\}$ it holds $M(t)-v_{t}\ge M(t)-M(0)+\lambda\ge\lambda$ being $M$ decreasing and proceeding as before we get 
\[
|Q_{-}(2)\cap \{v\le M(0)-\lambda\}|\le C/\lambda, \qquad Q_{-}(2):=B_{2}\times [-4, 0].
\]
Recalling the definition of $v$, the last two displays imply the weak-$L^{1}$ estimate
\begin{equation}
\label{wl1}
|Q_{+}(2)\cap \{\log u\ge M(0)+\lambda\}| \le C/\lambda,\qquad |Q_{-}(2)\cap \{\log u\le M(0)-\lambda\}|\le C/\lambda,
\end{equation}
where $M(0)$ is a weighted mean of $\log u$. 
To proceed, we let
\[
w=u\, e^{-M(0)},\qquad Q_{+}(r)=B_{r}\times [4-r^{2},4],\qquad \varphi(r)=\sup_{Q_{+}(r)}\log w
\]
for $r\in [1, 2]$. Since $Q_{+}(r)\subseteq Q_{+}(2)$,  for all $\lambda>0$,
\[
|Q_{+}(r)\cap \{\log w\ge\lambda\}|\le C/\lambda. 
\]
We will prove  a universal bound on $\varphi(r)$ so we may suppose that $\varphi(r)$ is large. Estimate the integral of $w^{p}$ on $Q_{+}(r)$ splitting it according to $\log w\le \varphi(r)/2$ or $\log w>\varphi(r)/2$, to get
\[
\begin{split}
\iint_{Q_{+}(r)} w^{p}\, dx\, dt =\iint_{Q_{+}(r)}e^{p\log w}\, dx\, dt &\le e^{p\varphi(r)}|Q_{+}(r)\cap \{\log w\ge \varphi(r)/2\}|+|Q_{+}(r)| \, e^{\frac{p}{2}\varphi(r)}\\
&\le  \frac{2\, C}{\varphi(r)}\, e^{p\varphi(r)}+ c_{N}\, e^{\frac{p}{2}\varphi(r)}. 
\end{split}
\]
Choose now $p=p(r)$ such that 
\[
 \frac{2\, C}{\varphi(r)}\, e^{p\varphi(r)}=c_{N}\, e^{\frac{p}{2}\varphi(r)}\quad \Leftrightarrow\quad p=\frac{2}{\varphi(r)}\log(c\, \varphi(r)), \quad c:=c_{N}/(2C)
 \]
(where $\varphi(r)$ is so large that $p$ is positive and sufficiently small), so that
 \[
 \iint_{Q_{+}(r)} w^{p}\, dx\, dt\le 2c_{N}\, e^{\frac{p}{2}\varphi(r)}.
 \]
We use \eqref{ml1} (the constant being independent of $p$ for small $p$), obtaining for a larger $C$,
 \[
\varphi(\rho)\le \frac{1}{p}\log\left(\frac{C\, e^{\frac{p}{2}\varphi(r)}}{(r-\rho)^{N+2}}\right)=\frac{\varphi(r)}{2}+\frac{1}{p}\log\left(\frac{C}{(r-\rho)^{N+2}}\right)=\frac{\varphi(r)}{2}\left(1+\frac{\log(C/(r-\rho)^{N+2})}{\log(c\, \varphi(r))}\right)
\]
Therefore, either the second term in the parenthesis is greater than $1/2$, which is equivalent to
\[
\varphi(r)\le \frac{C^{2}}{c\, (r-\rho)^{2(N+2)}}
\]
or the opposite is true, giving  $\varphi(\rho)\le \frac 3 4\varphi(r)$. All in all we obtained
\[
\varphi(\rho)\le \frac{3}{4}\varphi(r)+\frac{C}{(r-\rho)^{2N+4}}.
\]
The latter can be iterated on an infinite sequence of radii $1=r_{0}\le r_{n}\le r_{n+1}\le \dots\le 2$ with, say, $r_{n+1}-r_{n}\simeq (n+1)^{-2}$, to get
\[
\varphi(1)\le C\sum_{n=0}^{\infty}(3/4)^{n}\, n^{4(N+2)},
\]
which implies $\sup_{Q_{+}(1)} u\, e^{-M(0)}\le C$ for some $C$ depending on $N$, $\Lambda$ and $\lambda$. Thanks to the second estimate in \eqref{wl1}, a completely similar argument holds true for $w=u^{-1}\, e^{M(0)}$ on the cylinders $Q_{-}(r)=B_{r}\times [-r^{2}, 0]$, yielding $\sup_{Q_{-}(1)}u^{-1}\, e^{M(0)}\le C$, i.e. $\inf_{Q_{-}(1)}u\, e^{M(0)}\ge C^{-1}$. Therefore we obtained 
\[
\frac{\sup_{Q_{+}(1)}u}{\inf_{Q_{-}(1)}u}= \frac{\sup_{Q_{+}(1)}u\, e^{M(0)}}{\inf_{Q_{-}(1)}u\, e^{M(0)}}\le C^{2}.
\]

\subsection{First consequences}
As in the elliptic case, the parabolic Harnack inequality provides an oscillation estimate giving the H\"older continuity of solutions to \eqref{PL} subjected to \eqref{EC}. Moreover, \eqref{mph} readily yields a strong minimum principle like the one in Corollary \ref{minp} for nonnegative solutions of \eqref{PL}.

\vskip5pt

On the other hand,  Liouville theorems in the parabolic setting are more subtle and don't immediately follow from the parabolic version of the Harnack inequality. In fact, the Liouville property is false in general since, for example, the function $u(x, t)=e^{x+t}$ is clearly  a nontrivial positive eternal (i.e., defined on $\R^{N}\times \R$) solution of the heat equation. 
A two sided bound is needed, and a fruitful setting where to state Liouville properties in the one of {\em ancient} solutions, i.e. those defined on an unbounded interval $]-\infty, T_{0}[$. An example is the following.
\begin{theorem}[Widder]\label{widder}
Let $u>0$ solve the heat equation in $\R^{N}\times\, ]-\infty, T_{0}[ $. Suppose for some $t_{0}<T_{0}$ it holds
\[
u(x, t_{0})\leq Ce^{o(|x|)},\qquad \text{for $|x|\to +\infty$}.
\]
Then, $u$ is constant.
\end{theorem}
 
The latter has been proved for $N=1$ in \cite{W}, and we sketch the proof in the general case. 
\begin{proof} By the Widder representation for ancient solutions (see \cite{LZ}) it holds
\begin{equation}
\label{W}
u(x, t)=\int_{\R^{N}}e^{x\cdot \xi+t|\xi|^{2}}\, d\mu(\xi)
\end{equation}
for some nonnegative Borel measure $\mu$. Let $\nu:=e^{t_{0}|\xi|^{2}}\mu$ and observe that by H\"older inequality with respect to the measure $\nu$ implies that for all $t\in ]0, 1[$
\[
u(tx+(1-t)y)=\int_{\R^{N}}e^{(tx+(1-t)y)\cdot \xi}\, d\nu(\xi)\le \left(\int e^{x\cdot \xi}\, d\nu(\xi)\right)^{t}\left(\int e^{y\cdot \xi}\, d\nu(\xi)\right)^{1-t}=u^{t}(x)u^{1-t}(y),
\]
i.e.,  $x\mapsto \log u(x, t_{0})$ is convex. As $\log u(x, t_{0})=o(|x|)$ by assumption, it follows that $x\mapsto u(x, t_0)$ is constant. Differentiating under the integral sign the Widder representation, we obtain
\[
0=\left.P(D_{x})u(x, t_{0})\right|_{x=0}=\int_{\R^N} P(\xi)\, d\nu(\xi)
\]
for any polinomial $P$ such that $P(0)=0$. By a classical Fourier transform argument, this implies that $\nu=c\, \delta_{0}$ and thus $u(x, t)\equiv c$ due to the representation \eqref{W}.
\end{proof}

Compare with \cite{SZ} where it is proved that under the growth condition $0\leq u\leq Ce^{o(|x|+\sqrt{|t|})}$ for $t\leq T_{0}$, there are no ancient  non-constant  solutions to the heat equation on a complete Riemannian manifold with ${\rm Ric}\geq 0$.

\vskip5pt

Using Moser's Harnack inequality, Aronsson  proved in \cite{A}  a {\em two sided} bound on the fundamental solution of \eqref{PL}, which reads
\begin{equation}
\label{kb}
\frac{1}{C (t-s)^{N/2}}e^{-C\frac{|x-y|^{2}}{t-s}}\leq \Gamma(t, x; s, y)\leq \frac{C}{(t-s)^{N/2}}e^{-\frac 1 C\frac{|x-y|^{2}}{t-s}}
\end{equation}
for some $C=C(N, \Lambda, \lambda)$ and $t>s>0$, where the fundamental solution (or {\em heat kernel}) is solves, for any fixed $(s, y)\in \R_{+}\times \R^{N}$
\[
\begin{cases}
\partial_{t}\Gamma =\sum_{j,i=1}^N D_{x_{i}}(a_{ij}(x,t) D_{x_{j}}\Gamma)&\text{in $\R^{N}\times\ ]s, +\infty[$},\\
\Gamma(x, t;  \cdot, s)\weakto^{*}\delta_{y},&\text{as $t\downarrow s$, in the measure sense}.
\end{cases}
\]
In \cite{FAB}, the previous kernel estimate was proved through Nash's approach, and was shown to be equivalent to the parabolic Harnack inequality.  

A {\em global} Harnack inequality also follows from \eqref{kb}, whose proof we will now sketch. If $u\geq 0$ is a solution to \eqref{PL} on $\R^{N}\times \R_+$ and $t>s>\tau\geq 0$, then using the representation
\[
u(x, t)=\int_{\R^{N}}\Gamma(x, t; \xi, \tau)u(\xi, \tau)\, d\xi,\qquad t>\tau,
\]
and the analogous one for $(y, s)$, we get
\[
\begin{split}
u(x, t)&=\int_{\R^{N}}\Gamma(x, t; \xi, \tau)\, \Gamma^{-1}(y, s; \xi, \tau)\, \Gamma(y, s; \xi, \tau)\, u(\xi, \tau)\, d\xi\\
&\geq \frac{u(y, s)}{C^{2}}\left(\frac{s-\tau}{t-\tau}\right)^{\frac N 2}\inf_{\xi}e^{\, \frac 1 C\frac{|y-\xi|^{2}}{s-\tau}-C\frac{|x-\xi|^{2}}{t-\tau}},
\end{split}
\] 
where $\tau\ge 0$ is a free parameter. Recalling that 
\[
\inf_{\xi} a\, |y-\xi|^{2}-b\, |x-\xi|^{2}=\frac{a\, b}{b-a}|x-y|^{2},\qquad a>b\ge 0,
\]
we consider two cases. If $s/t\leq 1/(2C^{2})$ we choose $\tau=0$ and compute
\[
\inf_{\xi}\frac{|y-\xi|^{2}}{C\, s} - C\frac{|x-\xi|^{2}}{t}\geq -2C\frac{|x-y|^{2}}{t-s}.
\]
If instead $s/t\in \ ]1/(2C^{2}), 1]$, we set $\tau=s-(t-s)/(2C^{2})>0$ obtaining
\[
\inf_{\xi}\frac{1}{C}\frac{|y-\xi|^{2}}{s-\tau}-C\frac{|x-\xi|^{2}}{t-\tau}\geq -2C\frac{|x-y|^{2}}{ t-s}.
\]
while $(s-\tau)/(t-\tau)=1/(1+2C^{2})$.
Therefore the kernel bounds \eqref{kb} imply the following Harnack inequality {\em at large}, often called {\em sub-potential lower bound}, for positive solutions $u$ of \eqref{PL} on $\R^{N}\times ]0, T[\, $: there exists a constant $C=C(N, \Lambda, \lambda)>1$ such that
\begin{equation}
\label{sub}
u(x, t)\geq \frac{1}{C}\, u(y, s)\left(\frac{s}{t}\right)^{\frac{N}{2}}e^{-C\frac{|x-y|^{2}}{t-s}}\qquad \text{for all $T>t>s>0$. }
\end{equation}
A similar global estimate, with a non-optimal exponent $\alpha=\alpha(N, \Lambda, \lambda)>N/2$ for the ratio $s/t$, was already derived through the so-called {\em Harnack chain} technique by Moser in \cite{MOS}.

\subsection{Riemannian manifolds and beyond}
Following the differential approach of \cite{Y75}, Li and Yau proved in \cite{LY} their celebrated parabolic differential Harnack inequality.

\begin{theorem}
Let $M$ be a complete Riemannian manifold of dimension $N\geq 2$ and ${\rm Ric}\geq 0$, and let $u>0$ solve the heat equation on $M\times\R_+ $. Then it holds
\begin{equation}
\label{LY}
|D (\log u)|^{2}-\partial_{t}(\log u)\leq \frac{N}{2t}.
\end{equation}
\end{theorem}
In the same paper, many variants of the previous inequality are considered, including one for local solutions in $B_{R}(x_{0})\times \, ]t_{0}-T, t_{0}[\, $ much in the spirit of \cite{CY75}, and several consequences are also derived. Integrating inequality \eqref{LY} along geodesics provides, for any positive solution of the heat equation of $M\times \R_+$
\begin{equation}
\label{LYp}
u(x, t)\geq u(y, s)\left(\frac{s}{t}\right)^{\frac{N}{2}}e^{-\frac{d^2(x, y)}{4(t-s)}},\qquad t>s>0,
\end{equation}
where $d(x, y)$ is the geodesic distance between two points $x, y \in M$. This, in turn, gives the heat kernel estimate (see \cite[Ch. 5]{SAL})
\begin{equation}
\label{hk}
\frac{1}{CV(x, \sqrt{t-s})}\ e^{-C\frac{d^2(x, y)}{t-s}}	\leq \Gamma(x, t; y, s)\leq \frac{C}{V(x, \sqrt{t-s})}\ e^{-\frac 1 C\frac{d^2(x, y)}{t-s}},
\end{equation}
where $V (x, r)$ is the
Riemannian volume of a geodesic ball $B(x, r)$. Notice that, in a general Riemannian manifold of dimension $N\geq 2$, 
\[
V(x, r)\simeq r^{N} \qquad \text{for small $r>0$}, 
\]
but, under the s\^ole assumption ${\rm Ric}\geq 0$, the best one can say is
\[
\frac{r}{C}\leq V(x, r)\leq Cr^{N},\qquad \text{for large $r>0$}.
\]
Therefore, while Li-Yau  estimate on the heat kernel coincides with Aronsson's one locally, it is genuinely different at the global level.

Other parabolic differential Harnack inequalities were then found by Hamilton in \cite{H93} for  compact Riemannian manifolds with ${\rm Ric}\geq 0$, and were later extended in  \cite{SZ, K} to complete, non-compact manifolds. Actually, far more general differential Harnack inequalities are available under suitable conditions on the Riemannian manifold, see the book \cite{M} for the history and applications of the latters.

Again, the differential Harnack inequality \eqref{LY} requires a good deal of smoothness both on the operator and on the ambient manifold. Yet, the corresponding pointwise inequality \eqref{LYp} doesn't depend on the smoothness of the metric $g_{ij}$ but only on its induced distance and the dimension, hence one is lead to believe that a smoothness-free proof exists. Indeed, the papers \cite{G92, SC92} showed that the parabolic Harnack inequality (and the corresponding heat kernel estimates)  can still be obtained through a Moser-type approach based solely on the Doubling \& Poincar\'e condition \eqref{dp}. Indeed, \cite{G92, SC92} independently showed the following {\em equivalence}.

\begin{theorem}[Parabolic Harnack principle]
For any Riemannian manifold the following are equivalent:
\begin{enumerate}
\item
The parabolic Harnack inequality \eqref{pH}.
\item
The heat kernel estimate \eqref{hk}.
\item
The Doubling \& Poincar\'e condition \eqref{dp}.
\end{enumerate}
\end{theorem}
Since Doubling \& Poincar\'e are stable with respect to quasi-isometries, the previous theorem ensures the stability of the parabolic Harnack inequality with respect to the latters, and thus its validity in a much wider class of Riemannian manifolds than those with ${\rm Ric}\geq 0$. Condition {\em (3)} also ensures that the parabolic Harnack inequality holds for general parabolic equations with elliptic and merely measurable coefficients, see \cite{SC2}. Actually, under local regularity conditions, it can be proved for {\em metric spaces} which are roughly isometric to a Riemannian manifold with ${\rm Ric}\geq 0$, such as suitable graphs or singular limits or Riemannian manifolds.

\subsection{The nonlinear setting}
An analysis of Moser's proofs reveals that the linearity of the second order operator is immaterial, and that
essentially the same arguments can be applied as well to nonnegative weak solutions
to a wide family of quasilinear equations. 
In \cite{AS, T68}, the  Harnack inequality in the form \eqref{mph} was proved to hold for nonnegative solutions of the quasilinear equation
\begin{equation}
\label{ql}
u_t =  {\rm div} A(x ,u ,Du)  
\end{equation}
where the function $A:\Omega\times \R\times \R^N\to \R^N$  is only assumed to be measurable and satisfying
\begin{equation}
\label{2gr}
\begin{cases}
A(x, s, z)\cdot z \geq  C_{0} |z|^2,\\
|A(x, s, z)| \leq  C_{1} |z|,
\end{cases}
\end{equation}
for some given positive constants  $0<C_{0}\le C_{1}$. These structural conditions are very general, as, for example, the validity of the comparison principle holds in general under the so-called {\em monotonicity condition}
\begin{equation}
\label{monot}
\big(A(x, s, z)-A(x, s, w)\big)\cdot (z-w)\ge 0 
\end{equation}
which does not follow from \eqref{2gr}. To appreciate the generality of \eqref{2gr}, consider the toy model case $N=1$, $A(x, s, z)=\varphi(z)$,  so that a smooth solution to \eqref{ql} fulfills $u_{t}=\varphi'(u_{x})u_{xx}$. Assuming \eqref{2gr} gives no information on the sign of $\varphi'$ except at $0$ (where $\varphi'(0)\ge C_{0}$), so that \eqref{2gr} is a backward parabolic equation in the region $\{u_{x}\in \{\varphi'<0\}\}$.

Trudinger noted in \cite{T68} that the Harnack inequality for the case of general $p$-growth conditions \eqref{pgrowth} with  $p\neq 2$ seemed instead a difficult task. He stated the validity of the Harnack inequality \eqref{mph} for positive solutions of the doubly nonlinear equation
\[
(u^{p-1})_t={\rm div} A(x, t ,u ,Du)
\]
where $A$ obeys \eqref{pgrowth} with the same $p$ as the one appearing on the right hand side, thus recovering a form of homogeneity in the equation which is lacking in \eqref{ql}. The (homogeneous) doubly nonlinear result has later been proved in \cite{GV, KK, UG}, (see also the survey \cite{KINN}), but it took around forty years to obtain the right form of the Harnack inequality for solutions of \eqref{ql} under the general $p$-growth condition \eqref{pgrowth} on the principal part.  The next chapter will be dedicated to this developement.

It is worth noting that another widely studied parabolic equation which presented the same kind of difficulties is the {\em porous medium equation}, namely
\[
u_t=\Delta u^m,\qquad m>0.
\]
 In fact, most of the results in the following sections have analogue statements and proofs for positive solutions of the porous medium equation. The interested reader may consult the monographs \cite{V, V2, HR} for the corresponding results for porous media and related literature.
More generally, the doubly nonlinear inhomogeneous equation 
\[  
u_t={\rm div} (u^{m-1}|Du|^{p-2}Du)
\]
has found applications in describing polytropic flows of a non-newtonian fluid in porous media \cite{BER} and soil science \cite{SW, BMo, KL}, see also the survey article \cite{Kal}. Regularity results can be found in \cite{PV, Iv1} and Harnack inequalities in \cite{FS} for the degenerate  and in \cite{FSV} in the singular case, respectively.
To keep things as simple as possible, we chose not to treat these equations, limiting our exposition to \eqref{ql}.

\section{Singular and degenerate parabolic equations}

\subsection{The prototype equation}
Let us consider the parabolic $p$-Laplace equation 

\begin{equation} \label{p_lapl}
u_t = {\rm div} (| D u|^{p-2} D u) , \ \ \  p>1,
\end{equation}
which can be seen as a parabolic elliptic equation with $|D u|^{p-2}$ as (intrinsic) isotropic coefficient. The coefficient  vanishes near a point where $D u=0$ when $p>2$, while it blows up near such a point when $p<2$. For this reasons we call \eqref{p_lapl} {\em degenerate} when $p>2$ and {\em singular} if $p<2$.

In the fifties, the seminal paper \cite{BAR} by Barenblatt
was the starting point of the study of the $p$-Laplacian
equation \eqref{p_lapl}. The following family of explicit solutions to \eqref{p_lapl} where found, and are since then called {\em Barenblatt solution} to \eqref{p_lapl}.

\begin{theorem}
For any $p>\frac{2N}{N+1}$ and $M>0$, there exist constants $a, b>0$ depending only on $N$ and $p$ such that the function
\begin{equation}           \label{barenblatt}
\B_{p, M}(x,t) := 
\begin{cases}
t^{-\frac{N}{\lambda}}
\left[a M^{\frac{p}{\lambda}\frac{p-2}{p-1}} - b\big(|x|\, t^{-\frac{1}{\lambda}}\big)^\frac{p}{p-1}\right]_+^\frac{p-1}{p-2}, &\text{if $p>2$},\\[15pt]
 t^{-\frac{N}{\lambda}}
\left[aM^{\frac{p}{\lambda}\frac{p-2}{p-1}} + b\big(|x|\,t^{-\frac{1}{\lambda}}\big)^\frac{p}{p-1}\right]^\frac{p-1}{p-2}&\text{if $2>p$},
\end{cases}
\end{equation}
where $\lambda = N(p-2) +p>0$, solves the problem
\[
\begin{cases}
u_t = {\rm div} (| D u|^{p-2} D u)&\text{in $ \R^N\times \, ]0, +\infty[$},\\
u(\cdot, t)\weakto^* M\delta_0&\text{ as $t\downarrow 0$}.
\end{cases}
\]
\end{theorem}

The functions $\B_{p, M}$ are also called {\em fundamental solution of mass $M$}, or simply fundamental solution when $M=1$, in which case one briefly writes $\B_{p, 1}=\B_{p}$. Uniqueness of the fundamental solution for the prototype equation  was proved by Kamin and V\'azquez in \cite{KAM} (the uniqueness for general monotone operators is still not known).

The Barenblatt solutions show that when \eqref{p_lapl} is degenerate the  diffusion is very slow and the speed of the propagation of the support is finite, while  in the singular case  the  diffusion is very fast and the solution may become extinct in finite time.  These two phenomena are incompatible with a parabolic Harnack inequality of the form \eqref{pH} or \eqref{mph}, (suitably modified taking account of the natural scaling) such as
\begin{equation}
\label{tph}
C^{-1}\sup_{B_\rho(x_0)}u(\cdot, t_0 -  \rho^p)\leq u(x_0, t_0)\leq C\inf_{B_\rho(x_0)} u(\cdot, t_0 +\rho^p)
\end{equation}
with a constant $C$ depending only on $N$.
 Indeed, in the degenerate case the Barenblatt solution has compact support for any positive time, violating the strong minimum principle dictated by \eqref{tph} (the proof of Corollary \ref{minp} still works). Regarding the singular case, this incompatibility is not immediately apparent from the Barenblatt profile itself and in fact the strong minimum principle still holds for solutions defined in $\R^N\times ]0, T[$ when $p>\frac{2N}{N+1}$.  However, consider the solution of the Cauchy problem associated to \eqref{p_lapl} in a cylindrical domain $\Omega\times \R_+$ with $u(x, 0)=u_0(x)\in C^\infty_c(\Omega)$ and Dirichlet boundary condition on $\partial\Omega\times \R_+$, with $\Omega$ {\em bounded}. An elementary energetic argument (see \cite[Ch VII]{DI}) gives a suitable {\em extinction time} $T^*(\Omega, u_0)$ such that $u(\cdot, t)\equiv 0$ for $t>T^*$, again violating the strong minimum principle which would follow from \eqref{tph}. 

Let us remark here that for $1<p\leq\frac{2N}{N+1}=:p_*$ the Barenblatt profiles cease to exists. The exponent $p_{*}$ is called the {\em critical exponent} for singular parabolic equations and, as it will be widely discussed in the following, the theory is mostly complete in the {\em supercritical} case $p>p_{*}$. Solutions of critical and subcritical equations (i.e. with $p\in \ ]1, p_{*}]$) on the other hand, even in the model case \eqref{p_lapl}, exhibit odd and, in some aspects still unclear, features.

\subsection{Regularity}

Let us consider  equations of the type 
\begin{equation}
\label{qlp}
u_t = {\rm div} A(x,u, Du)
\end{equation}
with general measurable coefficients obeying
\begin{equation}
\label{pgr}
\begin{cases}
A(x,s, z)\cdot z \geq C_{0} |z|^p,\\
|A(x, s, z)|\leq C_{1} |z|^{p-1}.
\end{cases}
\end{equation}

We are concerned with weak solutions in $\Omega\times [0, T]$, namely those satisfying
\[
\left.\int u\varphi\, dx\right|^{t_2}_{t_1}+\int_{t_1}^{t_2}\int_\Omega\left[-u\varphi_t+A(x, u, Du)\cdot \varphi\right]\, dx\, dt=0
\]
where $\varphi$ is an arbitrary function such that $\varphi\in W^{1, 2}_{\rm loc}(0, T; L^2(\Omega))\cap L^p(0, T; W^{1, p}_0(\Omega)$.
This readily implies that 
\[
u\in C_{\rm loc}(0, T; L^2_{\rm loc}(\Omega))\cap L^p_{\rm loc}(0, T; W^{1,p}_{\rm loc}(\Omega)).
\]

In the case $p=2$, the local H\"older continuity of solutions to \eqref{qlp} has been proved in \cite{LA} through a parabolic De Giorgi approach. The case $p\neq 2$ was considered as a major open problem in the theory of quasilinear parabolic equation for over two decades. The main obstacle to its solution was that the energy and logathmic estimates for \eqref{qlp} are non-homogeneous when $p\neq 2$. It was solved by DiBenedetto \cite{DB88} in the degenerate case and  Chen and DiBenedetto in  \cite{CDB} for the singular case through an approach nowadays called {\em method of intrinsic scaling} (see the monograph \cite{UR} for a detailed description). Roughly speaking, in order to recover from the lack of homogeneity in the integral estimates one works in cylinders whose natural scaling is modified by the oscillation of the solution itself. In the original proof, these rescaled cylinders are then sectioned in smaller sub-cylinders and the so-called {\em alternative} occurs: either there exists a sub-cylinder where $u$ is sizeably (in a measure-theoretic sense) away from its infimum or in each sub-cylinder it is sizeably away from its supremum. In both cases a reduction in oscillation can be proved, giving the claimed H\"older continuity.

Stemming from recent techniques built to deal with the Harnack inequality for \eqref{qlp}, simpler proofs are nowadays available,  avoiding the analysis of said alternative.  In the last section we will provide such a simplified proof, chiefly based on \cite{NEW} and \cite{ON}.

As it turned out, H\"older continuity of {\em bounded} solutions to \eqref{pgr} (in fact, to much more general equations) always holds. 
In the degenerate case $p\geq 2$, a-priori boundedness follows from the natural notion of weak solution given above, but in the singular case there is a precise threshold: local boundedness is guaranteed only for $p>p_{**}:=\frac{2N}{N+2}$, which is therefore another critical exponent for the singular equation, smaller than $p_*$. However, when $1<p<p_{**}$, weak solutions may be unbounded: for example, a suitable multiple of
\[
v(x, t)=(T-t)_+^{\frac{1}{2-p}} |x|^{\frac{p}{p-2}}
\]
solves the model equation \eqref{p_lapl} in the whole $\R^N\times \R$. 

The critical exponents $p_*>p_{**}$ arise from the so-called $L^r-L^\infty$-estimates for sub-solutions, which are parabolic analogues of \eqref{linftylr}. Namely, when $p>p^*$, a $L^1-L^\infty$ estimate holds true, eventually giving the intrinsic parabolic Harnack inequality. If only  $p>p_{**}$ is assumed, one can still obtain a weaker $L^r-L^\infty$ estimate with $r>1$ being the optimal exponent in the parabolic embedding
\[
L^\infty(0, T; L^2(B_R))\cap L^p(0, T; W^{1,p}(B_R))\hookrightarrow L^r(0, T; L^r(B_R)),\qquad r=p\frac{N+2}{N}
\]
which is ensured by the notion of weak solution. 

\vskip5pt

Finally, we briefly comment on the regularity theory for parabolic {\em systems}. The general measurable coefficient condition dictated by \eqref{pgr} is not enough to ensure continuity, and either some additional structure is required (the so-called {\em Uhlenbeck structure}, due to the seminal paper \cite{UH} in the elliptic setting) or regularity holds everywhere except in a small {\em singular set}. The parabolic counterpart of \cite{UH} has first been proved in \cite{DBF85} and systematized in the monograph \cite{DI} for a large class of  nonlinear parabolic system with Uhlenbeck structure. For the more recent developments on the partial regularity theory for parabolic system with general structure we refer to the memoirs \cite{DMK}, \cite{BDM}.

\subsection{Intrinsic Harnack inequalities}

 DiBenedetto and DiBenedetto \& Kwong in 
\cite{INT} and \cite{KW}  found and proved the suitable form of the parabolic Harnack inequality for the prototype equation \eqref{p_lapl}, respectively in the degenerate and singular case. The critical value $p_{*}=2N/(N+1)$ was shown to be the threshold below which no Harnack inequality, even in intrinsic form, may hold. However, comparison theorems where essential tools of the proof. A similar statement was later proved to hold for general parabolic quasilinear equations of $p$-growth in \cite{GIA} (degenerate case) and in \cite{FOR} (singular supercritical case), with no monotonicty assumption. We will now describe the results, starting from the degenerate case.

\begin{theorem}[Intrinsic Harnack inequality, degenerate case]\label{Hdeg}
Let $p\geq 2$ and $u$ be a non negative weak solution in $B_{2r}\times [-T, T]$ of \eqref{qlp} under the growth conditions \eqref{pgr}.
There exists $C>0$ and $\theta>0$, depending only on $N, p, C_{0}, C_{1}$ such that if $0<\theta \, u(0, 0)^{2-p}\, r^{p}\le T$, then
 \begin{equation}
 \label{pharnack}
C^{-1 }\sup_{B_{r}}u(\cdot, -  \theta \, u(0, 0)^{2-p}\, r^{p}) \leq u(0, 0) \leq  C \inf_{B_{r}}u(\cdot, \theta \, u(0, 0)^{2-p}\, r^{p}).
\end{equation}
\end{theorem}

Clearly, for $p=2$ we recover \eqref{pH}. For $p>2$, the waiting time is larger the smaller $u(0, 0)$ is; in other  terms  $u(0, 0)$ bounds from below $u$ on $p$-paraboloids of opening proportional to $u(0, 0)^{p-2}$. It is worth noting here two additional difficulties in the Harnack inequality theory with respect to the linear (or more generally, homogeneous) setting. While it is still true that the forward form in the quasilinear setting  implies the backward one, this is no more trivial due to the intrinsic waiting time depending on $u_{0}$. 
\vskip5pt
The Harnack inequality in the singular setting turns out to be much more rich and subtle than in the degenerate case. A natural guess would be that \eqref{pharnack} holds also in the singular case. However, consider the function
\begin{equation}
\label{ex0}
u(x, t)=(T-t)_+^{\frac{N+2}{2}}\left(a+b|x|^{\frac{2N}{N-2}}\right)^{-\frac{N}{2}},
\end{equation}
which is a bounded solution in $\R^N\times\R$ of the prototype equation \eqref{p_lapl} for any $p\in \, ]1, p_*[$, $N>2$ and suitably chosen $a, b>0$. The latter violates both the forward and backward Harnack inequality in \eqref{pharnack}, as the right hand side vanishes for sufficiently large $r$, while the left hand side goes to $+\infty$ for $r\to +\infty$. A similar phenomenon persist at the critical value $p=p_{*}$, as is shown by the entire solution 
\begin{equation}
\label{ex}
u(x, t)=\left(e^{ct}+|x|^{\frac{2N}{N-1}}\right)^{-\frac{N-1}{2}}
\end{equation}
for suitable $c>0$:  the left hand side of \eqref{pharnack} goes to $+\infty$ while the right hand one vanishes as $r\to +\infty$. It turns out that for $p\in \ ]p_{*}, 2[$, Theorem \ref{Hdeg} has a corresponding statement.

\begin{theorem}[Intrinsic Harnack inequality, singular supercritical case]\label{hsing}
Let $2>p> \frac{2N}{N+1}$ and $u$ be a non negative weak solution in $B_{4r}\times [-T, T]$ of \eqref{qlp} under the growth conditions \eqref{pgr}.
There exists $C>0$ and $\theta>0$, depending only on $N, p, C_{0}, C_{1}$ such that if $u(0, 0)>0$ and
\begin{equation}
\label{hypsing}
r^{p}\, \sup_{B_{2r}}u^{2-p}\le T,
\end{equation}
then
\begin{equation}
\label{thEsing}
C^{-1 }\sup_{B_{r}}u(\cdot, -  \theta \, u_{0}^{2-p}\, r^{p}) \leq u_{0} \leq  C \inf_{B_{r}}u(\cdot, \theta \, u_{0}^{2-p}\, r^{p}).
\end{equation}
\end{theorem}

Assumption \eqref{hypsing} seems technical, however no proof is known at the moment without it. Following the procedure in \cite{KW}, it can be removed for  solutions of monotone equations fullfilling \eqref{monot}  (and thus obeying the comparison principle). The proof of the intrinsic Harnack inequality for supercritical singular equations is considerably more difficult than in the degenerate case and crucially relies on the following $L^{1}$-form of the Harnack inequality, first observed in \cite{HP} for the porous medium equation, which actually holds in the full singular range.

\begin{theorem}[$L^{1}$-Harnack inequality for singular equations]
Let $p\in \ ]1, 2[$ and $u$ be a non negative weak solution in $B_{4r}\times [0, T]$ of \eqref{qlp} under the growth conditions \eqref{pgr}.
There exists $C>0$ depending only on $N, p, C_{0}, C_{1}$ such that
\[
\sup_{t\in [0, T]}\int_{B_{r}} u(x, t)\, dx\le C\, \inf_{t\in [0, T]}\int_{B_{2r}} u(x, t)\, dx +C\, \left(T/r^{p+N(p-2)}\right)^{\frac{1}{2-p}}.
\]
\end{theorem}

Notice that $p+N(p-2)>0$ if and only if $p>p_{*}$. Thanks to this deep result, an elliptic form of the intrinsic Harnack inequality can be proved.

\begin{theorem}[Elliptic Harnack inequality for singular supercritical equations]\label{IEH}
Let $p\in \ ]\frac{2N}{N+1}, 2[$ and $u$ be a non negative weak solution in $B_{4r}\times [-T, T]$ of \eqref{qlp} - \eqref{pgr}.
There exists $C>0$ and $\theta>0$, depending on $N, p, C_{0}, C_{1}$ such that if $u(0, 0)>0$ and \eqref{hypsing} holds, then
\begin{equation}
\label{fbh}
C^{-1}\sup_{Q_{r}}u\leq u(0, 0)\le C\,\inf_{Q_{r}}u,\qquad Q_{r}=B_{r}\times [-\theta\, u(0, 0)^{2-p}\, r^{p}, \theta\, u(0, 0)^{2-p}\, r^{p}].
\end{equation}
\end{theorem}

Recall that an elliptic form of the Harnack inequality such as \eqref{fbh} cannot not hold for the classical heat equation. This forces the constants appearing in the previous theorem to blow-up as $p\uparrow 2$, hence,  while this last form of the intrinsic Harnack inequality clearly implies \eqref{thEsing}, the constants in \eqref{thEsing} are instead stable as $p\uparrow 2$.  The previous examples also show that both constants must blow-up for $p\downarrow p_{*}$. The same comments following Theorem \ref{hsing} on the r\^ole of hypothesis \eqref{hypsing} can be made.

In the subcritical case, different forms of the Harnack inequality have been considered. Here we mention the one obtained in \cite{FV} generalizing to monotone operators a result of Bonforte and V\'azquez \cite{BON}, \cite{VAZ} on the porous medium equation.

\begin{theorem}[Subcritical case]
Let $p\in \ ]1, 2[$, $u$ be a positive, locally bounded weak solution in $B_{2r}\times [-T, T]$ of \eqref{qlp} under the growth conditions \eqref{pgr} and the monotonicty assumption \eqref{monot}. For any $s\geq 1$ such that $\lambda_{s}:=N\, (p-2)+p\,s>0$  there exists $C, \delta, \theta>0$, depending on $N, p, s, C_{0}, C_{1}$ such that letting
\[
\widetilde{Q}_{r}(u)=B_{r}\times \left[\theta\, \big(\intmed_{B_{r}} u(x, 0)\, dx\big)^{2-p}r^{p}, \theta\, \big(\intmed_{B_{r}} u(x, 0)\, dx\big)^{2-p}(2r)^{p}\right],
\]
if $u(0, 0)>0$ and $\widetilde{Q}_{2r}(u)\subseteq B_{2r}\times [0, T]$, then 
\begin{equation}
\label{gjh}
\sup_{\widetilde{Q}_{r}(u)} u\leq C \, A_{u}^{\delta}\,\inf_{\widetilde{Q}_{r}(u) } u,\qquad A_{u}=\left[\dfrac{\intmed_{B_{r}} u(x, 0)\, dx}{\left(\intmed_{B_{r}}u^{s}(x, 0)\, dx\right)^{\frac 1 s}}\right]^{\frac{p\, s}{\lambda_{s}}}
\end{equation}
\end{theorem} 

Notice that \eqref{gjh} is an elliptic Harnack inequality for later times, intrinsic in terms of the size of $u$ at the initial time $t=0$. In the singular supercritical case one can take $r=1$ and thus $A_{u}\equiv 1$ in the previous statement to recover partially Theorem \ref{IEH}. The main weakness of  \eqref{gjh} lies in the dependence of the Harnack constant from the solution itself. In general, a constant depending on $u$ won't allow to deduce H\"older continuity but, as noted in \cite{FV}, the peculiar structure of $A_{u}$ permits such a deduction.
 
Other weaker forms not requiring the monotonicity assumption \eqref{monot} are available (see \cite{DGV}), however the complete picture in the subcritical case is not completely clear up to now.

\subsection{Liouville theorems} 

As for the classical heat equation, a one sided bound is not sufficient to ensure triviality of the solutions of the prototype equation \eqref{p_lapl}. Indeed, a suitable positive multiple of the function 
\begin{equation}
\label{exl}
u(x, t)=(1-x+ct)_{+}^{\frac{p-1}{p-2}}
\end{equation}
solves \eqref{p_lapl} on $\R\times \R$ whenever $c>0$ and $p>2$. As is natural with parabolic Liouville theorems, the natural setting is the one of ancient solutions and it turns out a two-sided bound at a fixed time is sufficient. The basic tools to prove the following results are the previously discussed Harnack inequalities and the following results are contained in \cite{LIOU}.

\begin{theorem}
Let $p>2$ and $u$ be a non-negative solution of 
\begin{equation}
\label{po}
u_{t}={\rm div}(A(x, u, Du))\qquad \text{on $\R^{N}\times \, ]-\infty, T[$}
\end{equation}
under the growth condition \eqref{pgr}.
If for some $t_{0}<T$, $u(\cdot, t_{0})$ is bounded above, then $u$ is constant.
\end{theorem}

Notice that no monotonicity assumption on the principal part of the operator is needed. An optimal Liouville condition such as the one of Theorem \ref{widder} is unknown and clearly the example in \eqref{exl} shows that it must involve polynomial growth condition instead of a sub-exponential one. For the prototype parabolic $p$-Laplacian equation, a polynomial growth condition on both $x$ and $t$ more in the spirit of \cite{SZ} is considered in \cite{TU}.

On the complementary side, boundedness for {\em fixed} $x_{0}$ can also be considered, yielding:
\begin{theorem}
Let $p>2$ and $u$ be a nonnegative solution  in $\R^{N}\times \R$ of \eqref{po}, \eqref{pgr}. If 
\[
\limsup_{t\to +\infty} u(x_{0}, t)<+\infty\qquad \text{for some $x_{0}\in \R^{N}$},
\]
$u$ is constant.
\end{theorem}

In the singular, supercritical case, the elliptic form \eqref{fbh} of the Harnack inequality directly ensures that, contrary to what happens for classical heat equation, a one-sided bound suffices to obtain a Liouville theorem. This is no longer true in the critical and subcritical case, as the functions \eqref{ex} and \eqref{ex0} show. However, again a two sided bound suffices.

\begin{theorem}
Let $1<p<2$ and $u$ be a weak solution on $\R^{N}\times \, ]-\infty, T[$ of \eqref{po} under condition \eqref{pgr}. If $u$ is bounded, it is constant.
\end{theorem}

\subsection{Harnack estimates at large}

By Harnack estimates at large, we mean {\em global} results such as the sub-potential lower bound \eqref{sub} or the two-sided Kernel estimate \eqref{kb}. For the quasilinear equation
\begin{equation}
\label{hj}
 u_t = {\rm div} A(x,u, Du)
 \end{equation}
 with $p$-growth assumptions \eqref{pgr}, the natural candidates to state analogous inequalities are the Barenblatt profiles $\B_{p, M}$ given in \eqref{barenblatt}.
When $A$ satisfies smoothness and monotonicity assumptions such as
\begin{equation}
\label{ass}
\begin{cases}
(A(x, s, z)-A(x, s, w))\cdot(z-w)\geq 0 & \forall  s\in \R,\  x,  z, w \in \R^N, \\
|A(x, s, z)-A(x, r, z)|\leq \Lambda (1+|z|)^{p-1}|s-r|& \forall  s, r \in \R,\  x, z\in \R^{N}.
\end{cases} 
\end{equation}
 a comparison principle for weak solutions is available, as well as existence of solutions of the Cauchy problem with $L^1$ initial datum.

We start by considering the singular supercritical case, since the diffusion is fast and positivity spreads instantly on the whole $\R^{N}$, giving a behaviour similar to the one of the heat equation. The next result is contained in \cite{CAL}.

\begin{theorem}[Sub-potential lower bound, singular case]
Let  $\frac{2N}{N+1}<p<2$ and  $u$ be a nonnegative solution of \eqref{hj} in $\R^{N}\times \, ]0, +\infty[$ under assumptions \eqref{pgr}, \eqref{ass}.
There are constants $C, \delta>0$ depending on the data such that if $u(x_{0}, t_{0})>0$, then
\begin{equation}
\label{splb} 
u(x,t)\geq  \gamma\,  u(x_{0}, t_{0})\mathcal{B}_p\left(u(x_{0}, t_{0})^{\frac{p-2}{p}}\, \frac{x-x_0}{t_0^{ 1 /p}}, \frac{t }{t_0}\right),
\end{equation}
for all $(x, t)\in \R^{N}\times [t_{0}(1-\delta), +\infty[$.
\end{theorem}

As an example, assume  $x_{0}=0$, $t_{0}=1$ and  $u(0, 1) =1$.  Then, the previous sub-potential lower bound becomes 
\[
u(x,t) \geq \gamma \mathcal{B}_p (x,t)
\]
 for any  $ \displaystyle{(x,t) \in \mathbb R^N \times [ 1-\delta, \infty[}$. As a corollary, for any fundamental solution of \eqref{hj},  one obtains the two-sided kernel bounds (proved in \cite{RAG} for the first time)
\[
 C^{-1}\B_{p, M_{1}}(x, t)\leq \Gamma(x, t)\leq C \B_{p, M_{2}}(x, t)
\]
for some $C, M_{1}, M_{2}>0$ depending on the data.
Notice how the elliptic nature of \eqref{hj} for $p\in \, ]p_{*}, 2[$, as expressed by the forward-backward Harnack inequality \eqref{fbh}, allows to obtain the bound \eqref{splb} also for some $t<t_{0}$. Previously known sub-potential lower bounds correspond to the case $\delta=0$ above. As shown in \cite{CAL2}, the phenomenon of propagation of positivity for $t<t_{0}$ not only happens in the near past but, as long as the spatial diffusion has had enough room to happen, it also hold for arbitrarily remote past times. More precisely, in \cite{CAL2} it is proved that \eqref{splb} holds for all
\[
(x, t)\in {\mathcal P}^{c}:=\left\{t>0, |x-x_{0}|^{p} u(x_{0}, t_{0})^{2-p}>1-\frac{t}{t_{0}}\right\},
\]
while a weaker, but still optimal, lower bound holds in ${\mathcal P}$.

\vskip5pt

In the degenerate case $p>2$, the finite speed of propagation implies that if the initial datum $u_{0}$ has compact support, then any solution of \eqref{hj} keeps having compact support for any time $t>0$. The finite speed of propagation has been quantified in \cite{BOG}, under the s\^ole $p$-growth assumption \eqref{pgr}.

\begin{theorem}[Speed of propagation of the support]
Let $p>2$ and $u$ be a weak solution of the Cauchy problem
\[
\begin{cases}
u_{t}={\rm div} A(x, u, Du)&\text{in $\R^{N}\times \, ]0, +\infty[$},\\
u(x, 0)=u_{0}
\end{cases}
\]
under assumption \eqref{pgr}. If $R_{0}={\rm diam}({\rm supp}\, u_{0})<+\infty$, then
\[
{\rm diam}({\rm supp}\, u(\cdot, t))\leq  2R_{0} +C t^{1/\lambda}\|u_{0}\|_{L^{1}(\R^{N})}^{\frac{p-2}{\lambda}},
\]
where $\lambda=N(p-2)+p$ and $C$ depend only on $N, p, C_{0}$ and $C_{1}$.
\end{theorem}

Such an estimate actually holds for a suitable class of degenerate systems, see \cite{TV}. Sub-potential lower bounds  are obtained in \cite{BOG} as well. 

\begin{theorem}[Sub-potential lower bound, degenerate case]
Let  $p>2$ and  $u$ be a nonnegative solution of \eqref{hj} in $\R^{N}\times \, ]0, +\infty[$ under assumptions \eqref{pgr}, \eqref{ass}.
Then there are constants $C, \eps>0$ such that if $u(x_{0}, t_{0})>0$, then \eqref{splb} holds in the region 
\[
t>t_{0},\qquad |x-x_{0}|^{p}\leq \eps\, u(x_{0}, t_{0})^{p-2}t_{0}\min\left\{\frac{t-t_{0}}{t_{0}}, \left(\frac{t-t_{0}}{t_{0}}\right)^{p/\lambda}\right\},
\]
with $\lambda=N(p-2)+p$.
\end{theorem}

The last condition on the region of validity of \eqref{splb} is sharp, especially when $t\simeq t_{0}$ and the minimum is the first one (see \cite[Remark 1.3]{BOG} for details).

Under the additional assumptions \eqref{pgr} and \eqref{ass} fundamental solutions exist and, as in the singular case, the sub-potential lower bound implies a two-sided estimate on the kernel in terms of the Barenblatt solution.

\section{The expansion of positivity approach}

In this section we provide detailed proofs of some of  the Harnack inequalities stated until now. Historically, H\"older regularity  and Harnack inequalities have always been intertwined, with the former usually proved before the latter. The reason behind this is that H\"older regularity is a statement about a reduction in oscillation of $u$ in $B_{r}$ as $r\downarrow 0$, i.e. on the {\em difference} $\sup_{B_{r}}u-\inf_{B_{r}}u$. Thus it reduces to prove that {\em either} $\sup_{B_{r}}u$ decreases {\em or} $\inf_{B_{r}} u$ increases. On the other hand, a Harnack inequality implies the stronger statement that both $\sup_{B_{r}}u$ decreases {\em and} $\inf_{B_{r}} u$ increases at a certain rate (see the nice discussion in \cite[Ch. 1,  \S 10]{Landis}). 

The modern approach thus often shifted the statements, first proving a Harnack inequality and then deducing from it the H\"older continuity of solutions. We instead revert to the historical roadmap, for two main reasons. The first one is pedagogical, as it feels satisfactory to reach an important stepping-stone result such as H\"older regularity, which would anyway follow from the techniques needed  to prove the Harnack inequality. The second one is practical, since without continuity assumptions some of the arguments to reach, or even state, the Harnack inequality would be technically involved: for example, one would need to give a precise meaning to $u(0, 0)$ in \eqref{pharnack}.

We start in subsection 5.1 by considering the elliptic setting. The proof of the H\"older continuity follows closely the original De Giorgi approach, then we introduce the notion of expansion of positivity. A technique due to Landis allows to construct a largeness point from which to spread the positivity, thus giving the Harnack inequality. These are the common ingredients to all subsequent sections. In subsection 5.2 we apply this technique to homogeneous parabolic equations with only minor modifications. Then we start discussing degenerate and singular parabolic equation. Subsection 5.3 is devoted to the proof of common tools to both, subsection 5.4 to the degenerate case and the last one to singular supercritical equations.

While we won't prove basic propositions such as Energy estimates or Sobolev inequalities, the presentation will be mostly self contained. The only exception will be Theorem \ref{sl1}, which is the core tool to treat the singular supercritical Harnack inequality. It proof is rather technical and since we could not find any simplification we would simply rewrite \cite[Appendix A]{HR} word-by-word. Incidentally, this will also be the only $\sup$ estimate we will use. In striking contrast with the Moser method, in all the other subsections we will only assume {\em qualitative} boundedness of the solution (which certainly holds, as discussed in the previous section) without ever proving or using a quantitative integral $\sup$-bound.

Since some argument will be ubiquitous, a detailed discussion will be given at their first appearance, but we will  only sketch the relevant modifications on subsequent occurences. For this reason, the non-expert is advised to follow the path presented here from its beginning, rather than skipping directly to the desired result.
   
\subsection{Elliptic equations}\ \\

\vskip-8pt

We now describe the De Giorgi technique to prove $C^{\alpha}$-regularity and Harnack inequality for solutions of elliptic equations of the form
\begin{equation}
\label{ellittica}
{\rm div}A(x, u, Du)=0
\qquad \text{with}\quad
\begin{cases}A(x, s, z)\cdot z\ge C_{0}|z|^{p}\\
|A(x, s, z)|\le C_{1}|z|^{p-1}
\end{cases}
\quad p\in \ ]1, N[.
\end{equation}
We will not treat boundedness statements (which actually hold true in this setting) and always assume that solutions are locally bounded.

Roughly speaking, the approach of De Giorgi consisted in deriving pointwise  estimates on a solution $u$ by analizing the behaviour of $|\{u\le k\}\cap B_{r}|$ with respect to the level $k>0$. First, he proved that the relative size of the sublevel set shrinks as $k$ decreases, at a certain (logarithmic) rate. Then he showed that, when a suitable smallness threshold is reached, it starts decaying exponentially fast, so that it vanishes at a strictly positive level. This procedure produces a pointwise bound from below for $u$ in terms of the size of its sublevel set in a larger ball and is thus called a {\em measure-to-point estimate} in the literature. This estimate, moreover, expands in space, since the relative size of a sublevel set in a larger ball $B_{R}$ can also be bounded from below (polynomially in $r/R$) by its size in $B_{r}\subseteq B_{R}$. The quantitative statement arising from this simple observation is called {\em expansion of positivity} and is the basis for the proof of the Harnack inequality.

With a certain abuse of notation, we will say that $u$ is a (sub-)\,super-solution of \eqref{ellittica} if there exists and $A$ obeying the prescribed growth condition for which $-{\rm div}A(x, u, Du)(\le)\ge0$. 
Observe that, being \eqref{ellittica} homogeneous, the class of (sub-/super-) solutions of \eqref{ellittica} with $A$ is invariant by scaling, translation and (positive) scalar multiplication. More precisely, performing such transformations to a subsolution of \eqref{ellittica} for some $A$ results in a subsolution of \eqref{ellittica} for a possibly different $\tilde{A}$, which nevertheless obeys the same bounds. We will use the following notations: $K_{r}(x_{0})$ wil denote a cube of side $r$ and center $x_{0}$, $K_{r}=K_{r}(0)$ and, respectively, 
\[
P(K;\, u\lesseqgtr k)=\frac{|K\cap \{u\lesseqgtr k\}|}{|K|},
\]
thus, for example, $P(K;\, u\ge 1)$ is the percentage of the cube $K$ where $u\ge 1$. 
In the following, the dependance from $p$ in the constants will always be omitted, and any constant only depending on the $N$, $p$, $C_{0}$ and $C_{1}$ (the ``data'') will be denoted with a bar. Often we will also consider functions $f:\R_{+}\to \R_{+}$ which will also depend on the data, and we will omit such a dependance. We first recall some basic facts.

\begin{proposition}
\quad 
\begin{description}
\item[1)]\cite[Lemma II.5.1]{HR}
Let $X_{n}\ge 0$ obey for some $\alpha>0$, $b, C>0$, the iterative inequality  
\[
X_{n+1}\le C\, b^{n}\, X_{n}^{1+\alpha}.
\]
Then 
\begin{equation}
\label{iteration}
X_{0}\le C^{-1/\alpha} b^{-1/\alpha^{2}}\quad \Rightarrow \quad \lim_{n}X_{n}=0.
\end{equation}
\item[2)  De Giorgi - Poincar\'e inequality]\cite[Lemma II.2.2]{HR}
For any $u\in W^{1,1}(K_{r})$ and $k\le h$
\[
(h-k)|\{ u\le k\}|\le \frac{C(N)\, r^{N+1}}{|\{u\ge k\}|}\int_{\{k< u\le h\}}|Du|\, dx.
\]
\item[3)  Energy inequality]Let $u$ be a supersolution to \eqref{ellittica} in $K$. Then there exists $\bar C$ such that for any $k\in \R$ and $\eta\in C^{\infty}_{c}(K)$
\begin{equation}
\label{ellein}
\int_{K} |\nabla(\eta (u-k)_{-}))|^{p}\le \bar C\int_{K}(u-k)_{-}^{p}|D\eta|^{p}\, dx.
\end{equation}\end{description}
\end{proposition}

\begin{lemma}[Shrinking lemma]\label{esh}
Let $u\ge 0$ be a supersolution in $K_{R}$. For any $\mu>0$ there exists $\beta(\mu)>0$ such that  
\[
P(K_{R/2};\, u\ge 1)\ge \mu \quad \Rightarrow \quad P(K_{R/2};\,  u\le 1/2^{n})\le \beta(\mu)/n^{1-\frac 1 p }.
\]
\end{lemma}

\begin{proof}
Recale to $R=1$ and let  $k_{j}=2^{-j}$. By the De Giorgi-Poincar\'e inequality
\begin{equation}
\label{el4}
\begin{split}
(k_{j}-k_{j+1})|K_{1/2}\cap \{u\le k_{n+1}\}|&\le \frac{\bar C}{|K_{1/2}\cap \{u\ge k_{j}\}|}\int_{K_{1/2}\cap \{k_{j+1}<u\}}|D(u-k_{j})_{-}|\, dx\\
&\le\frac{\bar C}{\mu}\int_{K_{1/2}\cap \{k_{j+1}<u\}}|D(u-k_{j})_{-}|\, dx.
\end{split}
\end{equation}
If $\eta\in C^{\infty}_{c}(K_{1})$ is such that $0\le \eta\le 1$,  $\eta\equiv 1$ on $K_{1/2}$ and $|\nabla\eta|\le C(N)$, \eqref{ellein} gives
\[
\int_{K_{1/2}}|D(u-k_{j})_{-}|^{p}\, dx\le \bar C\int_{K_{1}}(u-k_{j})_{-}^{p}\, dx,
\]
so that the last integral in \eqref{el4} can be bounded through H\"older's inequality as 
\[
\begin{split}
\int_{K_{1/2}\cap \{ k_{j+1}<u\}}&|D(u-k_{j})_{-}|\, dx\le \left(\int_{K_{1/2}\cap \{ k_{j+1}\}}|D(u-k_{j})_{-}|^{p}\, dx\right)^{\frac{1}{p}}\left(|K_{1/2}\cap \{k_{j+1}< u\le k_{j}\}|\right)^{1-\frac{1}{p}}\\
&\le \bar C\left(\int_{K_{1}}(u-k_{j})_{-}^{p}\, dx\right)^{\frac{1}{p}}\left(|K_{1/2}\cap \{u\le k_{j}\}|-|K_{1/2}\cap \{u\le k_{j+1}\}|\right)^{1-\frac{1}{p}}
\end{split}
\]
Insert the latter into \eqref{el4}, use $(u-k_{j})_{-}\le k_{j}$ and $k_{j}-k_{j+1}=k_{j}/2$ to get
\[
\frac{k_{j}}{2}\, |K_{1/2}\cap \{u\le k_{j+1}\}|\le \frac{\bar C}{\mu}\, k_{j}\left(|K_{1/2}\cap \{u\le k_{j}\}|-|K_{1/2}\cap \{u\le k_{j+1}\}|\right)^{1-\frac{1}{p}}.
\]
Simplify the $k_{j}$'s, raise both sides to the power $p/(p-1)$ and sum over $j=0,\dots, n-1$. Since $|K_{1/2}\cap \{u\le k_{j}\}|$ is decreasing and $|K_{1/2}\cap \{u\le k_{j}\}|-|K_{1/2}\cap \{u\le k_{j+1}\}|$ telescopic, we obtain
\[
\begin{split}
n\, |K_{1/2}\cap \{u\le k_{n}\}|^{\frac{p}{p-1}}&\le \sum_{j=0}^{n-1}|K_{1/2}\cap \{u\le k_{j+1}\}|^{\frac{p}{p-1}}\\
&\le \frac{\bar C}{\mu^{\frac{p}{p-1}}}\left(|K_{1/2}\cap \{u\le k_{0}\}|-|K_{1/2}\cap \{u\le k_{n+1}\}|\right)\le \bar C \frac{1-\mu}{\mu^{\frac{p}{p-1}}}.
\end{split}
\]
\end{proof}

\begin{lemma}[Critical mass]\label{ecm}
Let $u\ge 0$ be a supersolution in $K_{R}$. There exists $\bar \nu$ s.\,t. 
\begin{equation}
\label{ell1}
P(K_{R};\, u\le 1)\le \bar \nu\quad \Rightarrow\quad u\ge 1/2 \quad \text{in $K_{R/2}$}.
\end{equation}
\end{lemma}

\begin{proof}
Scale back to $R=1$ and define for $n\ge 1$ $k_{n}=r_{n}=1/2+1/2^{n}$, $K_{n}=K_{r_{n}}$.
Let moreover  
\[
\eta_{n}\in C^{\infty}_{c}(K_{n}),\qquad 0\le \eta_{n}\le 1, \qquad \left.\eta_{n}\right|_{K_{n+1}}\equiv 1, \qquad |D\eta_{n}|\le \bar C\, 2^{n}
\]
and chain the Sobolev inequality with \eqref{ellein} with $k=k_{n}$, $\eta=\eta_{n}$, to obtain
\begin{equation}
\label{el3}
\int |(u-k_{n})_{-}\eta_{n}|^{p^{*}}\, dx\le \bar C\left(\int |\nabla((u-k_{n})_{-}\eta_{n})|^{p}\, dx\right)^{\frac{p^{*}}{p}}\le \bar C\left(\int_{K_{n}} 2^{np}(u-k_{n})_{-}^{p}\, dx\right)^{\frac{p^{*}}{p}}.
\end{equation}
On the right we use $(u-k_{n})_{-}\le k_{n}$ and $|K_{n}|\le 1$ to bound 
\[
\int_{K_{n}} (u-k_{n})_{-}^{p}\, dx\le  k_{n}^{p}\, |K_{n}\cap \{u\le k_{n}\}|\le   2^{-np}\, P(K_{n};\, u\le k_{n})
\]
while by $\eta_{n}\equiv 1$ on $K_{n+1}$ and Tchebicev's inequality, the left hand  side of \eqref{el3} bounds
\[
\begin{split}
\int &|(u-k_{n})_{-}\eta_{n}|^{p^{*}}\, dx\ge \int_{K_{n+1}}(u-k_{n})_{-}^{p^{*}}\, dx\ge \int_{K_{n+1}\cap \{u\le k_{n+1}\}}(u-k_{n})_{-}^{p^{*}}\, dx\\
&\ge  (k_{n}-k_{n+1})^{p^{*}} |K_{n+1}\cap \{u\le k_{n+1}\}|\ge 2^{-(n+1)p^{*}} 2^{-N} P(K_{n+1};\, u\le k_{n+1}).
\end{split}
\]
Use the previous two inequalities into \eqref{el3} to get
\[ 
P(K_{n+1};\, u\le k_{n+1})\le \bar C\, 2^{np^{*}}  P(K_{n};\, u\le k_{n})^{\frac{p^{*}}{p}}.
\]
The claim now follows from \eqref{iteration} applied to the sequence $X_{n}=P(K_{n};\, u\le k_{n})$.
\end{proof}

\begin{lemma}[Measure-to-point estimate]\label{MTP1}
Let $u\ge 0$ be a supersolution in $K_{R}$. For any $\mu>0$ there exists $m(\mu)>0$ such that
\begin{equation}
\label{ell5}
P(K_{R/2};\, u\ge k)\ge \mu\quad \Rightarrow \quad \inf_{K_{R/4}}u\ge m(\mu)\, k.
\end{equation}
\end{lemma}

\begin{proof}
Given $\mu>0$, choose $n_{\mu}\ge 1$ in Lemma \ref{esh}  such that $\beta(\mu)/n_{\mu}^{1-1/p}\le \bar\nu$, so that $P(K_{R/2};\, u\ge k/2^{n_{\mu}})\le \bar \nu$. Then apply \eqref{ell1} to $u/k$, obtaining \eqref{ell5} with  $m(\mu)=2^{-n_{\mu}-1}$. 
\end{proof}

\begin{theorem}[H\"older regularity]\label{ellcalpha}
Let $u$ solve \eqref{ellittica} in $K_{2R}$. There exists $\bar C, \bar \alpha>0$  s.\,t.  
\begin{equation}
\label{oscell}
{\rm osc}(u; K_{\rho})\le \bar C\, {\rm osc}(u;  K_{R})\, (\rho/R)^{\bar \alpha}\quad \text{for $0\le \rho\le R/2$}.
\end{equation}
\end{theorem}

\begin{proof}
Rescaling to $R=1$ and considering $u/{\rm osc}(u; K_{1})$ we can suppose ${\rm osc}(u, K_{1})=1$. Both $u_{+}=u-\inf_{K_{1}}u$ and $u_{-}=\sup_{K_{1}}u -u$ are non-negative solutions with ${\rm osc}(u_{\pm}; K_{1})=1$. Since
\[
P(K_{1};\, u_{+}\ge 1/2)=P(K_{1};\, u_{-}\le 1/2)=1-P(K_{1};\, u_{-}>1/2),
\]
at least one of $P(K_{1};\, u_{\pm}\ge 1/2)$ is at least $1/2$ and we can suppose without loss of generality that it is $u_{+}$. Then \eqref{ell5} with $R=2$, $k=1/2$ provides
\[
\inf_{K_{1/2}}u_{+}\geq m(1/2)/2=:\bar m\quad \Rightarrow\quad \inf_{K_{1/2}}u\ge\inf_{K_{1}} u+\bar m\quad \Rightarrow \quad {\rm osc}(u; K_{1/2})\leq 1-\bar m.
\]
Scaling back we obtained ${\rm osc}(u; K_{R/2})\le {\rm osc}(u; K_{R})(1-\bar m)$
which, iterated for $R_{n}=R/2^{n}$ gives 
\[
{\rm osc}(u; K_{R_{n}})\le {\rm osc}(u; K_{R})(1-\bar m)^{n}.
\]
For $\rho\le R/2$, let $n\ge 1$ obey $R_{n+1}\le \rho \le R_{n}$ and $\bar \alpha:=-\log_{2}(1-\bar m)$. Then, by monotonicity,
\[
{\rm osc}(u; K_{\rho})\le {\rm osc}(u; K_{R_{n}})\le {\rm osc}(u; K_{R})(1-\bar m)^{n}={\rm osc}(u; K_{R})2^{-n\, \bar\alpha}= 2^{\bar\alpha}\, {\rm osc}(u; K_{R})(2^{-(n+1)})^{\bar\alpha}
\]
giving the claim due to $2^{-(n+1)}\le \rho/R$.
\end{proof}

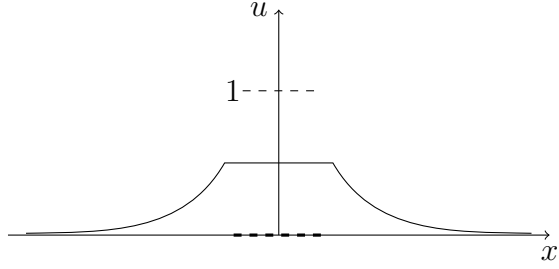
\begin{figure}
\centering
\begin{tikzpicture}[scale=1.2]

\draw[->] (-3, 0) -- (3, 0) node[below]{$x$};
\draw[->] (0, 0) -- (0, 2.5) node[left]{$u$};
\draw (-0.3,1.6) node[left]{$1$};
\draw[very thick, dashed] (-0.5, 0) -- (0.5, 0);
\draw[dashed] (-0.4, 1.6) -- (0.4, 1.6);
\draw (-2.8, 0.02) to[out=1, in=-120] (-0.6, 0.8) -- (0.6, 0.8) to[out=-60, in=179] (2.8, 0.02);
%\draw[very thick, white, dashed, opacity=1] (-0.5, 0.8) -- (0.5, 0.8);

\end{tikzpicture}
\caption{The expansion of positivity. If  $u\ge 1$ on the dashed part of the cube, it is bounded below  by a negative power of the distance from the cube. }
\label{EPf}
\end{figure}

\begin{theorem}[Expansion of positivity, see figure \ref{EPf}]\label{ellexpos}
Let $u\ge 0$ be a supersolution in $K_{R}$. There exists $\bar \lambda>1$ and,  for any $\mu>0$, $c(\mu)>0$   s.\,t. 
\begin{equation}
\label{ell6}
P(K_{r};\, u\ge 1)\ge \mu\quad \Rightarrow \quad \inf_{K_{\rho}}u\ge c(\mu) \, (r/\rho)^{\bar\lambda}\quad \text{if $r\le \rho\le R/2$}.
\end{equation}
\end{theorem}

\begin{proof}
Using the notations of Lemma \ref{MTP1}, we let $c=c(\mu):=m(\mu/2^{N})$ and iterate \eqref{ell5} as follows. From $P(K_{r};\, u\ge 1)\ge \mu$ we infer $P(K_{2r};\, u\ge 1)\ge \mu/2^{N}$ thus  \eqref{ell5} gives $\inf_{K_{r}} u\ge c$. If  $\bar \delta:=m(4^{-N})$ and $\rho_{n}=2^{n}r$, we thus have $P(K_{\rho_{0}};\, u\ge c\, \bar\delta^{0})=1$. Moreover
\[
P(K_{\rho_{n}};\, u\ge c\, \bar\delta^{n})= 1\quad \Rightarrow\quad  P(K_{\rho_{n+2}};\, u\ge c\, \bar\delta^{n })\ge 4^{-N}\ \  \underset{\eqref{ell5}}{\Rightarrow}\ \    P(K_{\rho_{n+1}};\, u\ge c\, \bar\delta^{n+1})=1.
\]
Thus, by induction, $u\ge c\, \bar\delta^{n}$ in $K_{\rho_{n}}$ for all $n\ge 0$ such that $\rho_{n+2}=2^{n+2}r\le R$. Given $\rho\in [r, R/2]$ let $n$ be such that $2^{n-1}\le \rho/r \le 2^{n}$. Then we obtained the claim with $\bar\lambda=-\log_{2}\bar\delta$, since
\[
\inf_{K_{\rho}}u\ge \inf_{K_{\rho_{n}}} u\ge c\, \bar \delta^{n}\ge c\, \bar \delta \, (\rho/r)^{\log_{2}\bar \delta}.
\]
\end{proof}

We call the exponent $\bar\lambda$ the {\em expansion of positivity rate}. 

\begin{theorem}[Harnack inequality]\label{harnackellittica}
There exists $\bar C>0$ such that for any locally bounded solution $u\ge 0$ to \eqref{ellittica} in $K_{8R}$  it holds 
\[
\sup_{K_{R}} u\le \bar C\, \inf_{K_{R}} u.
\]
\end{theorem}

\begin{proof}
Rescaling to $R=1$ and considering $u/\sup_{K_{1}}u$ we are reduced to prove that 
\begin{equation}
\label{claimharnack}
\sup_{K_{1}}u=1 \quad \Rightarrow\quad \inf_{K_{1}}u\ge \bar m>0
\end{equation}
for any solution $u\ge 0$ in $K_{8}$. We will find $\bar m>0$, $x_{0}\in K_{1}$ and $r>0$ such that 
\begin{equation}
\label{claimharnack2}
u(x_{0})\, r^{\bar \lambda}\ge \bar m, \qquad P(K_{r}(x_{0});\, u\ge u(x_{0})/2)\ge \bar \nu
\end{equation}
for $\bar \lambda$ given \eqref{ell6} and some universal $\bar \nu$. Theorem \ref{ellexpos} applied to $u/u(x_{0})$ will then prove \eqref{claimharnack} for such $r$, with the choices $R=8$, $k=u(x_{0})/2$, $\mu=\bar \nu$ and $\rho=2$, as $K_{1}\subseteq K_{2}(x_{0})\subseteq K_{4}$.\\
 To choose $x_{0}$ and $r$, observe that Theorem \ref{ellcalpha} implies that the function
\[
[0, 1]\ni \rho \mapsto \psi(\rho)=(1-\rho)^{\bar \lambda}\sup_{K_{\rho}}u
\]
is continuous and vanishes at $\rho=1$, thus it attains its maximum at some $\rho_{0}<1$ and we set
\[
\max_{[0, 1]}\psi=(1-\rho_{0})^{\lambda}\sup_{K_{\rho_{0}}}u=(1-\rho_{0})^{\lambda}\, u(x_{0})
\]
for some $x_{0}\in K_{\rho_{0}}$. Let $\xi\in \ ]0, 1[$ to be chosen and define  $r=\xi\, (1-\rho_{0})$. Then
\begin{equation}
\label{par10}
u(x_{0})\, r^{\bar \lambda}=\xi^{\bar \lambda}\, u(x_{0})\, (1-\rho_{0})^{\bar \lambda}=\xi^{\bar \lambda}\, \psi(\rho_{0})\ge \xi^{\bar \lambda}\, \psi(0)=\xi^{\bar\lambda}.
\end{equation}
Since $K_{r}(x_{0})\subseteq K_{\rho_{0}+r}$, we infer from $\psi(\rho_{0}+r)\le \psi(\rho_{0})$ that
\[
 \sup_{K_{r}(x_{0})}u\le \sup_{K_{\rho_{0}+r}}u= \frac{\psi(\rho_{0}+r)}{(1-\rho_{0}-r)^{\bar\lambda}}\le  \frac{\psi(\rho_{0})}{(1-\rho_{0}-r)^{\bar\lambda}}=\frac{(1-\rho_{0})^{\bar\lambda}}{(1-\rho_{0}-r)^{\bar \lambda}}\sup_{K_{\rho_{0}}}u=\frac{u(x_{0})}{(1-\xi)^{\bar\lambda}} .
\]
Choose $\bar\xi$ as per $(1-\bar \xi)^{-\bar\lambda}=2$, so that $u\le 2\, u(x_{0})$ in $K_{r}(x_{0})$, while \eqref{par10} gives the first condition in \eqref{claimharnack2} with $\bar m=\bar\xi^{\bar\lambda}$.  Apply \eqref{oscell} for $R=r$, $\rho=\bar \eta r$ with $\bar\eta$ s.\,t. $4\bar C \bar\eta ^{\bar \alpha}\le 1$, so that
\[
 {\rm osc}(u; K_{\bar\eta r}(x_{0}))\le \bar C\, {\rm osc}(u; K_{r}(x_{0}))\, \bar\eta^{\bar\alpha}\le  2\,\bar C\,  u(x_{0})\, \bar\eta^{\bar\alpha}\le u(x_{0})/2,
 \]
 implying $u\ge u(x_{0})/2$ in $K_{\bar\eta r}(x_{0})$. Thus, the second condition in \eqref{claimharnack2} holds for $\bar\nu=\bar\eta^{N}$.
\end{proof}

\subsection{Homogeneous parabolic equations}\ \\

\vskip-8pt
In the forthcoming subsections we will provide the extension of the previous techniques to the parabolic setting.
In order to highlight the similarities with the elliptic case, we will proceed step-by-step in increasing generality, gradually introducing the modifications needed to cater with the evolutionary framework.

First we will deal with homogeneous equations, i.e. those for which scalar multiplication still gives a solution of the same (from the structural point of view) type of equation. We chose for simplicity to deal with the quadratic case, i.\,e., with equations of the form
\begin{equation}
\label{parabo}
u_{t}={\rm div}A(x, u, Du),\qquad 
\begin{cases}
A(x, s, z)\cdot z\ge C_{0}|z|^{2}\\
|A(x, s, z)|\le C_{1}|z|.
\end{cases}
\end{equation}
As in the previous subsection, we say that $u$ is a (sub-)\,super-solution if there is some $A$ obeying the growth conditions and such that $u_{t} (\le)\ge {\rm div}A(x, u, Du)$.
An important feature of  \eqref{parabo} is that the class of its solutions is invariant by space/time translations, by the scaling $u_{\lambda}(x, t)=u(\lambda x, \lambda^{2}t)$, $\lambda>0$ and, more substantially, by scalar multiplication.
More generally, homogeneous problems of the form
\[
|u_{t}|^{p-2}u_{t}={\rm div}A(x, u, Du),\qquad 
\begin{cases}
A(x, s, z)\cdot z\ge C_{0}|z|^{p}\\
|A(x, s, z)|\le C_{1}|z|^{p-1}
\end{cases}
\]
can be dealt in the same way. In fact, as will be apparent from the proofs, in this homogeneous setting the Harnack inequality follows solely from the energy inequality. Indeed, in \cite{UG}, it has been proved for non-negative functions belonging to the so-called {\em parabolic De Giorgi classes} i.e., roughly speaking, functions obeying the energy inequality for truncations.

In the following, we set $Q_{R, T}=K_{R}\times [0, T]$. Given a rectangle  $Q=K\times [a, b]\subseteq \R^{N}\times \R$,  $u:Q\to \R$ and $k\in \R$ we define, respectively
\[
P\left(Q;\, u\lesseqgtr k\right)=\frac{\left|Q\cap \left\{u\lesseqgtr k\right\}\right|}{|Q|},\qquad P_{t}\left(K;\, u\lesseqgtr k\right)=\frac{\left|K\cap \left\{u(\cdot, t)\lesseqgtr k\right\}\right|}{|K|}.
\]
The dependance on $N$, $C_{0}$ and $C_{1}$ will always be omitted, and a constant $c$ depending only on the latters will be denoted by $\bar c$.
We also recall the relevant functional analytic tools.

\begin{proposition}
\quad 
\begin{description}
\item[1)  Parabolic Sobolev Embedding]\cite[Lemma II.4.1]{HR}
If $u\in L^{2}(0, T; W^{1, 2}_{0}(\Omega))$, then
\[
\int_{0}^{T}\int_{\Omega} |u|^{2\frac{N+2}{N}}\, dx\, dt\leq C_{N}\left(\sup_{t\in [0, T]}\int_{\Omega}u^{2}(x, t)\, dx\right)^{\frac{2}{N}}\int_{0}^{T}\int_{\Omega}|Du|^{2}\, dx\, dt.
\]
\item[2)  Energy inequality]\cite[Prop. III.2.1]{HR}
Let $u$ be a supersolution to \eqref{parabo} in $Q=K\times [0, T]$. There exists $\bar C>0$ s.\,t. for any $k\ge 0$ and $\eta\in C^{\infty}(a, b; C^{\infty}_{c}(K))$, $0\le \eta\le 1$ it holds
\begin{equation}
\label{einh}
\begin{split}
\sup_{t\in [0, T]}&\int_{K} (u(x, t)-k)_{-}^{2}\eta^{2} \, dx+\frac{1}{\bar C}\iint_{Q} |D(\eta(u-k)_{-})|^{2}\, dx\, dt\\
&\le \int_{K} (u(x, 0)-k)_{-}^{2}\eta^{2} \, dx+\bar C\iint_{Q} (u-k)_{-}^{2}|\nabla\eta|^{2}\, dx\, dt+\bar C\iint_{Q} (u-k)_{-}^{2}|\eta_{t}|\, dx\, dt.
\end{split}
\end{equation}
\end{description}
\end{proposition}

The first lemma shows how initial measure-theoretic positivity propagates at future times.

\begin{lemma}\label{lemmapar}
Let $u\ge 0$ be a supersolution  in $Q_{R, R^{2}}$. For any $\mu>0$ there are $k, \theta\in \, ]0, 1[$  s.\,t.
\begin{equation}
\label{assmeas2par}
P_{0}(K_{R};\, u\ge 1)\ge \mu\quad \Rightarrow\quad 
P_{t}\left(K_{R};\, u\ge k(\mu) \right)>\mu/2 \qquad  \forall t\in [0, \theta(\mu)\, R^{2}].
\end{equation}
\end{lemma}

\begin{proof}
Rescale to $R=1$ and, for any $\delta, \theta\in \ ]0, 1[$,  employ the energy inequality \eqref{einh} on $K_{1}\times [0, \theta ]$ with a cutoff function $\eta \in C^{\infty}_{c}(K_{1})$ independent of $t$ and such that 
\begin{equation}
\label{varphi2}
0\le \eta\le 1,\qquad \left.\eta\right|_{K_{\delta}}\equiv 1,\qquad |D\eta|\le C_{N}/(1-\delta).
\end{equation}
obtaining for any $t\in [0, \theta]$
\[
\int_{K_{\delta }} (u(x, t)-1)^{2}_{-}\, dx\le \int_{K_{1}} (u(x, 0)-1)^{2}_{-}\, dx+\frac{\bar C}{(1-\delta)^{2}}\int_{0}^{t}\int_{K_{1}}(u-1)^{2}_{-}\, dx\, dt\le 1-\mu+\frac{\bar C\, \theta}{(1-\delta)^{2}},
\]
where we used the assumption in \eqref{assmeas2par} in the last inequality. For $k\in \ ]0, 1[$ we have
\[
\int_{K_{\delta}} (u(x, t)-1)^{2}_{-}\, dx\ge \int_{K_{\delta}\cap \{u(\cdot, t)<k\}}(1-k)^{2}\, dx\ge  (1-k)^{2}| K_{\delta}\cap \{u(\cdot, t)<k\}|.
\]
Insert the latter into the previous one to obtain, for all $t\in [0, \theta]$
\begin{equation}
\label{par2}
\begin{split}
1-P_{t}(K_{1};\, u\ge k)&\le 
1-|K_{\delta}\cap \{u(\cdot, t)\ge k\}|=1-\delta^{N} + |K_{\delta}\cap \{u(\cdot, t)<k\}|\\
&\le 1-\delta^{N}+\frac{1}{(1-k)^{2} }\left(1-\mu+\frac{\bar C\, \theta }{(1-\delta)^{2}}\right).
\end{split}
\end{equation}
Successively choose $\delta, k\in\ ]0, 1[$ and, consequently, $\theta\in \ ]0, 1[$ so that:
\[
1-\delta^{N}=\frac{\mu}{8},\qquad \frac{1-\mu}{(1-k)^{2}}=1-\frac{3}{4}\mu,\qquad  \frac{1}{(1-k)^{2} }\frac{\bar C\, \theta}{(1-\delta)^{2}}\le\frac{\mu}{8} 
\]
to obtain that the right hand side in \eqref{par2} is less than $1-\mu/2$, proving the claim.
\end{proof}

The next two steps are fully in the spirit of the De Giorgi approach.

\begin{lemma}[Shrinking lemma]\label{slemmapar}
Suppose $u\ge 0$ is a supersolution in $Q_{2R,T}$  obeying
\begin{equation}
\label{ptpar}
P_{t}(K_{R};\, u\ge k)\ge \mu, \qquad \forall t\in [0, T]
\end{equation}
 for some $\mu\in \ ]0, 1[$, $k>0$. There exists $\beta=\beta(\mu)>0$ such that
\[
P\left(Q_{R, T};\,  u\le k/2^{n}\right)\le \beta(\mu)\Big(1+\frac{R^{2}}{T}\Big)^{1/2}\frac{1}{n^{1/2}}, \qquad 
\]
\end{lemma}

\begin{proof}
Let $k_{j}=k/2^{j}$, $j\ge 0$. The energy inequality \eqref{einh} with $\eta\in C^{\infty}_{c}(K_{2R})$ such that 
\begin{equation}
\label{varphi}
0\le \eta\le 1,\qquad \left.\eta\right|_{K_{R}}\equiv 1,\qquad  |D\eta|\le C_{N}/R
\end{equation}  
gives
\begin{equation}
\label{par5}
\begin{split}
\iint_{Q_{R, T}}|D(u-k_{j})_{-}|^{2}\, dx\, dt&\le \bar C\int_{K_{2R}} (u(x, 0)-k_{j})_{-}^{2}\, dx+\frac{\bar C}{R^{2}}\iint_{Q_{2R, T}}(u-k_{j})_{-}^{2}\, dx\, dt\\
&\le \bar C \, k_{j}^{2}\, R^{N} (1+T/R^{2}).
\end{split}
\end{equation}
For any $t\in [0, T]$, apply the De Giorgi - Poincar\'e inequality and \eqref{ptpar} to obtain
\[
\begin{split}
(k_{j}-k_{j+1})|K_{R}\cap \{u(\cdot, t)\le k_{j+1}\}|&\le \frac{C_{N}\, R^{N+1}}{|K_{R}\cap \{u(\cdot, t)<k_{j}\}|}\int_{K_{R}\cap \{k_{j+1}\le u(\cdot, t)\}}|D(u(x, t)-k_{j})_{-}|\, dx\\
&\underset{\eqref{ptpar}}{\le} \frac{C_{N}\, R}{\mu}\int_{K_{R}\cap \{k_{j+1}\le u(\cdot, t)\}}|D(u(x, t)-k_{j})_{-}|\, dx
\end{split}
\]
Integrate the latter over $[0, T]$, divide by $|Q_{R, T}|$ and use H\"older's inequality to obtain
\[
\begin{split}
\frac{k_{j}}{2}&P(Q_{R, T}; \, u\le k_{j+1})\le \frac{C_{N}\, R}{\mu\, |Q_{R, T}|}\iint_{Q_{R, T}\cap \{k_{j+1}\le u\}}|D(u-k_{j})_{-}|\, dx\, dt\\
&\le \frac{C_{N}\, R}{\mu}\left(\frac{1}{|Q_{R, T}|}\iint_{Q_{R, T}\cap \{k_{j+1}\le u\}}|D(u-k_{j})_{-}|^{2}\, dx\, dt\right)^{\frac{1}{2}}\left(\frac{|Q_{R, T}\cap \{k_{j+1}\le u\le k_{j}\}|}{|Q_{R, T}|}\right)^{\frac{1}{2}}\\
&\underset{\eqref{par5}}{\le} \frac{\bar C\, R}{\mu}\, k_{j}\,\frac{1}{T^{\frac{1}{2}}} \Big(1+\frac{T}{R^{2}}\Big)^{\frac{1}{2}}\left(P(Q_{R, T}; \, u\le k_{j})-P(Q_{R, T};\, u\le k_{j+1})\right)^{\frac{1}{2}},
\end{split}
\]
The latter reads
\[
\left(P(Q_{R, T};\, u\le k_{j+1})\right)^{2}\le C(\mu) \, \left(1+R^{2}/T\right)\left(P(Q_{R, T};\, u\le k_{j})-P(Q_{R, T};\,  u\le k_{j+1})\right),
\]
which, being the right hand side telescopic, can be summed over $j\le n-1$ to get the claim:
\[
n\left(P(Q_{R, T};\, u\le k_{n})\right)^{2}\le \sum_{j=0}^{n-1}\left(P(Q_{R, T};\, u\le k_{n})\right)^{2}\le  C(\mu)\, (1+ R^{2}/T). 
\]
\end{proof}

\begin{lemma}[Critical mass]\label{critlemmapar}
For any $\theta>0$ there exists $\nu(\theta)>0$ such that any supersolution $u\ge 0$ on $Q_{R, \theta R^{2}}$ fulfills 
\begin{equation}
\label{par3}
P\left(Q_{R, \theta R^{2}};\,  u\le k \right)\le\nu(\theta) \quad \Rightarrow\quad  
u\ge  k/2\quad \text{on $K_{R/2}\times [\theta\, R^{2}/8, \theta\, R^{2}]$}.
\end{equation}
\end{lemma}

\begin{proof}
Use homogeneity and scaling to reduce to $R=1$, $k=1$. Define for $n\ge 1$
\[
r_{n}=\frac{1}{2}+\frac{1}{2^{n}}, \qquad k_{n}=\frac{1}{2}+\frac{1}{2^{n}},\qquad \theta_{n}=\frac{\theta}{8}-\frac{\theta}{2^{n+3}},
\]
 let $K_{n}=K_{r_{n}}$, $Q_{n}=K_{n}\times [\theta_{n}, \theta]$ and choose $\eta_{n}\in C^{\infty}([\theta_{n}, \theta]; C^{\infty}_{c}(K_{n}))$ s.\,t. 
\begin{equation}
\label{etan}
\eta_{n}(\cdot, \theta_{n})\equiv 0,\qquad 0\le \eta_{n}\le 1,\qquad \left.\eta_{n}\right|_{Q_{n+1}}\equiv 1,\quad |D\eta_{n}|\le C_{N} \, 2^{n},\quad |(\eta_{n})_{t}|\le C_{N}\, 2^{n}/\theta.
\end{equation}
Inserting into the energy inequality \eqref{ein} and noting that $k_{n}\le k$, we get
\[
\begin{split}
&\sup_{t\in [\theta_{n+1}, \theta]}\int_{K_{n+1}}(u(x, t)-k_{n})_{-}^{2}\, dx+\iint_{Q_{n}}|D(\eta_{n}(u-k_{n})_{-})|^{2}\, dx\, dt\\
&\quad\le \bar C\, 2^{2n}(1+\theta^{-2})\iint_{Q_{n}}(u-k_{n})_{-}^{2}\, dx\, dt\le \bar C \, 2^{2n}(1+\theta^{-2}) k_{n}^{2} |Q_{n}\cap \{ u\le k_{n}\}|.
\end{split}
\]
By the parabolic Sobolev embedding
\[
\begin{split}
&\iint_{Q_{n+2}}(u-k_{n+1})_{-}^{2\frac{N+2}{N}}\, dx\, dt\le \iint_{Q_{n+1}}\left((u-k_{n+1})_{-}\eta_{n+1}\right)^{2\frac{N+2}{N}}\, dx\, dt\\
&\quad  \le C_{N}\iint_{Q_{n+1}}|D(\eta_{n+1}(u-k_{n+1})_{-}^{2})|^{2}\, dx\, dt\left(\sup_{t\in [\theta_{n+1}, \theta]}\int_{K_{n+1}}\eta_{n+1}^{2}(x, t)(u(x, t)-k_{n+1})_{-}^{2}\, dx\right)^{\frac{2}{N}}\\
&\quad \le \bar C\, 2^{2n\frac{N+2}{N}} (1+\theta^{-2})^{\frac{N+2}{N}} h_{n}^{2\frac{N+2}{N}}\, |Q_{n}\cap \{u\le k_{n}\}|^{\frac{N+2}{N}},
\end{split}
\]
while, being $(u-k_{n+1})_{-}\ge k_{n+1}-k_{n+2}=k_{n}/4$ when $u\le k_{n+2}$, 
\[
\iint_{Q_{n+1}}(u-k_{n+1})_{-}^{2\frac{N+2}{N}}\, dx\, dt\ge  (k_{n}/4)^{2\frac{N+2}{N}}| Q_{n+2}\cap \{ u\le k_{n+2}\}|.
\]
Chaining these latter two estimates and simplifying $k_{n}$ gives  the iterative inequality
\[
|Q_{n+2}\cap \{u\le k_{n+2}\}|\le \bar C\, b_{N}^{n} (1+\theta^{-2})^{1+\frac{2}{N}}|Q_{n}\cap \{u\le k_{n}\}|^{1+\frac{2}{N}}
\]
and \eqref{iteration} for $X_{n}:=|Q_{2n}\cap \{u\le k_{2n}\}|$ gives the claim.
\end{proof}

\begin{lemma}[Measure-to-point estimate]
Let $u\ge 0$ be a supersolution in $Q_{R, R^{2}}$. For all  $\mu\in \ ]0, 1[$ there are $c(\mu)>0$ and $\theta(\mu)\in \ ]0, 1[$  such that
\begin{equation}
\label{par8}
P_{0}(K_{\rho};\, u\ge h)\ge \mu\quad \Rightarrow\quad u\ge c(\mu)\, h\quad \text{in $K_{\rho/2}\times [\theta(\mu)\, \rho^{2}/8, \theta(\mu) \,\rho^{2}]$}.
\end{equation}
\end{lemma}

\begin{proof}
By homogeneity we can let $h=1$. Let  $\theta(\cdot), k(\cdot)$ be given in Lemma \ref{lemmapar}, so that $P_{t}(K_{\rho};\, u\ge k)\ge \mu/2$ for $t\in [0, \theta\,\rho^{2}]$, $\theta=\theta(\mu)$ and $k=k(\mu)$. Apply Lemma \ref{slemmapar}, chosing $n=n(\mu)$ such that 
\[
\beta(\mu/2)\, (1+\theta(\mu)^{-1})^{1/2}\, n^{-1/2}\le \nu(\theta),
\]
($\nu(\cdot )$ given in \eqref{par3}), to get $P(Q_{\rho, \theta\rho^{2}};\, u\le k\, 2^{-n})\le \nu(\theta)$. Then \eqref{par3} proves \eqref{par8}. 
\end{proof}

\begin{theorem}[H\"older regularity]
Any locally bounded solution of \eqref{parabo} is locally H\"older continuous, with H\"older exponent depending only on $N$, $C_{0}$ and $C_{1}$. 
\end{theorem}

\begin{proof}
By translation and scaling it suffices to prove an oscillation decay on the cubes $Q_{n}=K_{2^{-n}}\times [-\bar\theta\, 2^{-2n}, 0]$ with $\bar\theta=\theta(1/2)$ given in \eqref{par8}. Suppose ${\rm osc}(u, Q_{0})=1$. Then, one of 
\[
 P_{-\bar\theta}\big(K_{1};\, \sup_{Q_{0}}u-u\ge 1/2\big)\ge 1/2,\qquad   \text{or}\qquad P_{-\bar\theta}\big(K_{1};\, u-\inf_{Q_{0}}u\ge  1/2\big)\ge  1/2
  \]
holds. If it is the first one, apply \eqref{par8} to $\sup_{Q_{0}}u-u\ge 0$ translated in time to get $\sup_{Q_{0}}u-u\ge m(1/2)/2=:\bar m$ in $Q_{1}$, i.\,e. $\sup_{Q_{1}}u\le \sup_{Q_{0}}u-\bar m$. Therefore 
\[
{\rm osc}(u, Q_{1})\le \sup_{Q_{1}}u-\inf_{Q_{0}}u\le \sup_{Q_{0}}u-\inf_{Q_{0}}u=1-\bar m.
\]
The same holds in the other case and by homogeneity we have ${\rm osc}(u; Q_{1})\le (1-\bar m){\rm osc}(u, Q_{0})$. By scaling and induction, ${\rm osc}(u, Q_{n})\le {\rm osc}(u, Q_{0})(1-\bar m)^{n}$. Finally, for $(x, t)\in Q_{1}$ let $n\ge 1$ s.\,t.
 \[
 2^{-n-1}\le \max\{|x|, (|t|/\bar\theta)^{1/2}\}\le 2^{n}, 
 \]
 so that we have $(x, t)\in Q_{n}$ and, for $\bar \alpha=-\log_{2}(1-\bar m)$,
\[
|u(x, t)-u(0, 0)|\le {\rm osc}(u, Q_{n})\le \frac{{\rm osc}(u, Q_{0})}{1-\bar m}(1-\bar m)^{n+1}\le \frac{{\rm osc}(u, Q_{0})}{1-\bar m}\max\left\{|x|, (|t|/\bar\theta)^{1/2}\right\}^{\bar \alpha}.
\]

\end{proof}

\begin{figure}
\centering
\begin{tikzpicture}[scale=1.2]

\shade[shading=axis, shading angle=180] (-0.6, 0.5) to [out=130, in=-80] (-2, 3.8) -- (2, 3.8) to [out=-100, in=50] (0.6,0.5) -- (-0.6, 0.5);
\draw (0.6, 0.5) to [out=50, in=-100] (2, 3.8);
\draw[->] (0, -0.5) -- (0, 4) node[above left]{$t$};
\draw[->] (-3, 0) -- (3, 0) node[below right]{$x$};
\draw (-0.6, 0.5) to [out=130, in=-80] (-2, 3.8);
\draw (0.6, 0.5) to [out=50, in=-100] (2, 3.8);
\draw (-0.6, 0.5) -- (0.6,0.5);
\draw[very thick, dashed] (-0.5, 0) -- (0.5, 0);
\draw (0, 0.5) node[below right]{$\bar\gamma$};
\end{tikzpicture}
\caption{The parabolic expansion of positivity. If at time $t=0$ $u\ge 1$ on the dotted part of given measure, after a waiting time $\bar\gamma$, $u$ is pointwise bounded from below in the paraboloid by a large negative power of $t$. }
\label{EPfig}
\end{figure}
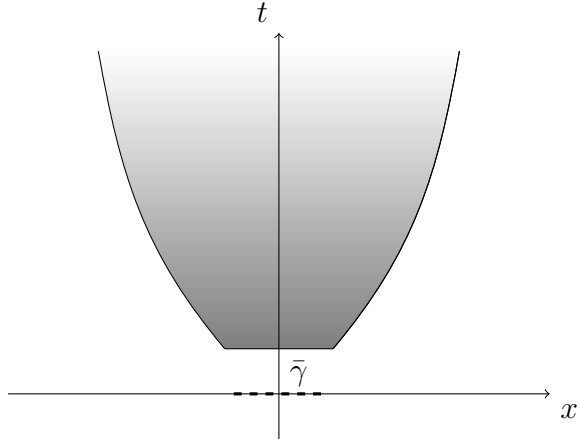

\begin{lemma}[Expansion of positivity, see figure \ref{EPfig}]\label{epospar}
Let $u\ge 0$ be a supersolution in $Q_{R, R^{2}}$. There exists $\bar \lambda>1, \bar \gamma\in \ ]0, 1/4[$ and, for any $\mu>0$ a constant $c(\mu)>0$, such that 
\[
P_{0}(K_{r};\, u\ge 1)\ge \mu\quad \Rightarrow\quad \inf_{K_{\rho}}u(\cdot, \bar\gamma\, \rho^{2})\ge c(\mu)\, (r/\rho)^{\bar \lambda}\qquad\forall  \rho\in [r, R/8].
\]
\end{lemma}

\begin{proof}
First expand \eqref{par8} in space observing that  $P_{0}(K_{\rho};\, u\ge h)\ge \mu$ implies $P_{0}(K_{4\rho};\, u\ge h)\ge \mu\, 4^{-N}$, so that by changing the constants $\theta(\mu)$ and $c(\mu)$, we get
\begin{equation}
\label{par9}
P_{0}(K_{\rho};\, u\ge h)\ge \mu\quad \Rightarrow\quad u\ge c(\mu) \, h\quad \text{in $K_{2\rho}\times [\theta(\mu)\, \rho^{2}/8, \theta(\mu)\, \rho^{2}]$}.
\end{equation}
To prove the lemma, let $\bar c=c(1)$, $\bar\theta=\theta(1)$, $\rho_{n}=2^{n}\, r$ and define recursively the sequences
\[
t_{0}=\theta(\mu)r^{2}/8, \quad t_{n+1}=t_{n}+\bar\theta\, \rho_{n+1}^{2}/8,\qquad s_{0}=\theta(\mu)\, r^{2},\quad s_{n}=t_{n-1}+\bar\theta\, \rho_{n}^{2},\quad n\ge 1.
\]
Letting furthermore $Q_{n}=K_{\rho_{n+1}}\times [t_{n}, s_{n}]$, apply recursively \eqref{par9} as
\[
 P_{0}(K_{r};\, u\ge 1)\ge \mu\quad \Rightarrow\quad P(Q_{0};\, u\ge c(\mu))=1\quad \Rightarrow \quad P(Q_{1};\, u\ge c(\mu)\, \bar c)=1\quad \dots
 \]
 to get by induction $P(Q_{n};\, u\ge c(\mu)\, \bar c^{n})=1$. It is easily checked that $s_{n}> t_{n+1}$ for $n\ge 1$, hence
  \[
 \inf_{K_{\rho_{n+1}}}u(\cdot, t)\ge c(\mu)\, \bar c^{n}\qquad t_{n}\le t\le t_{n+1},\qquad n\ge 1.
 \]
 Notice that we can suppose that $\theta(\mu)\le \bar\theta\le 1/16$, so that it holds $\bar\theta\, \rho^{2}_{n-1}\le t_{n}\le  \bar\theta\, \rho_{n}^{2}$ for $n\ge 1$ and a monotonicity argument gives
 \[
  \inf_{K_{\rho_{n}}}u(\cdot, t)\ge c(\mu)\, \bar c^{n+2}\qquad  \bar\theta\, \rho_{n}^{2}\le t\le \bar\theta\, \rho_{n+1}^{2},\quad n\ge 0.
  \]
 For $\rho\ge r$, let $n$ be such that $\rho_{n}\le \rho\le \rho_{n+1}$, hence $\bar\theta\, \rho_{n+1}^{2}\le 4\, \bar\theta\, \rho^{2}\le \bar\theta\, \rho_{n+2}^{2}$. Then the lemma is proved for $\bar\lambda=-\log_{2}\bar c$, and $\bar\gamma=4\,\bar\theta$, since
\[
\inf_{K_{\rho}}u(\cdot, 4\, \bar\theta\, \rho^{2})\ge \inf_{K_{\rho_{n+1}}}u(\cdot, 4\, \bar\theta\, \rho^{2})\ge c(\mu)\, \bar c^{n+3}\ge c(\mu)\, \bar c^{3}\,  (r/\rho_{n})^{\bar\lambda}\ge c(\mu)\, \bar c^{3}\,  (r/\rho)^{\bar\lambda}.
\]
\end{proof}

\begin{theorem}[Harnack inequality]\label{HIhom}
Let $u\ge 0$ be a locally bounded solution of \eqref{parabo} in $K_{2R}\times [-(2R)^{2}, (2R)^{2}]$. There exists $\bar C$  such that $u(0, 0)\le  \bar C\, \inf_{K_{R}}u(\cdot,  R^{2})$.
\end{theorem}

\begin{proof}
By homogeneity, scaling and a Harnack chain argument it suffices to prove 
\begin{equation}
\label{claimpar}
u(0, 0)=1\quad \Rightarrow \quad \inf_{K_{1}}u(\cdot, 1)\ge \bar c>0
\end{equation}
for any solution $u$ of \eqref{parabo}, nonnegative in $K_{\bar L}\times [-\bar L^{2}, \bar L^{2}]$ for some $\bar L$ to be chosen.
Let 
\[
\psi(\rho)=(1-\rho)^{\bar\lambda}\sup_{Q_{\rho}^{-}}u,\qquad Q^{-}_{\rho}:=K_{\rho}\times [-\rho^{2}, 0],\quad  \rho\in [0, 1]
\]
where  $\bar\lambda$ is given in Lemma \ref{epospar}. By continuity, we can choose $\rho_{0}\in [0,1]$, $(x_{0}, t_{0})\in Q^{-}_{\rho_{0}}$ s.\,t.
\[
\max_{[0,1]}\psi(\rho)=(1-\rho_{0})^{\bar\lambda}u_{0}\qquad u_{0}:=u(x_{0}, t_{0}).
\]
For $\xi\in \ ]0, 1[$ to be determined let $r=\xi\, (1-\rho_{0})$, so that, being $\psi(0)=u(0,0)=1$,
\begin{equation}
\label{etazero}
u_{0}\, r^{\bar \lambda}=\xi^{\bar\lambda}\, u_{0}(1-\rho_{0})^{\bar\lambda}=\xi^{\bar\lambda}\, \psi(\rho_{0})\ge \xi^{\bar\lambda}\, \psi(0)= \xi^{\bar\lambda}.
\end{equation}
If $\widetilde Q_{r}=K_{r}(x_{0})\times [t_{0}-r^{2}, t_{0}]$, it holds $\widetilde Q_{r}\subseteq Q^{-}_{\rho_{0}+r}$ hence, being $\rho_{0}$ maximum for $\psi$,
\begin{equation}
\label{jk45}
\sup_{\widetilde Q_{r}} u\le \sup_{Q^{-}_{\rho_{0}+r}} u=\frac{\psi(\rho_{0}+r)}{(1-\rho_{0}-r)^{\bar\lambda}}\le \frac{(1-\rho_{0})^{\bar\lambda}}{(1-\rho_{0}-r)^{\bar\lambda}}\, u_{0}=(1-\xi)^{-\bar\lambda}\, u_{0}.
\end{equation}
Choose $\bar \xi$ as per $(1-\bar \xi)^{-\bar\lambda}\le 2$, so that $u\le 2\, u_{0}$ in $\widetilde{Q}_{r}$, and let $\bar\theta=\theta(1/2)\in \ ]0, 1[$ be given in \eqref{par8}. Since $K_{r}(x_{0})\times [t_{0}- \bar\theta\, r^{2}, t_{0}]\subseteq \widetilde{Q}_{r}$, the previous proof shows that for all $\rho\le r/2$
\[
{\rm osc}(u(\cdot, t_{0}), K_{\rho}(x_{0}))\le \bar C\, \sup_{\widetilde Q_{r}} u\, (r/\rho)^{\bar\alpha}\le 2\, \bar C\, u_{0}\, (r/\rho)^{\bar\alpha},
\]
and choosing $\rho=\bar\eta\, r$ with $\bar C \, \bar \eta^{\bar\alpha}\le 1/4$ gives ${\rm osc}(u(\cdot, t_{0}), K_{\bar\eta r}(x_{0}))\le u_{0}/2$. The latter ensures $P_{t_{0}}(K_{r}(x_{0});\, u\ge u_{0}/2)\ge \bar\eta^{N}$ and the expansion of positivity Lemma \ref{epospar} for $2\, u/u_{0}$ implies
\begin{equation}
\label{lapo}
\inf_{K_{\rho}(x_{0})}u(\cdot, t_{0}+\bar\gamma\, \rho^{2})\ge c(\bar\eta^{N})\, \frac{u_{0}}{2}  \frac{r^{\bar\lambda}}{\rho^{\bar\lambda}}\underset{\eqref{etazero}}{\ge}\frac{c(\bar\eta^{N})\, \bar\xi^{\bar\lambda}}{2\, \rho^{\bar\lambda}},\qquad r\le \rho\le \bar L/8.
\end{equation}
Solve $t_{0}+\bar\gamma\, \rho=1$ in $\rho$: from $\bar\gamma\le 1/4$ and $t_{0}\in [-1, 0]$ we infer $2\le \rho\le 2/\bar\gamma$. Therefore $K_{\rho}(x_{0})\supseteq K_{1}$ and we can let $\bar L/8:=2/\bar\gamma$ in \eqref{lapo}, giving \eqref{claimpar} and completing the proof.
\end{proof}

\subsection{Inhomogeneous parabolic equations}\ \\

\vskip-8pt

In the last subsection, we heavily took advantage of the homogeneous structure of the equation. 
The situation is quite different for inhomogeneous equations whose model is 
\begin{equation}
\label{pt10}
u_{t}={\rm div}A(x, u, Du),\qquad 
\begin{cases}
A(x, s, z)\cdot z\ge C_{0}|z|^{p}\\
|A(x, s, z)|\le C_{1}|z|^{p-1}
\end{cases}
\end{equation}
for $p\ne 2$, as it is no longer true that $\lambda u$ is a solution of a similar equation for $\lambda\ne 1$. The traslation invariance still holds, and the scale invariance says that if $u$ solves \eqref{pt10} then $u_{\lambda}(x, t)=u(\lambda x, \lambda^{p} t)$ is a solution (in the usual sense that there exists an $A$ obeying the growth condition such that $u$ solves the corresponding equation). More generally, given $R, T>0$ and a (sub-)\,super-solution of \eqref{pt10},
\begin{equation}
\label{scalingin}
u_{R, T}(x, t)=R^{\frac{p}{2-p}}T^{\frac{1}{p-2}}u(R\, x, T\, t)
\end{equation}
is still  a (sub-)\,super-solution (in the structural sense) an equation of the kind \eqref{pt10}. This shows that statements  for $\lambda u$ can be derived from those for $u$ by scaling the space-time variables conveniently (actually, with one degree of freedom).

It is worth noting that, in the inhomogeneous setting, it is not known wether the energy inequality alone suffices to prove the Harnack inequality. In our proof, we will indeed use a clever change of variable introduced in \cite{NEW}, which crucially relies on the equation. Moreover, as extensively discussed in the previous chapter, the degenerate ($p>2$) and singular ($p<2$) cases require different treatments. We thus first derive some common tools in this subsection, and discuss in details the two families of equations in the following ones. The notation will be the same as in the previous one, with the additional dependance on $p$ omitted in constants.

\begin{proposition}
\quad 
\begin{description}
\item[1)  Parabolic Sobolev Embedding]
If $u\in L^{p}(0, T; W^{1, p}_{0}(\Omega))$, and $p^{*}=p(1+2/N)$, then
\[
\int_{0}^{T}\int_{\Omega} |u|^{p^{*}}\, dx\, dt\leq C(N)\left(\sup_{t\in [0, T]}\int_{\Omega}u^{2}(x, t)\, dx\right)^{\frac{p}{N}}\int_{0}^{T}\int_{\Omega}|Du|^{p}\, dx\, dt.
\]
\item[2)  Energy inequality]\cite[Prop. III.2.1]{HR}
Let $v$ be a supersolution to \eqref{pt10} in $Q_{T}$ under condition \eqref{pgr}. There exists $C=C(C_{0}, C_{1})>0$ s.\,t. for any $k\ge 0$ and $\eta\in C^{\infty}(0, T; C^{\infty}_{c}(K))$, $0\le \eta\le 1$ it holds
\begin{equation}
\label{ein}
\begin{split}
\sup_{t\in [0, T]}&\int_{K} (v(x, t)-k)_{-}^{2}\eta^{p} \, dx+\frac{1}{C}\iint_{Q_{T}} |D(\eta(v-k)_{-})|^{p}\, dx\, dt\\
&\le \int_{K} (v(x, 0)-k)_{-}^{2}\eta^{p} \, dx+C\iint_{Q_{T}} (v-k)_{-}^{p}|\nabla\eta|^{p}\, dx\, dt+C\iint_{Q_{T}} (v-k)_{-}^{2}|\eta_{t}|\, dx\, dt.
\end{split}
\end{equation}
\end{description}
\end{proposition}

We start by sketching the proof of the relevant critical mass lemma.

\begin{lemma}[Critical mass]\label{critlemma}
Let $v\ge 0$ be a supersolution of \eqref{pt10} on $Q_{R, T}$ for $p\ne 2$ and let $h\ge 0$. There exists $\nu>0$ s.t.
\begin{equation}
\label{pt3}
P\left(Q_{R, T};\,  v\le h \right)\le\nu(h\, R^{\frac{p}{2-p}}T^{\frac{1}{p-2}})\quad \Rightarrow\quad  
v\ge  h/2\quad \text{on $K_{R/2}\times [T/2, T]$}.
\end{equation}
\end{lemma}

\begin{proof}
Consider the supersolution $v_{R, T}(x, t)=R^{\frac{p}{2-p}}T^{\frac{1}{p-2}}v(R\, x, T\, t)$: as \eqref{pt3} is invariant by this transformation, it suffices to prove it for $R=T=1$.
Define for $n\ge 1$
\[
\begin{split}
&r_{n}=\frac{1}{2}+\frac{1}{2^{n}}, \qquad h_{n}=\frac{h}{2}+\frac{h}{2^{n}},\qquad t_{n}=\frac{1}{2}-\frac{1}{2^{n+1}}\\
&K_{n}=K_{r_{n}},\qquad  Q_{n}=K_{n}\times [t_{n}, 1],\qquad  A_{n}=Q_{n}\cap \{v\le h_{n}\}.
\end{split}
\]
Fix $\eta_{n}$ as per \eqref{etan} with $\theta=4$. Inserting into \eqref{ein} and noting that $h_{n}\le h$, we get
\[
\begin{split}
&\sup_{t\in [t_{n+1}, 1]}\int_{K_{n+1}}(v(x, t)-h_{n})_{-}^{2}\, dx+\iint_{Q_{n}}|D(\eta_{n}(v-h_{n})_{-})|^{p}\, dx\, dt\\
&\quad\le \bar C\, 2^{np}\iint_{Q_{n}}(v-h_{n})_{-}^{p}\, dx\, dt+\bar C\, 2^{n}\iint_{Q_{n}}(v-h_{n})_{-}^{2}\, dx\, dt\le \bar C (h^{p}+h^{2})\, 2^{np}\, |A_{n}|.
\end{split}
\]
 Use $h_{n+1}-h_{n+2}=h/2^{n+3}$,  Tchebicev and the parabolic Sobolev embedding to get
 \[
 \begin{split}
 \frac{h^{p^{*}}}{2^{p^{*}(n+3)}}| A_{n+2}|&\le \iint_{A_{n+2}}(v-h_{n+1})_{-}^{p^{*}}\, dx\, dt\le \iint_{Q_{n+1}}\left((v-h_{n+1})_{-}\eta_{n+1}\right)^{p^{*}}\, dx\, dt\\
 &  \le \bar C\iint_{Q_{n+1}}|D(\eta_{n+1}(v-h_{n+1})_{-}^{2})|^{p}\, dx\, dt\left(\sup_{t\in [t_{n+1}, 1]}\int_{K_{n+1}}(v(x, t)-h_{n+1})_{-}^{2}\, dx\right)^{\frac{p}{N}}\\
 & \le \bar C\, \bar b^{n}\, (h^{p}+h^{2})^{1+\frac{p}{N}}|A_{n+1}||A_{n}|^{\frac{p}{N}}\le \bar C\, \bar b^{n}\, (h^{p}+h^{2})^{1+\frac{p}{N}}|A_{n}|^{1+\frac{p}{N}}. 
 \end{split}
 \]
 This amounts to $ |A_{n+2}|\le \bar b^{n}\, \bar C(h)|A_{n}|^{1+\frac{p}{N}}$ and  \eqref{iteration} for $X_{n}=|A_{2n}|$ gives the conclusion.
 \end{proof}

\begin{lemma}\label{sdg}
Let $v\ge 0$ be a supersolution in $Q_{R,T}$ of \eqref{pt10}.  There exists $\bar \sigma$ s.\,t. 
\[
\inf_{K_{R}}v(x, 0)\ge h\quad \Rightarrow \quad  v\ge h/2\ \text{ on $K_{R/2}\times \big[0, \min\{\bar \sigma\, R^{p}\, h^{2-p}, T\}\big]$}.
\]
\end{lemma}

\begin{proof}
Consider the supersolution $\tilde v(x, t)=R^{\frac{p}{2-p}}v(Rx, t)$ to reduce to the case $R=1$, $\tilde{v}(\cdot, 0)\ge \tilde{h}=h\, R^{\frac{p}{2-p}}$ on $K_{1}$. Proceed as in the previous proof with $t_{n}\equiv 0$, $\eta_{n}$ independent of $t$ and $Q_{n}=K_{r_{n}}\times [0, T]$. Since $\tilde{v}(\cdot, 0)\ge \tilde{h}_{n}$ and $(\eta_{n})_{t}\equiv 0$, the first and third term on the right of \eqref{ein} vanish, giving
\[
\sup_{t\in [0, T]}\int_{K_{n+1}}(\tilde v(x, t)-\tilde{h}_{n})_{-}^{2}\, dx+\iint_{Q_{n}}|D(\eta_{n}(\tilde v-\tilde h_{n})_{-})|^{p}\, dx\, dt\le \bar C\, 2^{np} h^{p}\, |A_{n}|.
\]
where $A_{n}=Q_{n}\cap \{\tilde v\le \tilde h_{n}\}$. Continuing exactly as before, we get the iterative inequality
\[
\tilde{h}^{p^{*}}/2^{p^{*}(n+3)}|A_{n+2}|\le \bar C\, \bar b^{n}\,  h^{p\frac{N+p}{N}}|A_{n}|^{1+\frac{p}{N}}
\]
which, recalling that $p^{*}=p(N+2)/N$ and enlarging $\bar b$, reads
\[
|A_{n+2}|\le \bar C\, \bar b^{n}\, \tilde{h}^{\frac{p}{N}(p-2)}|A_{n}|^{1+\frac{p}{N}}.
\]
Since $|A_{0}|\le T$,  \eqref{iteration} ensures the existence of $\bar \sigma$ such that 
\[
T\le \bar\sigma \tilde{h}^{2-p}\quad \Rightarrow \quad  \lim_{n}|A_{2n}|=0\quad \Rightarrow
\quad \inf_{Q_{1/2, T }} \tilde{v}\ge \tilde{h}/2\quad \Leftrightarrow\quad \inf_{Q_{R/2, T}}v\ge h/2.
\] 

\end{proof}

We conclude this section with a useful tool to prove H\"older continuity of solutions. 
\begin{lemma}\label{lemmaosc}
Suppose there exist $\bar T>0$ and $\bar m, \bar \theta\in \ ]0, 1[$, depending only on the data, s.\,t. any solution $u$ of \eqref{pt10} with $p\ne 2$ in $Q_{2, \bar T}$  fulfills
\begin{equation}
\label{pt11}
P_{0}(K_{1};\, u\ge 1/2)\ge 1/2\quad \Rightarrow \quad u\ge \bar m\ \text{on $K_{1/4}\times [(1-\bar \theta)\bar T, \bar T]$}.
\end{equation}
There exists $\bar C, \bar\alpha$ depending on $\bar m$ and $\bar\theta$ such that any solution with $\|u\|_{L^{\infty}(Q_{2, \bar T})}\le 1$ obeys
\begin{equation}
\label{oscest1}
{\rm osc}(u, K_{r}\times [\bar T(1-r^{p}), \bar T])\le \bar C\, r^{\bar\alpha}, \qquad 0\le r\le 1.
\end{equation}
\end{lemma}
 
 \begin{proof}
Fix  $\delta\in \ ]0, 1/4]$, $\theta\in \ ]0, \bar \theta]$  so that $\theta^{\frac{1}{2-p}}\delta^{\frac{p}{p-2}}=\gamma:=(1+\bar m)^{-1}<1$. We claim by induction  
\begin{equation}
\label{pt5}
 {\rm osc}(u, Q_{n})\le (1+\bar m)\, \gamma^{n},\quad \text{for all $n\ge 0$, where}\quad Q_{n}:=K_{\delta^{n}}\times [\bar T(1- \theta^{n}), \bar T].
\end{equation}
Since $\|u\|_{L^{\infty}(Q_{2, \bar T})}\le 1$, \eqref{pt5} holds true for $n=0$, so  suppose by contradiction that 
\begin{equation}
\label{pt8}
{\rm osc}(u, Q_{n})\le(1+\bar m)\, \gamma^{n}\quad \text{\&}\quad {\rm osc}(u, Q_{n+1})>(1+\bar m)\, \gamma^{n+1}
\end{equation}
for some $n\ge 1$. Being ${\rm osc}(u, Q_{n})\ge {\rm osc}(u, Q_{n+1})$ we infer 
\begin{equation}
\label{pt7}
{\rm osc}(u, Q_{n})>(1+\bar m)\, \gamma^{n+1}.
\end{equation}
By scaling and traslation invariance, the function $v(x, t)=\gamma^{-n}u\big(\delta^{n} x, (t-\bar T)\theta^{n}+\bar T\big)$
solves in $Q_{0}$ an equation of the type \eqref{pt10} and, recalling that $\gamma= 1/(\bar m+1)$, we have
\[
 {\rm osc}(v, Q_{0})=\gamma^{-n}{\rm osc}(u, Q_{n})\underset{\eqref{pt7}}{>}(1+\bar m) \, \gamma=1.
\]
We infer from the latter that the assumption in \eqref{pt11} holds for at least one of the nonnegative supersolutions  $v_{+}:= v-\inf_{Q_{0}} v$ or $v_{-}=\sup_{Q_{0}}v-v$: indeed, for example, $P_{0}(K_{1};\, v_{+}\ge 1/2)<1/2$ is equivalent to $P_{0}(K_{1};\, v_{+}<1/2)\ge 1/2$ and then $ {\rm osc}(v, Q_{0})\ge 1$ ensures
\[
P_{0}(K_{1};\, v_{-}\ge 1/2)\ge P_{0}(K_{1};\, v_{-}>{\rm osc}(v, Q_{0})-1/2)=P_{0}(K_{1};\, v_{+}< 1/2)\ge 1/2.
\]
Suppose, without loss of generality, that  $P_{0}(K_{1};\, v_{+}\ge 1/2)\ge 1/2$: then, since $\theta\le\bar\theta$ and $\delta\le 1/4$, \eqref{pt11} implies $\inf_{Q_{1}}v_{+}\ge \bar m$ and thus 
\[
{\rm osc} (v, Q_{1})={\rm osc} (v_{+}, Q_{1})\le {\rm osc} (v_{+}, Q_{0}) - \bar m={\rm osc} (v, Q_{0}) - \bar m.
\]
Scaling back to $u$ and using the relations in  \eqref{pt8}, we obtained the contradiction
\[
(1+\bar m)\, \gamma< \gamma^{-n}{\rm osc}(u, Q_{n+1})={\rm osc} (v, Q_{1})\le {\rm osc} (v, Q_{0}) - \bar m= \gamma^{-n}{\rm osc}(u, Q_{n})-\bar m\le 1+\bar m -\bar m.
\]
To prove \eqref{oscest1} let $\eta=\delta\, \min\{1, \gamma^{\frac{p-2}{p}}\}$ and suppose $\eta^{n+1}\le r\le \eta^{n}$ for some $n$. Then, using $\theta^{\frac{1}{2-p}}\, \delta^{\frac{p}{p-2}}=\gamma$, we infer $K_{r}\times [\bar T(1-r^{p})]\subseteq Q_{n}$. Letting $\bar\alpha=\log_{\eta}\gamma$ and using \eqref{pt5} we have 
\[
{\rm osc}(u, K_{r}\times [\bar T(1-r^{p}), \bar T])\le \gamma^{n}= \gamma^{-1}\, (\eta^{n+1})^{\bar \alpha}\le \gamma^{-1}\, r^{\bar\alpha}.
\]
\end{proof}
 
\subsection{Degenerate parabolic equations}\ \\

This subsection is devoted to the case $p>2$ of \eqref{pt10}. Compared to the homogeneous case $p=2$, the most delicate part is the proof of the measure-to-point estimate, Lemma \ref{EP0} below.

\begin{lemma}\label{lemmapt}
Assume that $u\ge 0$ is a  supersolution in $Q_{1, T}$ of \eqref{pt10} with $p> 2$. For any $\mu>0$ there exists $k(\mu)\in \ ]0, 1[$  such that
\begin{equation}
\label{pt0}
P_{0}(K_{1};\, u\ge 1)\ge \mu\quad \Rightarrow\quad P_{t}\left(K_{1};\, u\ge k(\mu)/(t+1)^{\frac{1}{p-2}}\right)>\mu/2 \qquad \text{for all $t\in [0, T]$}
\end{equation}
\end{lemma}

\begin{proof}
For any $k, \delta\in \ ]0, 1[$, we employ  \eqref{ein} with $\eta$ as in \eqref{varphi2}, obtaining for $t\in [0, T]$ 
\[
\begin{split}
\int_{K_{\delta}} (u(x, t)-k)^{2}_{-}\, dx&\le \int_{K_{1}} (u(x, 0)-k)^{2}_{-}\, dx+\frac{\bar C}{(1-\delta)^{p}}\iint_{Q_{1, t}}(u-k)^{p}_{-}\, dx\, dt\\
&\le k^{2}(1-\mu)+\frac{\bar C\, k^{p}\, t}{(1-\delta)^{p}}.
\end{split}
\]
For $\eps\in \ ]0, 1[$ we have
\[
\int_{K_{\delta}} (u(x, t)-k)^{2}_{-}\, dx\ge \int_{K_{\delta}\cap \{u(\cdot, t)<\eps k\}}(k-\eps k)^{2}\, dx\ge  k^{2}(1-\eps)^{2}| K_{\delta}\cap \{u(\cdot, t)<\eps k\}|
\]
which, inserted into the previous estimate and dividing by $k^{2}(1-\eps)^{2}$ gives 
\begin{equation}
\label{kdelta}
|K_{\delta}\cap \{u(\cdot, t)< \eps k\}|\le \frac{1}{(1-\eps)^{2}}\left(1-\mu+\frac{\bar C\, k^{p-2}\, t}{(1-\delta)^{p}}\right).
\end{equation}
Therefore
\[
\begin{split}
1-P_{t}(K_{1};\,  u\ge \eps k)&\le 1-|K_{\delta}\cap \{u(\cdot, t)\ge \eps k\}|=1-\delta^{N}+|K_{\delta}\cap \{u(\cdot, t)< \eps k\}|\\
&\le 1-\delta^{N} +\frac{1}{(1-\eps)^{2}}\left(1-\mu+\frac{\bar C\, k^{p-2}\, (t+1)}{(1-\delta)^{p}}\right).
\end{split}
\]
Choose $\delta, \eps\in\ ]0, 1[$ and, for each $t\in [0, T]$, $k_{t}\in \ ]0, 1[$ such that 
\[
1-\delta^{N}=\frac{\mu}{8},\qquad \frac{1-\mu}{(1-\eps)^{2}}=1-\frac{3}{4}\mu,\qquad \frac{\bar C\, k_{t}^{p-2}(t+1)}{(1-\eps)^{2}(1-\delta)^{p}}=\frac{\mu}{8}. 
\]
Clearly it holds $\delta=\delta(\mu)$, $\eps=\eps(\mu)$ and therefore $k_{t}=k(\mu)/(t+1)^{\frac{1}{p-2}}$. With these choices we have $1-P_{t}(K_{1};\,  u\ge \eps k_{t})\le 1-\mu/2$, proving the claim.
\end{proof}

The previous Lemma suggests to consider the function $(t+1)^{\frac{1}{p-2}}u(x, t)$, which is a supersolution to an equation similar to \eqref{pt10}, but with structural constants depending on $t$ (and degenerating for large times). In order to keep the structural conditions independent of $t$, it turns out that the change of time variable $t+1=e^{\tau}$ suffices, so that we consider instead
\begin{equation}
\label{cv1}
v(x, e^{\tau})=e^{\frac{\tau}{p-2}}u(x, e^{\tau}-1).
\end{equation}
A straightforward calculation shows that $v$ is a solution on $Q_{1,\log (T+1)}$ of 
\[
v_{t}={\rm div} \tilde A(x, v, Dv)+v/(p-2)
\]
with $\tilde A(x, s, z):=e^{\frac{\tau}{p-2}}A\left(x, se^{\frac{-\tau}{p-2}}, ze^{\frac{-\tau}{p-2}}\right)$ obeying the structural conditions in \eqref{pt10}. In particular, if $u\ge 0$, $v$ belongs to the class of nonnegative supersolution of \eqref{pt10}.

\begin{lemma}[Shrinking lemma]\label{slemma}
Suppose $v\ge 0$ is a supersolution in $Q_{2,S}$ of \eqref{pt10} for $p\ge  2$ such that 
\begin{equation}
\label{pt}
P_{t}(K_{1}; \, v\ge k)\ge \mu\qquad \forall t\in [0, S]
\end{equation} 
for some $\mu>0$, $k\ge 0$. There exists $\beta=\beta(\mu)$ such that
\begin{equation}
\label{pt2}
P\left(K_{1}\times [0, (2^{n}/k)^{p-2}]; v\le k/2^{n}\right)\le \beta(\mu)/n^{1-\frac{1}{p}}, \qquad \text{if  $(2^{n}/k)^{p-2}\le S$}
\end{equation}
\end{lemma}

\begin{proof}
The proof is very similar to the one of Lemma \ref{slemmapar} and we only sketch it. Suppose $n\ge 1$ satisfies $2^{n(p-2)}\le S\, k^{p-2}$ and let $k_{j}=k/2^{j}$ for $j=0, \dots, n$. Using the energy inequality \eqref{ein} on $Q_{1, S}$ with $\eta$ as in \eqref{varphi} with $R=1$, we get
\[
\iint_{Q_{1, S}}|D(v-k_{j})_{-}|^{p}\, dx\, dt\le \bar C\int_{K_{1}} (v(x, 0)-k_{j})_{-}^{2}\, dx+\bar C\iint_{Q_{1, S}}(v-k_{j})_{-}^{p}\, dx\, dt
\le \bar C (k_{j}^{2}+Sk_{j}^{p}).
\]
For any $t\in [0, S]$ apply the De Giorgi - Poincar\'e inequality and \eqref{pt} to obtain
\[
(k_{j}-k_{j+1})P_{t}(K_{1};\, v(\cdot, t)\le k_{j+1})\le \frac{\bar C}{\mu}\int_{K_{1}\cap \{k_{j+1}\le v(\cdot, t)\}}|D(v(x, t)-k_{j})_{-}|\, dx.
\]
Integrate over $[0, S]$, use H\"older's inequality and the energy estimate to get for $j=0, \dots, n-1$
\begin{equation}
\label{pt12}
\frac{k_{j}}{2}P(Q_{1, S}; \, v\le k_{j+1})\le \frac{\bar C}{\mu}\left(\frac{k_{j}^{2}}{S}+k_{j}^{p}\right)^{\frac{1}{p}}\left(P(Q_{1,S}; \, v\le k_{j})-P(Q_{1,S};\, v\le k_{j+1})\right)^{1-\frac{1}{p}}.
\end{equation}
As $j\le n$, it holds $2^{j(p-2)}\le S\, k^{p-2}$ as well, implying $k_{j}^{2}/S\le k_{j}^{p}$. Thus we can simplify all the factors involving $k_{j}$ above, giving for all $j\le n-1$
\[
\left(P(Q_{\delta_{0}, S};\, v\le k_{j+1}\right)^{\frac{p}{p-1}}\le \bar C\, \mu^{\frac{p}{1-p}} \left(P(Q_{1, S};\, v\le k_{j})-P(Q_{\delta_{0},S};\,  v\le k_{j+1})\right).
\]
which, summed over $j\le n-1$ gives \eqref{pt2} by the usual telescopic argument.
\end{proof}

\begin{lemma}[Measure-to-point estimate]\label{EP0}
 For any $\mu\in \ ]0, 1[$ there exists $m(\mu)\in \ ]0, 1[$, $T(\mu)>1$ such that any supersolution $u\ge 0$ in $Q_{2, T(\mu)}$ fulfills
\begin{equation}
\label{exp0}
P_{0}(K_{1}; u\ge 1)\ge \mu\quad \Rightarrow \quad u\ge m(\mu) \quad \text{in $K_{1/2}\times [ T(\mu)/2, T(\mu)]$}.
\end{equation}
\end{lemma}

\begin{proof}
Let $T$ to be determined and suppose $u\ge 0$ is a supersolution in $Q_{1, T}$. By Lemma \ref{lemmapt},  
\[
P_{t}(K_{1};\, (t+1)^{\frac{1}{p-2}}\, u\ge k(\mu))\ge \mu/2,\qquad \forall t\in [0, T].
\]  
If $v$ is defined as per \eqref{cv1}, the previous condition reads
\[
P_{\tau}(K_{1};\, v\ge k(\mu))\ge \mu/2,\qquad \forall \tau\in [0, S], \quad S=\log(T+1)>0
\]
so that,  being $v\ge 0$ a supersolution, Lemma \ref{slemma} implies
\[
P(Q_{1, S_{n}};\, v\le S_{n}^{\frac{1}{2-p}})\le \beta(\mu/2)/n^{1-\frac{1}{p}},\qquad \text{for $S_{n}:=(2^{n}/ k(\mu))^{p-2}$}.
\]
Next choose $n=n(\mu)$, (and thus $S=S(\mu):=S_{n}$ and $T=T(\mu):=e^{S}-1$) so that $\beta(\mu/2)/n^{1-\frac{1}{p}}\le \nu(1)$, with $\nu$ given in \eqref{pt3}. Lemma \ref{critlemma} applied on $Q_{1, S}$ with $h=S^{\frac{1}{2-p}}$ thus gives
\[
v\ge S^{\frac{1}{2-p}}/2\quad \text{on $K_{1/2}\times [S/2, S]$}.
\]
Recalling the definition \eqref{cv1} of $v$, in terms of $u$ the latter implies
\[
u\ge e^{-\frac{T}{p-2}}\log^{\frac{1}{2-p}}(T+1)/2\quad \text{on $K_{1/2}\times \left[\sqrt{T+1}-1, T\right]\supseteq  K_{1/2}\times \left[ T/2, T\right]$}.
\]
 \end{proof}

\begin{theorem}[H\"older regularity]\label{holdp>2}
Any $L^{\infty}_{\rm loc}(\Omega_{T})$ solution of \eqref{pt10} in $\Omega_{T}$ for $p>2$ belongs to $C^{\bar\alpha}_{\rm loc}(\Omega_{T})$, with $\bar\alpha$ depending only on $N$, $p$, $C_{0}$ and $C_{1}$. Moreover, there exist $\bar T\ge 1$ and $\bar C>0$ such that  if $Q_{R}^{-}(\bar T):=K_{2R}\times [-\bar T\, R^{p}, 0]\subseteq \Omega_{T}$, for any $r\in [0, R]$ it holds
\begin{equation}
\label{oscp>2}
{\rm osc}(u(\cdot, 0), K_{r})\le \bar C\, \max\left\{1, \|u\|_{L^{\infty}(Q_{R}^{-}(\bar T))}\right\} \big(\frac{r}{R}\big)^{\bar\alpha}.
\end{equation}
\end{theorem}

\begin{proof} Let $\bar T=T(1/2)$ be given in the previous Lemma. By space/time traslation, it suffices to prove an oscillation decay near $(0, 0)$, with $Q^{-}_{r_{0}}(\bar T)\subseteq \Omega_{T}$ for some $r_{0}>0$. By \eqref{scalingin},  $u(x\, r_{0}, t\, r_{0}^{p})$ (still denoted by $u$) solves \eqref{pt10} on $Q^{-}_{1}(\bar T)$. Let $M=\|u\|_{L^{\infty}(Q^{-}_{1}(\bar T))}$: if $M>1$ consider  $v(x, t)=M^{-1}\, u(x, M^{2-p}\, t)$, 
which, being $p>2$, solves \eqref{pt10} on $Q^{-}_{1}(\bar T)$ and $\|v\|_{L^{\infty}(Q^{-}_{1}(\bar T))}\le 1$. Applying Lemma \ref{lemmaosc} to $v(\cdot, \bar T+t)$ (notice that $Q_{1}^{-}(\bar T)$ translates to $Q_{2, \bar T}$) proves the H\"older continuity of $u$, while \eqref{oscp>2} is obtained from \eqref{oscest1} for $v$, scaling back to $u$.
\end{proof}

The next Lemma shows that expansion of positivity geometry in the degenerate setting is very similar to the nondegenerate case. Compared to figure \ref{EPfig}, the only difference is in the shape of the paraboloid which is thinner for larger $p$.

\begin{lemma}[Expansion of positivity]\label{expdeg}
There exists $\bar \lambda>0$ and, for any $\mu>0$,  $c(\mu)\in \ ]0, 1[$, $ \gamma(\mu)\ge 1$,  such that if $u\ge 0$ is a supersolution to \eqref{pt10} in $Q_{4R, T}$, then
\[
P_{0}(K_{r};\, u\ge k)\ge \mu\quad \Rightarrow\quad \inf_{K_{\rho}}u\big(\cdot, \gamma(\mu) \big(\frac{k\, r^{\bar\lambda}}{\rho^{\bar\lambda}}\big)^{2-p}\rho^{p}\big)\ge c(\mu)\, \frac{k\, r^{\bar\lambda}}{\rho^{\bar\lambda}}
\]
whenever $r\le \rho\le R$ and  $\gamma(\mu) \big(k\, r^{\bar\lambda}/\rho^{\bar\lambda}\big)^{2-p}\rho^{p}\le T/c(\mu)$.
\end{lemma}

\begin{proof}
We first generalize \eqref{exp0} as follows: there exists $\theta(\mu)>0$ such that for any $\eta\ge 1, h> 0$  
\begin{equation}
\label{exp1}
P_{0}(K_{\rho};\, u\ge h)\ge \mu\quad \Rightarrow\quad u\ge c(\mu)\, h/\eta^{\frac{1}{p-2}}\quad \text{in } K_{2\rho}\times \left[\frac{\theta(\mu)}{2}\, \frac{\rho^{p}}{h^{p-2}}, \eta\,\theta(\mu)\,  \frac{\rho^{p}}{h^{p-2}}\right].
\end{equation}
By considering $v(x, t)=h^{-1}u(\rho\, x, h^{2-p}\, \rho^{p}\, t)$ and recalling \eqref{scalingin}, it suffices to prove the claim for $\rho=h=1$. By Lemma \ref{lemmapt}, \eqref{pt0} holds true, implying $P_{s}(K_{4}; u\ge k(\mu)/(s+1)^{\frac{1}{p-2}})\ge \mu\, 4^{-N-1}$ where $s$ is a parameter in  $[0, \eta-1]$. Rescale \eqref{exp0} considering 
\[
v(x, t)=k_{s}(\mu)^{-1}\, u(4\, x, k_{s}(\mu)^{2-p}\, 4^{p}\, t), \qquad k_{s}(\mu):=k(\mu)/(s+1)^{\frac{1}{p-2}}
\]
which fulfills $P_{s}(K_{1};\, v\ge 1)\ge \mu\, 4^{-N-1}$, to obtain, with the notations of \eqref{exp0}  
\[
v\ge m(\mu\, 4^{-N-1})\quad \text{in } \quad K_{1/2}\times \big[s+T(\mu\, 4^{-N-1})/2, s+T(\mu\, 4^{-N-1})\big],
\]
If $c(\mu):=k(\mu)\, m(\mu\, 4^{-N-1})$, using $s\in [0, \eta-1]$, the latter reads in terms of $u$  
\[
\begin{split}
&\inf_{K_{2}}u(\cdot, t)\ge k_{s}(\mu)\, m(\mu\, 4^{-N-1})\ge c(\mu)\, \eta^{\frac{1}{2-p}} \qquad \text{if $t\in I_{s}$ for some $s\in [0, \eta-1]$}\\
&\qquad I_{s}:= \big[4^{p}\, k_{s}(\mu)^{2-p}\, (s+ T(\mu\, 4^{-N-1})/2), 4^{p}\, k_{s}(\mu)^{2-p}\, (s+T(\mu\, 4^{-N-1}))\big].
\end{split}
\]
Finally, let $\theta(\mu)=4^{p}\, k(\mu)^{2-p}T(\mu\, 4^{-N-1})$ and observe that $\cup_{s\in [0, \eta-1]}I_{s}\supseteq [\theta(\mu)/2, \eta\, \theta(\mu)]$
\footnote{Both $a(s)=\inf I_{s}$ and $b(s)=\sup I_{s}$ are continuous, hence $\cup_{s\in [0, \eta-1]}I_{s}=\big[\inf_{s\in [0, \eta-1]}a(s), \sup_{s\in [0, \eta-1]}b(s)\big]$. Then observe that $a(0)=\theta(\mu)/2$ while  $b(\eta-1)\ge \eta\, \theta(\mu)$.}, proving \eqref{exp1}. Notice that all the argument goes through as long as it holds $\sup_{s\in [0, \eta-1]}I_{s}= 4^{p}k(\mu)^{2-p}\eta(\eta-1+T(\mu\, 4^{-N-1}))\le T$ which, scaling back, is ensured e.g. by $\eta^{2}\, \theta(\mu)\, \rho^{p}\, h^{2-p}\le T$.

To prove the lemma, again we can suppose $k=1$, (otherwise consider $v(x, t)=k^{-1}u(x, k^{2-p}\, t)$). Let, as per \eqref{exp1}, $\bar\theta=\theta(1)$ and $\bar c=c(1)$, $\rho_{n}=2^{n}\, r$ and define  recursively 
\begin{equation}
\label{tn}
t_{0}=\frac{\theta(\mu)}{2}\,  r^{p},\qquad t_{n+1}=t_{n}+\frac{\bar\theta}{2}\, (c(\mu)\,   {\bar c}^{n})^{2-p}\, \rho_{n}^{p}.
\end{equation}
Applying \eqref{exp1} with $\eta=1$, we get
\[
P_{0}(K_{r};\, u\ge 1)\ge \mu \quad \Rightarrow \quad P_{t_{0}}(K_{\rho_{1}};\, u\ge c(\mu))=1\quad \Rightarrow \quad P_{t_{1}}(K_{\rho_{2}};\, u\ge c(\mu)\, \bar c)=1
\]
and, proceeding by induction, we infer $P_{t_{n}}(K_{\rho_{n+1}};\, u\ge c(\mu)\, \bar c^{n})=1$, for all $n\ge 0$. In particular $P_{t_{n}}(K_{\rho_{n}};\, u\ge  c(\mu)\, \bar c^{n})=1$, so we again use \eqref{exp1}  for $\eta$ to be determined to obtain
\begin{equation}
\label{last1}
u\ge c(\mu)\, \bar c^{n+1}\, \eta^{\frac{1}{2-p}}\quad \text{in } K_{\rho_{n+1}}\times [t_{n+1}, t_{n}+\eta \, \bar\theta \, (c(\mu)\, \bar c^{n})^{2-p}\rho_{n}^{p}].
\end{equation}
Choose $\eta$ so that 
\[
t_{n}+\eta \, \bar\theta \, (c(\mu)\, \bar c^{n})^{2-p}\rho_{n+1}^{p}=t_{n+2}\quad \Leftrightarrow\quad \eta=\bar\eta:=(1+\bar c^{2-p}\, 2^{p})/2.
 \]
For this choice \eqref{last1} holds in the time interval $[t_{n+1}, t_{n+2}]$ giving, by monotonicity,
\[
\inf_{K_{\rho_{n}}}u(\cdot, t)\ge c(\mu)\, \bar c^{n+m}\qquad \text{for all  \ $ t_{n}\le t\le t_{n+m}$, \  $n, m\ge 0$}
\]
for a smaller $c(\mu)$.  Let $s_{n}:= \bar c^{n(2-p)}\rho_{n}^{p}$; computing $t_{n}$ we find $t_{n}\simeq_{\mu}  s_{n}$ with constants depending on $\mu$, therefore, for sufficiently large $m=m(\mu)\in \mathbb{N}$,  $t_{n}\le \gamma(\mu)\, s_{n} \le \gamma(\mu)\, s_{n+1}\le t_{n+m}$ and 
\[
\inf_{K_{\rho_{n}}}u(\cdot, t)\ge c(\mu)\, \bar c^{n}\qquad \text{for all  \   $ \gamma(\mu)\, s_{n}\le t\le \gamma(\mu)\, s_{n+1}$, \ $n\ge 0$}.
\]
The same argument as in the end of the proof of Lemma \ref{epospar} gives the thesis.
 \end{proof}

\begin{theorem}[Forward Harnack inequality]
Let $u$ be a nonnegative solution of \eqref{pt10} in $K_{16R}\times [-T, T]$. Then there exists $\bar C>\bar\theta>0$ such that if $\bar C\, u(0, 0)^{2-p}\, R^{p}\le T$
\begin{equation}
\label{degharnack}
u(0, 0)\le \bar C\inf_{K_{R}} u(\cdot, \bar\theta\, u(0, 0)^{2-p}\, R^{p}).
\end{equation}
\end{theorem}

\begin{proof}
Thanks to \eqref{scalingin}, the function $v(x, t)=u(0, 0)^{-1}u(R^{p}\, x,  u(0, 0)^{2-p}\, R^{p}\, t)$ solves \eqref{pt10} in $K_{16}\times [-T\, u(0, 0)^{p-2}\, R^{-p},  T\, u(0, 0)^{p-2}\, R^{-p}]$ and $v(0, 0)=1$. It then suffices to prove the existence of  $\bar\theta\ge 1, \bar c\in \ ]0, 1[$ such that any solution $u\ge 0$ of \eqref{pt10} in $K_{16}\times [-2, \bar\theta/\bar c]$ obeys
\begin{equation}
\label{claim89}
u(0, 0)=1\quad \Rightarrow\quad\inf_{K_{1}} u(\cdot, \bar\theta)\ge \bar c.
\end{equation}
As in Theorem \ref{HIhom}, let $Q^{-}_{\rho}:=K_{\rho}\times [-\rho^{p}, 0]$ and consider $\psi(\rho):=(1-\rho)^{\bar\lambda}\sup_{Q_{\rho}^{-}}u$ for $\rho\in [0, 1]$,
where  $\bar\lambda$ is given in Lemma \ref{expdeg}. Let by continuity $\rho_{0}\in [0,1]$, $(x_{0}, t_{0})\in Q^{-}_{\rho_{0}}$ such that
\[
\max_{[0,1]}\psi(\rho) =(1-\rho_{0})^{\bar\lambda}u_{0}\qquad u_{0}:=u(x_{0}, t_{0}),
\]
choose $\bar \xi\in \ ]0, 1[$ such that $(1-\bar\xi)^{-\bar \lambda}=2$ and let $r=\bar \xi\, (1-\rho_{0})$. As in \eqref{etazero}, it holds
$u_{0}\, r^{\bar\lambda}\ge \xi^{\bar\lambda}$. Let  $\bar T$ be given in Theorem \ref{holdp>2} and let  $\widetilde Q_{r}:=K_{\bar T^{-1/p}\,r}(x_{0})\times [t_{0}- r^{p}, t_{0}]$.  Since $\bar T\ge 1$, it holds $\widetilde Q_{r}\subseteq Q^{-}_{\rho_{0}+r}$ and we can deduce as in \eqref{jk45} that $\sup_{\widetilde Q_{r}} u\le (1-\bar\xi)^{-\bar\lambda}\, u_{0}$. Then  \eqref{oscp>2} ensures 
\[
{\rm osc}(u(\cdot, t_{0}), K_{\rho}(x_{0}))\le 2\, \bar C\, u_{0} \, (\rho/r)^{\bar \alpha} \quad\text{for}\quad  \rho\le \bar T^{-1/p}r.
\]
Since $u(x_{0}, t_{0})=u_{0}$, we infer  $u(\cdot, t_{0})\ge u_{0}/2$ in $K_{\bar\eta r}(x_{0})$ for some $\bar\eta>0$. Therefore $P_{t_{0}}(K_{r}(x_{0});\ u\ge u_{0}/2)\ge \bar\eta^{N}$ and being $u_{0}\, r^{\bar\lambda}\ge \bar \xi^{\bar\lambda}$, {\em a fortiori} it holds 
$P_{t_{0}}(K_{r}(x_{0});\ u\ge \bar \xi^{\bar\lambda}\, r^{-\bar\lambda}/2)\ge \bar\eta^{N}$. Since $K_{12}(x_{0})\subseteq K_{16}$, Lemma \ref{expdeg} with $k= \bar \xi^{\bar\lambda}\, r^{-\bar\lambda}/2$ gives for suitable $\bar\gamma\ge 1>\bar c>0$
\[
\inf_{K_{\rho}(x_{0})} u\big(\cdot, t_{0}+\bar\gamma\, \bar \xi^{\bar\lambda(2-p)}\, \rho^{p+\bar\lambda(p-2)}\big)\ge \bar c\, \xi^{\bar\lambda}/\rho^{\bar\lambda},\qquad r\le \rho\le 3,\quad \bar\gamma\, \xi^{\bar\lambda(2-p)}\, \rho^{p+\bar\lambda(p-2)}\le T/\bar c.
\]
In \eqref{claim89} we let $\bar\theta:=\bar\gamma\,  2^{p+\bar\lambda(p-2)}$  and choose $\rho$ such that $ t_{0}+\bar\gamma\, \xi^{\bar\lambda(2-p)}\rho^{p+\bar\lambda(p-2)}=\bar\theta$. From $t_{0}\le 0$ we get $\rho\ge 2$ (and thus $K_{\rho}(x_{0})\supseteq K_{1}$) and from $t_{0}\ge -1$ we infer 
\[
\bar\gamma\, \xi^{\bar\lambda}\, \rho^{p+\bar\lambda(p-2)}\le 1+  \bar\gamma\, \xi^{\bar\lambda(2-p)}\,  2^{p+\bar\lambda(p-2)}\le\bar\gamma\, \xi^{\bar\lambda(2-p)}(1+   2^{p+\bar\lambda(p-2)})\quad \Rightarrow \quad \rho\le 3.
\]
Hence (by eventually lowering $\bar c$), such $\rho$ is admissible and its upper bound proves \eqref{claim89}. 
\end{proof}

\begin{theorem}[Backward Harnack inequality]
Let $u$ be a nonnegative solution of \eqref{pt10} in $K_{16R}\times [-T, T]$. Then there exists $\bar C'>\bar\theta'>0$ such that if $\bar C'\, u(0, 0)^{2-p}\, R^{p}\le T$
\begin{equation}
\label{degharnackback}
\sup_{K_{R}} u(\cdot, -\bar\theta'\, u(0, 0)^{2-p}\, R^{p})\le \bar C'\, u(0, 0).
\end{equation}
\end{theorem}

\begin{proof}
By the same scaling argument as before, we can reduce to the case $R=1$, $u(0, 0)=1$. Let, for $t\ge 0$, $w(t):=u(0, - t)$ and apply \eqref{degharnack} to $u$ with $(0, -t)$ instead of $(0, 0)$ to get  
\[
u(0, - t+\bar\theta\, w^{2-p}(t)\, \rho^{p})\ge w(t)/\bar C, \qquad 0<\rho\le 1.
\]
 If $w(t)\le 2\, \bar C$ for some $t\le \bar\theta/(2\, \bar C)^{p-2}$, we can choose $\rho(t)>0$ such that $\rho(t)^{p}=t\, w^{p-2}(t)/\bar\theta\le 1$, obtaining $u(0, 0)=u(0, - t+\bar\theta \, w^{2-p}(t)\, \rho^{p}(t))$. Therefore we proved
\[
0\le t\le \bar\theta/\bar (2C)^{p-2}\quad \&\quad  
w(t)\le 2\, \bar C\qquad \Rightarrow \qquad w(t)\le \bar C
\]
which implies $w(t)\le 2\, \bar C$ for all $0\le t\le \bar\theta/\bar (2\, \bar C)^{p-2}$ by a continuity argument. Letting $\bar \theta'= \bar\theta/\bar (2\, \bar C)^{p-2}$, $\bar C'=2\, \bar C$ we prove \eqref{degharnackback} by contradiction: from $u(0, -\bar\theta')\le \bar C$ and $\sup_{K_{1}}u(\cdot, -\bar\theta')>2\, \bar C$, by continuity there exists $\bar x\in K_{1}$ such that $u(\bar x, -\bar\theta')=2\, \bar C$. Since $0\in K_{1}(\bar x)$ and $\bar \theta \, (2\, \bar C)^{2-p}=\bar \theta'$, the Harnack inequality \eqref{degharnack} for $u$ at the point $(\bar x, -\theta')$ implies 
\[
1=u(0, 0)\ge \inf_{K_{1}(\bar x)} u(\cdot, -\bar \theta'+\bar \theta (2\, \bar C)^{2-p})\ge u(\bar x, -\bar\theta')/\bar C=2.
\]
\end{proof}

\subsection{Singular parabolic equations}\ \\

We conclude with the Harnack inequality for solutions of parabolic singular supercritical equations. The measure-to-point estimate will be treated through a change of variable analogous to the degenerate case, but requires a little bit more care. From this we'll derive a H\"older continuity result for all {\em bounded} solutions in the full range $p\in \ ]1, 2[$. As mentioned at the introduction of the section, the proof of the Harnack inequality will rely on theorem \ref{sl1}, which we state without proof.

\begin{lemma}\label{lemmaspt}
Let $u\ge 0$ be a supersolution  in $Q_{1, T}$ of \eqref{pt10} with $p\le 2$. For any $\mu>0$ there exists  $k(\mu)\in \ ]0, 1[$, $T(\mu)\in \ ]0, \min\{1, T\}]$  such that  
\[
P_{0}(K_{1};\, u\ge 1)\ge \mu \quad \Rightarrow\quad 
P_{t}\left(K_{1};\, u\ge k(\mu)\right)>\mu/2 \qquad  \forall t\in [0, T(\mu)].
\]
\end{lemma}

\begin{proof}
Proceed as in Lemma \ref{lemmapt} to get \eqref{kdelta} for $k=1$, $\delta, \eps, \in \ ]0, 1[$ and $t\in [0, T]$. Thus 
\[
1-P_{t}(K_{1};\,  u\ge \eps)\le 1-\delta^{N} +\frac{1}{(1-\eps)^{2}}\left(1-\mu+\frac{\bar C\, t}{(1-\delta)^{p}}\right).
\]
Choose $\delta=\delta(\mu)$ and $\eps=\eps(\mu)$ as per $1-\delta^{N}=\mu/8$ and $(1-\mu)/(1-\eps)^{2}=1-3\mu/4$, so that
\[
P_{t}(K_{1};\,  u\ge \eps(\mu))\ge \frac{5}{8}\mu-C(\mu)\, t, \qquad \text{ for any $t\in [0, T]$}.
\]
Choosing $T(\mu)\le T$ such that $C(\mu)\, T(\mu)\le \mu/8$ gives the claim.
\end{proof}

\begin{lemma}[Shrinking lemma]\label{slemmas}
Let $v\ge 0$ be a supersolution in $Q_{2, S}$ of \eqref{pt10} with $p\in \ ]1, 2[$ such that for some $\mu, k\in \ ]0, 1[$
\[
P_{t}\left(K_{1};\, v\ge k\right)>\mu \qquad  \forall t\in [0, S].
\]
Then there exists $\beta=\beta(\mu)>0$ such that for any $n\ge 1$
\[
P\left(Q_{1, S};\, v\le k/2^{n}\right)\le \beta(\mu) \left(1+k^{2-p}/S\right)^{\frac{1}{p-1}}/n^{1-\frac{1}{p}}.
\]
\end{lemma}

\begin{proof}
Proceed as in the proof of Lemma \ref{slemma} up to \eqref{pt12}.
As $j\ge 1$ and $p< 2$, it holds $k_{j}^{p}+k_{j}^{2}/S\le k_{j}^{p}(1+k^{2-p}/S)$, so that 
\[
P(Q_{1, S};\, v\le k_{j+1})^{\frac{p}{p-1}}\le \bar C\, \mu^{\frac{p}{p-1}} \left(1+k^{2-p}/S\right)^{\frac{1}{p-1}} \left(P(Q_{1, S};\, v\le k_{j})-P(Q_{1, S};\,  v\le k_{j+1})\right), 
\]
which yields the conclusion summing over $j\le n-1$.
\end{proof}

\begin{lemma}[Measure-to-point estimate]
Let $u\ge 0$ be a supersolution of \eqref{pt10} for $p\in \ ]1, 2]$. For any $\mu\in \ ]0, 1]$ there exists $m(\mu), T(\mu)\in \ ]0, 1[$  such that 
\begin{equation}
\label{claims}
P_{0}(K_{1};\, u\ge 1)\ge \mu\quad \Rightarrow\quad u\ge m(\mu)\quad \text{ in $K_{1/4}\times [T(\mu)/2, T(\mu)]$}.
\end{equation}
Moreover, $T(\mu)$ can be chosen arbitrarily small by decreasing $m(\mu)$.
\end{lemma}

\begin{proof}
Let $T(\mu)$, $k(\mu)$ be given in  Lemma \ref{lemmaspt}: clearly $T(\mu)$ can be chosen arbitrarily small. Since $p<2$, an explicit computation shows that for any fixed $T\in [T(\mu)/2, T(\mu)]$, the function 
\[
v(x, \tau)=e^{\frac{\tau}{2-p}}u(x, T-e^{-\tau}), \qquad x\in K_{1}, \tau\ge -\log T
\]
is a supersolution to  \eqref{pt10}. The conclusion for $u$ of Lemma \ref{lemmaspt} becomes for $v$
\[
P_{\tau}(K_{1};\, v\ge e^{\frac{\tau}{2-p}}k(\mu))\ge \mu/2, \qquad \forall \tau \ge -\log T,
\]
and for $s\ge -\log  T$ to be chosen, the latter implies (thanks to $p<2$)
\begin{equation}
\label{spt13}
P_{\tau}(K_{1};\, v\ge e^{\frac{s}{2-p}}k(\mu))\ge \mu/2, \qquad \forall \tau\ge s \ge -\log T .
\end{equation}
 For $\nu(\mu)$ and $\beta(\mu)$ given in  Lemma \ref{critlemma} and \ref{slemmas}, let $n=n(\mu)\ge 1$ be such that
\[
\beta(\mu)\, n^{\frac{1}{p}-1}\left( k(\mu)^{2-p}+1\right)^{\frac{1}{p}}\le \nu(k(\mu)),
\]
and for  $s\ge -\log  T $ let $I_{s}=[e^{s}, 2e^{s}]$. Due to \eqref{spt13},  Lemma \ref{slemmas} applies to $v$  on $K_{1}\times I_{s}$ for $k=k(\mu)\, e^{\frac{s}{2-p}}$, giving, by the choice of $n=n(\mu)$,
\begin{equation}
\label{spt7}
P(K_{1}\times I_{s};\, v\le k(\mu)\, e^{\frac{s}{2-p}}/2^{n})\le \nu(k(\mu)).
\end{equation}
Subdivide $I_{s}$ in $[2^{n(2-p)}]+1$ disjoint intervals, each of length $\lambda \in [e^{s}\, (2^{-n(2-p)}-1),e^{s}\, 2^{-n(2-p)}]$. On at least one of them, say $J=[a, a+\lambda]\subseteq I_{s}$, \eqref{spt7} holds for $J$ instead of $I_{s}$, thus {\em a fortiori}
\[
P(K_{1}\times J;\, v\le k(\mu)\, \lambda^{\frac{1}{2-p}})\le P(K_{1}\times J;\, v\le k(\mu)\, e^{\frac{s}{2-p}}/2^{n})\le \nu(k(\mu)).
\]
Apply \eqref{pt3} to $v$ on $K_{1}\times J$ to obtain 
\[
v(x, \tau)\ge k(\mu)\lambda^{\frac{1}{2-p}}/2\qquad \forall\  \tau\in [a+\lambda/2, a+\lambda]\subseteq I_{s}, \ x\in K_{1/2}.
\]
Since $\lambda\ge e^{s}\, (2^{-n(2-p)}-1)$,  in terms of $u$ and $s$, the latter implies that for some $\tau_{s}\in J\subseteq [e^{s}, 2e^{s}]$
\[
\inf_{K_{1/2}}u(\cdot, T-e^{-\tau_{s}})= e^{-\frac{\tau_{s}}{2-p}}\inf_{K_{1/2}}v(\cdot, \tau_{s})\ge  \frac{k(\mu)\, e^{\frac{s-\tau_{s}}{2-p}}}{2^{2n}}=:c(\mu)\, e^{\frac{s-\tau_{s}}{2-p}}.
\]
Apply Lemma \ref{sdg} to  $u$ in $K_{1/2}\times [T- e^{-\tau_{s}}, T]$ with $h= c(\mu)\, e^{\frac{s-\tau_{s}}{2-p}}$ to get 
\begin{equation}
\label{spt20}
\inf_{K_{1/4}}u(\cdot, t)\ge c(\mu)\, e^{\frac{s-\tau_{s}}{2-p}}/2 \qquad \forall \, t\in [T-e^{-\tau_{s}}, T-e^{-\tau_{s}} +\min\{e^{-\tau_{s}}, \bar \sigma\, 2^{-p}\, c(\mu)\, e^{s-\tau_{s}}\}].
\end{equation}
Finally, let $\tilde s=s(\mu)=\max\{-\log (T(\mu)/2), -\log(\bar\sigma\, 2^{-p}\, c(\mu)\}$, so that it holds
\[
\tilde s\ge -\log T\qquad\text{and}\qquad   \bar \sigma\, 2^{-p}\, c(\mu)\, e^{\tilde s-\tau_{\tilde s}}\ge e^{-\tau_{\tilde s}}.
\]
 Therefore \eqref{spt20} holds for $t=T$ and from  $\tau_{\tilde s}\le 2e^{\tilde s}$ we deduce a lower bound on $e^{\tilde s-\tau_{\tilde s}}$ depending only on $\mu$, which proves \eqref{claims} by the arbitrariness of $T\in [T(\mu)/2, T(\mu)]$.
 \end{proof}

 \begin{theorem}[H\"older regularity]
Any $L^{\infty}_{\rm loc}(\Omega_{T})$ solution of \eqref{pt10} in $\Omega_{T}$ for $p\in \ ]1, 2[$ belongs to $C^{\bar\alpha}_{\rm loc}(\Omega_{T})$, with $\bar\alpha$ depending only on the data. Moreover there exists $\bar S$, also depending on the data, with the following property: if $S\ge\bar S$ there exist $\bar C(S)>0$ such that  
\begin{equation}
\label{oscp<2}
\sup_{K_{2R}\times [-k^{2-p}\, R^{p}, 0]}u\le S\, k\quad \Rightarrow\quad {\rm osc}(u, K_{r}\times [-k^{2-p}\, r^{p}, 0])\le \bar C(S)\, k\, \big(\frac{r}{R}\big)^{\bar\alpha},\quad r\le R,
\end{equation}
for any $k, R>0$ for which $K_{2\, R}\times [-k^{2-p}\, R^{p}, 0]\subseteq \Omega_{T}$.
 \end{theorem}
 
 \begin{proof}
Let $\bar T=T(1/2)\in \ ]0, 1]$ given in the previous Lemma. By space-time translations and rescaling we are reduced to prove H\"older continuity near  $(0, 0)$ with $\bar Q:=K_{2}\times[-\bar T, 0]\subseteq \Omega_{T}$. If $M:=\|u\|_{L^{\infty}(\bar Q)}>1$ consider  $M^{-1}\, u(M^{(p-2)/p}x, t)$
which, being $p\in \ ]1, 2[$, solves \eqref{pt10} in $\bar Q$ and fulfills $\|v\|_{L^{\infty}(\bar Q)}\le 1$. Applying Lemma \ref{lemmaosc} gives the first statement. To prove \eqref{oscp<2}, suppose that $S\ge \bar T^{\frac{1}{p-2}}=:\bar S$, rescale to $R=1$, then let $\bar\gamma(S):=S^{p-2}\, \bar T^{-1}$  and consider 
\[
v(x, t)=(S\, k)^{-1}u(\rho\, x, \tau\, t)\qquad \rho=\bar \gamma(S)^{1/p},\quad \tau= k^{2-p}\, \bar T^{-1}.
\]
Thanks to \eqref{scalingin}, it is readily verified that $v$ solves \eqref{pt10} in $\bar Q$ and by the assumption in \eqref{oscp<2} it is bounded by $1$. Applying \ref{oscest1} (notice that $\bar T$ is the same) and rescaling back gives \eqref{oscp<2} for all $r\le \bar\gamma(S)^{1/p} R$ and hence for all $r\le R$ with eventually a bigger constant.
 \end{proof}
 
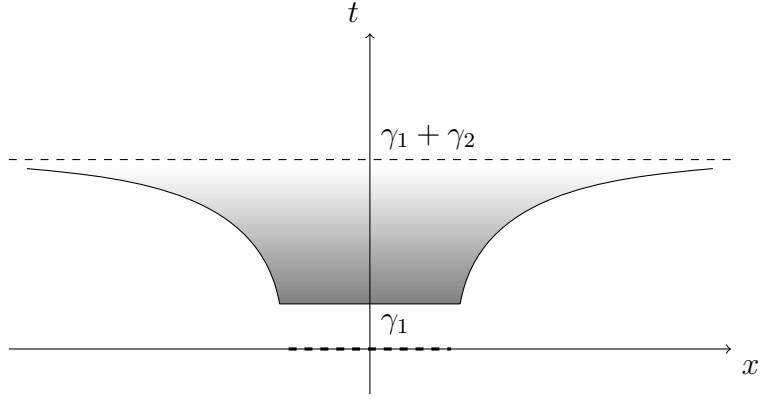
\begin{figure}
\centering
\begin{tikzpicture}[scale=1.2]

\shade[shading=axis, shading angle=180] (-1, 0.5) to [out=100, in=-5] (-3.8, 2) -- (3.8, 2) to [out=185, in=80] (1, 0.5) -- (-1, 0.5);
\draw[->] (0, -0.5) -- (0, 3.5) node[above left]{$t$};
\draw[->] (-4, 0) -- (4, 0) node[below right]{$x$};
\draw (-1, 0.5) to [out=100, in=-5] (-3.8, 2);
\draw (1, 0.5) to [out=80, in=185] (3.8, 2);
\draw (-1, 0.5) -- (1,0.5);
\draw[dashed] (-4, 2.1) -- (4, 2.1);
\draw (0, 0.5) node[below right]{$\gamma_{1}$};
\draw (0, 2.1) node[above right]{$\gamma_{1}+\gamma_{2}$};
\draw[very thick, dashed] (-0.9, 0) -- (0.9, 0);
\end{tikzpicture}
\caption{The expansion of positivity in the singular case. If at time $t=0$, $u\ge 1$ on the dotted part of given measure, after a waiting time $\gamma_{1}$, $u$ is pointwise bounded from below in the shaded region by a large power of $(\gamma_{1}+\gamma_{2}-t)$.}
\label{EPsingfig}
\end{figure}

 \begin{lemma}[Expansion of positivity, see figure \ref{EPsingfig}] \label{epossing}
There exists $\bar \lambda>p/(2-p)$ and, for any $\mu>0$,  $c(\mu), \gamma_{1}(\mu), \gamma_{2}(\mu)\in \ ]0, 1[$ such that if $u\ge 0$ is a supersolution in $Q_{8R, T}$ 
\begin{equation}
\label{exposingu}
 P_{0}(K_{r};\, u\ge k)\ge \mu\quad \Rightarrow\quad \inf_{K_{\rho}}u\big(\cdot,  k^{2-p}\, r^{p}\, \big(\gamma_{1}(\mu)+\gamma_{2}(\mu)\big(1-(r/\rho)^{\bar\lambda(2-p)-p}\big)\big)\ge c(\mu)\, \frac{k\, r^{\bar\lambda}}{\rho^{\bar\lambda}}
\end{equation}
whenever $r\le \rho\le R$ and $ k^{2-p}\, r^{p}\, (\gamma_{1}(\mu)+\gamma_{2}(\mu)(1-(r/\rho)^{\bar\lambda(2-p)-p}))\le T$. Moreover, the $\gamma_{i}(\mu)$ can be chosen arbitrarily small by lowering $c(\mu)$.
 \end{lemma}

 \begin{proof}
 The proof is very similar (and in fact simpler) to the one of Lemma \ref{expdeg} and we only sketch it. First expand in space \eqref{claims} through 
 \[
 P_{0}(K_{1};\, u\ge 1)\ge \mu\quad\Rightarrow\quad P_{0}(K_{8};\, u\ge 1)\ge \mu/8^{N}\quad \Rightarrow\quad u\ge c(\mu) \quad \text{in $K_{2}\times [\theta(\mu)/2, \theta(\mu)]$},
 \]
where we have set, with the notations in \eqref{claims}, $\theta(\mu):=T(\mu/8^{N})$, $c(\mu):=m(\mu/8^{N})$. Notice that, since $p<2$, we can suppose that $2^{p}\, c(1)^{2-p}\le 1$. Through a scaling argument, we infer that for any supersolution $u\ge 0$ in $K_{8\rho}\times [0, \theta(\mu)\, h^{2-p}\, \rho^{p}]$ it holds
 \begin{equation}
 \label{expsing}
 P_{0}(K_{\rho};\, u\ge h)\ge \mu\quad \Rightarrow \qquad u\ge c(\mu) h\qquad \text{in }\quad K_{2\rho}\times \big[\frac{\theta(\mu)}{2}\, h^{2-p}\, \rho^{p}, \theta(\mu)\, h^{2-p}\, \rho^{p}\big],
 \end{equation}
To prove \eqref{exposingu}, we can suppose that $k=1$ by scaling and define 
 \[
 c(\mu):=c(\mu, 1/2),\quad \bar\theta:=\theta(1),\quad \bar c:=c(1)\le 2^{\frac{p}{p-2}},\quad \rho_{n}=2^{n}\, r
 \]
 and $t_{n}$ as per \eqref{tn}. Since by assumption $P_{0}(K_{r}; u\ge 1)\ge \mu$, a first application of \eqref{expsing} implies $P_{t_{0}}(K_{r}; u\ge c(\mu))=1$.  Iterating \eqref{expsing} with $\mu=1$ we thus obtain 
\[
 u\ge c(\mu)\, \bar c^{n}\qquad \text{in }\quad K_{\rho_{n}}\times \big[t_{n}, t_{n}+\frac{\bar\theta}{2}\, (c(\mu)\, \bar c^{n-1})^{2-p}\, \rho_{n-1}^{p}\big]
\]
for all $n\ge 1$. From $2^{p}\, \bar c^{2-p}\le 1$ we infer  $t_{n}+2^{-1}\, \bar\theta\, (c(\mu)\, \bar c^{n-1})^{2-p}\, \rho_{n-1}^{p}\ge t_{n+1}$,
so that 
\[
u\ge c(\mu)\, \bar c^{n}\qquad \text{in }\quad K_{\rho_{n}}\times [t_{n}, t_{n+1}], \qquad n\ge 1.
\]
Finally, an explicit calculation shows that for suitable $\gamma_{1}(\mu), \gamma_{2}(\mu)>0$ it holds
\[
t_{n}=\gamma_{1}(\mu)\, r^{p}+\gamma_{2}(\mu)\big(1- (2^{p}\, \bar c^{2-p})^{n}\big)=r^{p}\, \big(\gamma_{1}(\mu)+\gamma_{2}(\mu)\big(1-(r/\rho_{n})^{\bar\lambda(2-p)-p}\big)\big)
\]
where $\bar\lambda=-\log_{2}\bar c>p/(2-p)$.
A monotonicity argument then gives the claim for any $\rho\ge r$.
 \end{proof}

\begin{theorem}[Appendix A of \cite{HR}]\label{sl1}
Let $u\ge 0$ solve \eqref{pt10} in $K_{2R}\times [t-2h, t]$ for some $p\in \ ]p_{*}, 2[$ . Then 
\begin{equation}
\label{l1har}
\sup_{K_{R}\times [t-h, t]} u\le \frac{\bar c}{h^{\frac{N}{N(p-2)+p}}}\left(\inf_{s\in [t-2h, t]}\int_{K_{2R}}u(x, s)\, dx\right)^{\frac{p}{N(p-2)+p}}+\bar c\big(\frac{h}{R^{p}}\big)^{\frac{1}{2-p}}.
\end{equation}
\end{theorem}

\begin{theorem}[Harnack inequality]
Let $p\in \ ]p_{*}, 2[$. There exists constants $\bar C\ge 1$, $\bar\theta>0$ such that any solution $u\ge 0$ of \eqref{pt10} in $K_{8R}\times [- T, T]$ obeying  $u(0, 0)>0$ and
\begin{equation}
\label{tass}
 4\, R^{p}\, \sup_{K_{2R}}u(\cdot, 0)^{2-p}\le T
\end{equation}
satisfies the following Harnack inequality 
\begin{equation}
\label{singharnack}
\bar C^{-1}\, \sup_{K_{R}} u(\cdot, s)\le u(0, 0)\le \bar C\, \inf_{K_{R}}u(\cdot, t),\qquad - \bar\theta\, u(0, 0)^{2-p}\, R^{p}\le s, t\le\bar\theta\, u(0, 0)^{2-p}\, R^{p}.
\end{equation}
\end{theorem}

\begin{proof}
Consider the solution $u(0, 0)^{-1}\, u(R\, x, R^{p}\, u(0, 0)^{2-p}\, t)$ in $K_{8}\times [-T', T']$ (still denoted by $u$) with $T'= T\, R^{-p}\, u(0, 0)^{p-2}$. This reduces us to $u(0, 0)=1$, $R=1$, $T'\ge 4$ and \eqref{tass} implies 
\begin{equation}
\label{tass2}
 1\le M^{2-p}:=\sup_{K_{1}}u(\cdot, 0)^{2-p}\le T'/4.
\end{equation}
We first prove the $\inf$  bound in \eqref{singharnack}. Let $\bar\lambda\ge p/(2-p)$ be the expansion of positivity exponent, define $\psi(\rho)=(1-\rho)^{\bar\lambda}\, \sup_{K_{\rho}}u(\cdot, 0)$ for $\rho\in [0, 1]$ and choose $\rho_{0}$, $x_{0}\in K_{\rho_{0}}$ such that 
\[
\max_{[0, 1]}\psi=\psi(\rho_{0})=(1-\rho_{0})^{\bar\lambda}\, u_{0},\qquad u_{0}=u(x_{0}, 0)\ge 1.
\]
As in the proof of Theorem \ref{harnackellittica}, we can let $\bar\xi\in [0, 1]$ obey $(1-\bar\xi)^{-\bar\lambda}=2$ to find for $r=\bar\xi\, (1- \rho_{0})$
\begin{equation}
\label{padf}
u_{0}\, r^{\bar\lambda}\ge \bar\xi^{\bar\lambda},\qquad \sup_{K_{r}(x_{0})}u(\cdot, 0) \le (1-\bar\xi)^{-\bar\lambda}\,  u_{0}=2\, u_{0}.
\end{equation}
Let $a:=u_{0}^{2-p}\, r^{p}$. By construction $u_{0}\le M$ and by  \eqref{tass2}, $u$ solves \eqref{pt10} in $K_{r}(x_{0})\times [-4\, a,  4\, a]$. We can apply \eqref{l1har} for $R=r/2$, $t= a$, $s=0$ and $h=2\, a$ to get
\begin{equation}
\label{final}
\begin{split}
\sup_{K_{\frac{r}{2}}(x_{0})\times [-a, a]}u&\le \frac{\bar c}{a^{\frac{N}{N(p-2)+p}}}\left(\int_{K_{r}(x_{0})}u(x, 0)\, dx\right)^{\frac{p}{N(p-2)+p}}+\bar c\, a^{\frac{1}{2-p}}\, r^{\frac{p}{p-2}}\\
&\le \bar c\,\frac{ (2\, u_{0}\, r^{N})^{\frac{p}{N(p-2)+p}}}{(u_{0}^{2-p}\, r^{p})^{\frac{N}{N(p-2)+p}}}+\bar c\, u_{0}\le \bar c\, u_{0},
\end{split}
\end{equation}
where we used the second inequality in \eqref{padf} to bound the integral. Since $a=u_{0}^{2-p}\, r^{p}$, we can apply \eqref{oscp<2} with $k=u_{0}$ in both $K_{r/2}(x_{0})\times [-a, 0]$ and $K_{r/2}(x_{0})\times [0, a]$ to get
\[
{\rm osc}(u, K_{\rho}(x_{0})\times [- a, a])\le \bar c\, u_{0}\, (\rho/r)^{\bar\alpha},\qquad \rho\le r/2.
\]
As $u(x_{0})=u_{0}$ we infer that $u\ge u_{0}/2$ in $K_{\bar\eta r}(x_{0})\times [-\bar\eta^{p}\, a, \bar\eta^{p}\, a]$ for suitable $\bar\eta\in \ ]0, 1/2[$, so that $P_{t}(K_{r}(x_{0});\, u\ge u_{0}/2)\ge \bar\eta^{N}$ for all $|t|\le \bar\eta\, u_{0}^{2-p}\, r^{p}$. Apply the expansion of positivity Lemma \ref{epossing} at an arbitrary time $t$ such that $|t|\le \bar\eta\, u_{0}^{2-p}\, r^{p}$, choosing the $\gamma_{i}(\bar\eta^{N})$ so small that $\gamma_{1}(\bar\eta^{N})+\gamma_{2}(\bar\eta^{N})<\bar\eta/2$. Its conclusion for $k=u_{0}/2$, $\rho=2$ implies, thanks to $K_{2}(x_{0})\supseteq K_{1}$,
\[
\inf_{K_{1}} u(\cdot, t+ \gamma_{r} \, u_{0}^{2-p}\, r^{p})\ge \bar c\, u_{0}\, r^{\bar\lambda},\qquad \gamma_{r}:=\gamma_{1}(\bar\eta^{N})+\gamma_{2}(\bar\eta^{N})\big(1-(r/2)^{\bar\lambda(2-p)-p}\big)<\bar\eta/2 
\]
for all $|t|\le \bar\eta\, u_{0}^{2-p}\, r^{p}$. The latter readily gives $u(x, t)\ge  \bar c\, u_{0}\, r^{\bar\lambda}$ for $x\in K_{1}$ and $|t|\le\bar\eta\,  u_{0}^{2-p} \, r^{p}/2$. Finally, observe that since $r\le 1$ and $\bar\lambda\ge p/(2-p)$, it holds $u_{0}^{2-p} \, r^{p}\ge (u_{0}\, r^{\bar\lambda})^{2-p}$, so that the first inequality in \eqref{padf} yields $u(x, t)\ge  \bar c\, \bar\xi^{\bar\lambda}=:1/\bar C$ for $x\in K_{1}$ and $|t|\le\bar\eta\,  \bar\xi^{\bar\lambda(2-p)}/2=:\bar\theta$.

To prove the $\sup$ bound we proceed similarly. Indeed, let $x_{*}\in K_{R}$ be such that $u(x_{*}, 0)=\sup_{K_{R}}u$. Notice that $K_{R}(x_{*})\subseteq K_{2R}$, hence \eqref{tass} still implies \eqref{tass2} for the rescaled (and translated)  function. Hence, the same proof as before carry over, giving after rescaling back $\inf_{K_{R}} u(\cdot, 0)\ge c\, u(x_{*}, 0)$. This implies $\sup_{K_{1}}u(\cdot, 0)\le C\, u(0, 0)$ and we can proceed as in \eqref{final} for $r=2R$, $x_{0}=0$ and  $a= R^{p}\sup_{K_{2R}}u$ to get the final $\sup$ estimate.
\end{proof}

\vskip4pt
\noindent
{\small {\bf Acknowledgement.}  S. Mosconi and V. Vespri  are members of GNAMPA (Gruppo Nazionale per l'Analisi Matematica, la Probabilit\`a e le loro Applicazioni) of INdAM (Istituto Nazionale di Alta Matematica). F. G. D\"uzg\"un is partially funded by Hacettepe University BAP through project FBI-2017-16260; S. Mosconi is partially funded by the grant PdR 2016-2018 - linea di intervento 2: ``Metodi Variazionali ed Equazioni Differenziali'' of the University of Catania.}

\end{document}